\documentclass[preprint,10pt,3p]{elsarticle}
\biboptions{sort&compress}

\usepackage{graphicx}

\usepackage{mathtools,amscd,amssymb,amsfonts,amsthm, amsmath}

\usepackage[utf8]{inputenc}
\usepackage{geometry}
\usepackage{longtable}
\usepackage{float}
\usepackage{esvect}
\usepackage{hyperref}
\usepackage{xcolor}
\usepackage{bm}
\usepackage{booktabs, multirow, caption}
\usepackage{subfig}

\usepackage[active]{srcltx}
\usepackage{cleveref}
\usepackage{kotex}
\usepackage{mathrsfs}
\usepackage{cancel}

\usepackage{autonum}
\usepackage{algorithm}

\usepackage{algpseudocode}

\usepackage{pgfplots}
\usepackage{tikz}
\usetikzlibrary{arrows.meta,positioning,fit,calc,shapes.misc}

\newtheorem{theorem}{Theorem}[section]
\newtheorem{lemma}{Lemma}[section]

\newtheorem{corollary}{Corollary}[section]

\theoremstyle{remark}
\newtheorem{remark}{Remark}[section]

\DeclareMathAlphabet\mathbfcal{OMS}{cmsy}{b}{n}

\journal{Journal of Applied Mathematics and Computing}

\begin{document}

\begin{frontmatter}


\title{Fast and Scalable Caputo Fractional Gradient Descent via Perturbation-Preserving Memory Compression}

\author[label1]{Hwanseo Lee}
\author[label1]{Junseo Lee}
\author[label1]{Hyunju Kim\corref{cor1}}
\address[label1]{Department of Energy Engineering, Korea Institute of Energy Technology, Naju 58330, Republic of Korea}
\ead{hjkim@kentech.ac.kr}


\cortext[cor1]{Corresponding author. Tel:+82-61-320-9240 }

\date{\today}

\begin{abstract}
Fractional gradient descent (FGD) incorporates long-range memory through Caputo-type operators and has been shown to improve stability in ill-conditioned and nonconvex optimization problems. Despite these advantages, its practical use remains limited, mainly due to the high computational cost of evaluating history-dependent convolutions, which scales quadratically with the number of iterations.

In this paper, we focus on making Caputo-based optimization computationally viable without sacrificing its intrinsic memory structure. We begin by expressing the fractional descent direction as a discrete convolution over past gradients, which provides a unified view of the method. Based on this formulation, we introduce two complementary mechanisms to reduce the cost of the memory term. The first uses a sum-of-exponentials (SOE) approximation of the power-law kernel, leading to efficient recursive updates. The second approach, newly proposed in this paper as dyadic hierarchical discrete convolution (DHDC), compresses the gradient history through a multiscale aggregation strategy.

Rather than treating these approximations as purely numerical accelerations, we interpret them as perturbations of the ideal Caputo operator. This viewpoint allows us to analyze how the compressed memory affects the optimization dynamics. Under standard $\mu$-strong convexity and $L$-smoothness assumptions, we show that the resulting method still exhibits monotone descent and linear convergence, provided that the approximation error remains controlled.

We validate the proposed approach on several representative tasks, including battery state estimation, stochastic signal reconstruction, and image classification on CIFAR-10. Across these settings, the method consistently demonstrates improved stability compared to classical gradient descent and better robustness than truncated fractional schemes, while significantly reducing computational and memory costs.

To further investigate the optimization behavior in highly multimodal landscapes, we additionally study the Rastrigin benchmark problem. The experiments demonstrate that the proposed DHDC method maintains persistent but bounded oscillatory dynamics induced by compressed fractional memory, enabling the optimizer to repeatedly escape shallow local minima while remaining stable near the global minimizer. We further analyze the oscillation magnitude with respect to the fractional order and empirically verify the computational complexity of the proposed methods.

These results suggest that fractional-order optimization can be made scalable in practice, without losing the essential benefits of long-memory dynamics.
\end{abstract}


\end{frontmatter}

\section{Introduction}
\label{sec:intro}

Fractional calculus extends classical differentiation and integration to non-integer orders, providing a natural framework for modeling systems with memory and nonlocal interactions. 
Since its formal development through the works of Liouville, Riemann, and Caputo \cite{Podlubny1999,Kilbas2006,Oldham1974,Caputo1967}, it has become an essential tool in the analysis of processes where past states influence present behavior. 
In modern scientific computing, fractional operators are particularly valued for their ability to capture long-range temporal and spatial dependencies, which are difficult to represent using classical integer-order models \cite{Metzler2000,Mainardi2010,Diethelm2010,Meerschaert2012}. 
Among the various formulations, the Caputo derivative is especially well suited for physical modeling, as it accommodates classical initial conditions and yields interpretable constitutive laws in viscoelasticity, anomalous diffusion, and other hereditary systems \cite{Caputo1967,Bagley1983,Podlubny1999,Mainardi2010,Diethelm2010}.

The nonlocal memory inherent in fractional calculus has naturally motivated its extension to optimization problems \cite{LiZeng2015,Wang2017FractionalBP,Pu2015Fractional}. 
In classical gradient-based methods, parameter updates depend solely on the instantaneous gradient, leading to inherently memoryless dynamics \cite{Polyak1987,Nesterov2004}. 
In contrast, fractional-order updates incorporate a history-dependent mechanism, where past gradients contribute through a kernel-weighted accumulation \cite{Wang2017FractionalBP,LiZeng2015}. 
This formulation is consistent with the memory structures observed in fractional differential equations \cite{Podlubny1999,Diethelm2010}, and enables optimization dynamics that explicitly account for long-range temporal interactions.

The distinction between momentum-based methods and fractional gradient descent (FGD) becomes particularly clear when viewed through their underlying memory kernels. 
In classical heavy-ball momentum, an auxiliary variable evolves according to
\[
\bm{v}_{n+1} = \beta \bm{v}_n + \bm{g}_n,
\qquad 0 \le \beta < 1,
\]
which can be interpreted as a linear discrete-time dynamical system \cite{Polyak1964,Nesterov2004}. 
Assuming $\bm{v}_0=\bm{0}$, unrolling the recursion yields
\[
\bm{v}_n = \sum_{j=0}^{n-1} \beta^{\,n-1-j}\,\bm{g}_j,
\]
revealing that momentum is equivalent to a convolution with an exponentially decaying kernel \cite{Goh2017,Lessard2016}. 
This exponential weighting implies that past gradients are discounted geometrically, resulting in a short-range memory mechanism. 
In particular, when $\beta$ approaches $1$, the effective memory length scales as $(1-\beta)^{-1}$, indicating that the influence of past information is governed by a single characteristic time scale.

By contrast, Caputo-based FGD induces a fundamentally different memory structure in the descent dynamics, characterized by a power-law kernel. 
The discrete fractional gradient at iteration $n$ is computed as a weighted sum of historical gradients:
\begin{equation}\label{eq:discrete_update}
    \widehat{\bm{g}}_n = \sum_{j=0}^{n-1} w_{n-j}^{(\alpha)} \bm{g}_j,
\end{equation}
where the weights $w_k^{(\alpha)}$ are positive and non-increasing. These coefficients are defined as follows, with their detailed derivation and properties provided in \ref{app:LSA}:
\begin{equation}\label{eq:weights_def}
    w_{k}^{(\alpha)} = \frac{\Delta t^{1-\alpha}}{\Gamma(2-\alpha)} \left[ k^{1-\alpha} - (k-1)^{1-\alpha} \right], \quad k=1, 2, \dots, n,
\end{equation}
where $0 < \alpha < 1$ denotes the fractional order.
Unlike exponential decay, this power-law weighting introduces a scale-free memory mechanism with no intrinsic characteristic time scale. 
As a consequence, the influence of past gradients persists over long horizons in a non-negligible manner, rather than being rapidly attenuated. Importantly, this behavior cannot be replicated by any finite combination of exponential kernels, since such mixtures inevitably retain an exponential tail and therefore impose a finite memory scale. 
This structural difference implies that fractional gradient descent is not merely a momentum variant with a modified parameter, but represents a fundamentally distinct class of nonlocal optimization dynamics.

The development of fractional optimization can be broadly organized into two closely related strands. 
The first stems from the direct discretization of fractional gradient flows, where Caputo or Riemann–Liouville derivatives are employed to define nonlocal descent dynamics \cite{LiZeng2015,Wang2017FractionalBP}. 
These approaches have demonstrated reduced oscillatory behavior and improved stability in nonlinear optimization settings, highlighting the practical benefits of incorporating long-memory effects. 
The second strand adopts an algorithmic perspective, interpreting fractional updates as history-weighted gradient schemes and analyzing their convergence under convexity assumptions \cite{Pu2015Fractional,HerreraAlcantara2022}. 
From this viewpoint, fractional descent can be understood as inducing a power-law memory kernel, providing a conceptual link between optimization algorithms and fractional dynamical systems. 

While these developments have clarified the role of memory in optimization, most existing works either focus on the formulation of fractional updates or their theoretical properties, with limited attention to how such long-range memory can be implemented efficiently in large-scale settings.
More recently, fractional optimization has been explored in neural network training and signal processing applications, where empirical studies indicate improved robustness to noise and ill-conditioning \cite{Wang2017FractionalBP,Pu2015Fractional,HerreraAlcantara2022,Wang2022Neurocomputing}. 
In parallel, numerical analysis has developed efficient techniques for handling long-memory kernels, including sum-of-exponentials approximations and hierarchical history compression \cite{Jiang2017SOE,Baffet2017}. 

Despite the growing interest in fractional-order optimization, existing approaches face a fundamental challenge: the incompatibility between long-memory dynamics and scalable computation. 
FGD, particularly in its Caputo formulation, introduces a nonlocal temporal structure in which the descent direction depends on the entire history of past gradients. 
This power-law memory is a key feature that distinguishes fractional methods from classical gradient-based algorithms, enabling improved stability and robustness in complex optimization landscapes.

However, a direct implementation of Caputo-based FGD is computationally prohibitive. 
The discrete realization of the fractional convolution requires access to all past gradients, resulting in $\mathcal{O}(n)$ cost per iteration and $\mathcal{O}(N^2)$ total complexity over $N$ iterations. 
This quadratic scaling severely limits the applicability of fractional optimization methods in large-scale learning problems.
A natural approach to address this issue is to adopt fast-memory techniques developed in the numerical analysis of fractional differential equations, such as sum-of-exponentials approximation (SOE)\cite{jiang2017fast} and hierarchical convolution schemes. 
These methods reduce computational complexity by compressing the history dependence of the fractional kernel. But, they are primarily designed for forward time-stepping problems, where the objective is accurate operator evaluation under a fixed trajectory.

In contrast, optimization algorithms involve an evolving sequence of iterates, where gradient information must be updated adaptively. 
As a result, a direct transfer of fast-memory techniques is not straightforward. Naive memory compression may distort the effective descent direction, potentially altering the optimization dynamics. This observation raises a fundamental question: how can one achieve scalable computation while preserving the intrinsic power-law memory structure of fractional optimization?
In this paper, we address this challenge by embedding fast-memory mechanisms directly into the optimization dynamics. 
Rather than treating memory compression as an external approximation, we reinterpret fractional gradient descent as a history-dependent convolution operator and construct a compressed representation that preserves its structural properties. 

A key insight of this work is that the proposed fast-memory schemes—based on SOE approximation and dyadic hierarchical discrete convolution (DHDC)—can be viewed as inducing a controlled perturbation of the ideal Caputo convolution. 
This perspective allows us to analyze their effect on optimization dynamics and establish that the convergence behavior of fractional gradient descent is preserved under suitable conditions.
To make the formulation concrete, we consider Caputo-based fractional gradient descent. 
For $0<\alpha<1$, the Caputo fractional derivative is defined by
\begin{equation}\label{eq:caputo-def}
    \prescript{C}{a}{D}_{t}^{\alpha} f(t)
    =
    \frac{1}{\Gamma(1-\alpha)}
    \int_{a}^{t} \frac{f'(\tau)}{(t-\tau)^{\alpha}}\, d\tau,
\end{equation}
which introduces a slowly decaying kernel and induces long-range memory effects. 
Such operators naturally arise in the modeling of systems with history dependence and have been extensively studied in fractional calculus \cite{Diethelm2010}.
Building on this formulation, we define a history-dependent descent direction as a convolution over past gradients, leading to a discrete fractional gradient descent update. 
This formulation captures the essential nonlocal structure of the Caputo operator while providing a natural bridge to numerical discretization.
The main contributions of this paper are summarized as follows:
\begin{enumerate}
\item We reinterpret Caputo-based fractional gradient descent as a nonlocal convolution operator acting on the gradient history, clarifying the role of power-law memory in optimization.

\item We develop a scalable computational framework by integrating SOE-based approximation and DHDC-based memory compression, reducing the overall complexity from $\mathcal{O}(N^2)$ to nearly $\mathcal{O}(N)$.

\item We show that the proposed fast-memory schemes can be interpreted as controlled perturbations of the ideal fractional descent direction, and establish convergence preservation under smooth strong convexity.

\item We provide a unified framework that connects fractional numerical analysis with large-scale optimization, enabling efficient realization of long-memory effects without truncating the intrinsic power-law structure.
\end{enumerate}

The proposed approach bridges the gap between high-fidelity fractional modeling and scalable optimization, providing a principled pathway for incorporating long-range memory into modern learning algorithms.

The effectiveness of the method is demonstrated through one highly multimodal benchmark problem and three representative problem classes: Rastringin optimization, battery-related regression tasks exhibiting diffusion-driven memory effects, stochastic signal estimation with long-range dependence, and image classification on CIFAR-10, where nonconvexity and ill-conditioning are prominent. 
These experiments are intentionally chosen to cover optimization landscapes with fundamentally different memory and stability characteristics. In particular, the Rastrigin benchmark highlights the ability of the proposed DHDC dynamics to escape shallow local minima through persistent but bounded oscillatory behavior induced by compressed fractional memory. The remaining applications further demonstrate that power-law memory can improve optimization stability, robustness, and learning behavior compared with classical gradient descent and momentum-based methods while substantially reducing the computational burden associated with full-history fractional optimization.

The remainder of the paper is organized as follows. 
Section~\ref{sec:SOE_and_DHDC} introduces the proposed fast-memory framework and presents the SOE- and DHDC-based realizations of fractional gradient descent. 
Section~\ref{subsec:DHDC} describes the construction of the DHDC algorithm and its hierarchical discrete convolution mechanism for compressed long-memory accumulation. 
Section~\ref{sec:convergence} establishes convergence preservation results under smooth strong convexity and analyzes the effect of fast-memory perturbations on the optimization dynamics. 
Section~\ref{sec:numerical_experiments} describes the experimental setup and benchmark configurations, while Section~\ref{sec:numerical_result} presents the numerical results, oscillation analysis, and empirical computational complexity studies that validate the effectiveness and scalability of the proposed framework.

\section{From Caputo gradient flow to a discrete optimizer}\label{sec:SOE_and_DHDC}

To construct a practical algorithm, we work on a uniform time grid $t_n = n\Delta t$ with step size $\Delta t>0$, and denote $\bm{\theta}_n \approx \bm{\theta}(t_n)$ and
$\bm{g}_n = \nabla L(\bm{\theta}_n)$.
Under a Caputo-based formulation, the descent direction at step $n$ is no longer given by a single gradient, but by a history-weighted combination of all past gradients.
As discussed in Section \ref{sec:intro}, this leads naturally to a discrete convolution (DC) structure, which serves as the starting point for all FGD methods considered
in this paper.

The central challenge is therefore not how to define fractional descent, but how to evaluate this history-dependent convolution efficiently and robustly inside an optimizer.
In the following, we present two complementary strategies to reduce this computational burden.
The first is introduced in Section \ref{subsec:SOE} and accelerates the convolution by approximating the Caputo kernel via a SOE representation.
The second, introduced in Section \ref{subsec:DHDC}, compresses the gradient history directly using a DHDC.
Both approaches preserve the essential power-law memory structure, but differ in how memory is represented and updated during training.

\subsection{Fast Caputo memory in fractional gradient descent via Sum-Of-Exponentials}\label{subsec:SOE}

A defining characteristic of Caputo-based FGD is the presence of long-range, power-law memory in the update direction. After discretization, the fractional gradient at iteration $n$ takes the convolutional form:
\begin{equation}\label{eq:caputo-kernel-conv}
    \widehat{\bm{g}}_n = \sum_{j=0}^{n-1} w_{n-j}^{(\alpha)} \bm{g}_j,
\end{equation}
where $\bm{g}_j = \nabla L(\bm{\theta}_j)$, and the weights $w_{k}^{(\alpha)}$ are defined as in \eqref{eq:weights_def}. This formulation faithfully realizes fractional dynamics, as all past gradients contribute to the descent direction with algebraically decaying influence. 

However, a direct implementation of \eqref{eq:caputo-kernel-conv} poses a significant computational challenge. As the iteration count $n$ increases, both the memory required to store the gradient history and the computational work per step grow linearly, resulting in an $\mathcal{O}(n^2)$ total cost after $n$ steps. Such growth is prohibitive for long training horizons and large-scale learning problems where $n$ can be extremely large. To overcome this "growing memory" bottleneck, an efficient numerical approach is required to approximate the non-local convolution with a fixed per-step cost. This motivates the use of the Sum-of-Exponentials (SOE) approximation, which transforms the global summation into a set of local recursive updates.

Our first acceleration strategy exploits a structural property of the Caputo kernel itself. The kernel
\[
    k(t)=\frac{1}{\Gamma(1-\alpha)}\,t^{-\alpha},
    \qquad 0<\alpha<1,
\]
is completely monotone, and therefore admits accurate approximations by positive weighted sums of exponentials.
More precisely, for any prescribed tolerance $\varepsilon>0$ on a finite time window $t\in[t_{\min},t_{\max}]$, one can construct coefficients
$\{\xi_m,\omega_m\}_{m=1}^M$ with $0<\xi_1<\cdots<\xi_M$ and $\omega_m>0$ such that 
\begin{equation}\label{eq:soe-approx}
    k(t)
    \;\approx\;
    \sum_{m=1}^{M}\omega_m e^{-\xi_m t},
    \qquad
    \frac{\big\|k-\sum_{m=1}^{M}\omega_m e^{-\xi_m(\cdot)}\big\|_\infty}{\|k\|_\infty}
    <\varepsilon.
\end{equation}
Substituting the SOE approximation into the convolution yields a practical fractional descent direction:
\begin{equation}\label{eq:soe-fractional-gradient}
    \widehat{\bm{g}}^{SOE}_n \approx \sum_{m=1}^{M}\omega_m\,\bm{R}_m^{(n)},
\end{equation}
where each auxiliary state $\bm{R}_m^{(n)}\in\mathbb{R}^d$ evolves according to the recursive update rule:
\begin{equation}\label{eq:soe-recursion}
    \bm{R}_m^{(n)} = e^{-\xi_m\Delta t}\,\bm{R}_m^{(n-1)} + \bm{g}_{n-1}, \quad \bm{R}_m^{(0)}=\bm{0}.
\end{equation}
The detailed derivation from the continuous-time convolution and the corresponding pseudocode are provided in \ref{Appendix:B} and Algorithm \ref{alg:soe-update} in \ref{Appendix:Algorithm}, respectively.

In this form, the entire Caputo memory is encoded in $M$ exponentially weighted
accumulators.
The resulting SOE-based fractional gradient descent (SOE-FGD) requires only
$\mathcal{O}(M)$ memory and $\mathcal{O}(M)$ operations per iteration, with $M\ll n$
in typical applications.

In this work, SOE is not treated merely as a fast quadrature for fractional operators, but as an explicit optimizer engine.
Each iteration updates a fixed set of memory modes and constructs a history-aware descent direction that can be combined with standard learning-rate schedules and
preconditioners. From an optimization viewpoint, the SOE modes admit a natural interpretation: modes with large $\xi_m$ capture short-term dynamics, while modes with small $\xi_m$
encode long-term memory. In this way, SOE-FGD replaces the single-gradient update of classical gradient descent by a multi-timescale descent direction with controlled memory depth.

The SOE parameters are computed once during initialization.
Given an expected training horizon, we select a window $[t_{\min},t_{\max}]$ and place
the decay rates on a logarithmic grid,
\begin{equation}\label{eqn:SOE_log}
    \xi_m
    =
    \exp\!\left[
    \log\!\left(1/t_{\min}\right)
    +
    \frac{m-1}{M-1}
    \Big(
    \log\!\left(1/t_{\max}\right)
    -
    \log\!\left(1/t_{\min}\right)
    \Big)
    \right],
\end{equation}
and determine the weights $\{\omega_m\}$ by solving a nonnegative least-squares fit of the target kernel on logarithmically spaced sample points.
The resulting coefficients are stored as fixed tensors and reused throughout training.

At the same time, SOE compression imposes a structural constraint: the exponential basis is fixed over a prescribed time window, and the associated coefficients depend on the chosen fractional order $\alpha$.
If $\alpha$ is adapted during training, the SOE coefficients must be refitted or switched from a precomputed library.
This observation motivates the alternative strategy introduced next.
Rather than approximating the kernel, the DHDC method compresses the gradient history itself, yielding an order-agnostic memory representation that naturally supports
dynamic fractional-order scheduling.
$\text{where} ~\tilde{\bm{g}}_n = \frac{1}{\Gamma(1-\alpha)}\int_{t_0}^t (t-\eta)^{-\alpha} u'(\eta)d\eta,$

\subsection{Dyadic Hierarchical Discrete Convolution compression}\label{subsec:DHDC}

We now consider an alternative strategy that does not approximate the kernel itself, but instead reduces the cost of evaluating the history-dependent convolution.

The approach adopted here is inspired by the shifted binary block partition (SBBP) method for the fast evaluation of variable-order Caputo derivatives \cite{Fang2020SBBP}.
In that work, the Caputo integral is decomposed over a hierarchy of dyadic subintervals, and each subintegral is approximated using linear interpolation and stored moment
information, enabling efficient updates in time-fractional differential equations.
In contrast, the FGD update is already expressed as a history-weighted convolution of past gradients, as shown in \eqref{eq:discrete_update}.
Rather than approximating subintegrals of a continuous operator, we compress the discrete gradient history itself.
This leads to a simplified dyadic structure referred to as dyadic hierarchical discrete convolution (DHDC).
We begin with the full discrete Caputo-type convolution at iteration \(n\),
\begin{equation}\label{eq:DHDC-start}
    \widetilde{\bm{g}}_n
    =
    \sum_{j=0}^{n-1}
    k_{\alpha}((n-j)\Delta t) \bm{g}_j,
    \qquad
    k_{\alpha}(t)=\frac{t^{-\alpha}}{\Gamma(1-\alpha)}\Delta t,\quad \alpha\in(0,1),
\end{equation}
where \(\bm{g}_j=\nabla L(\bm{\theta}_j)\). A detailed derivation of this discrete Caputo-type convolution from the continuous Caputo fractional derivative, together with the corresponding quadrature weights for the constant approximation scheme, is provided in \ref{subsec:constant_approximation_scheme}.

As discussed in \cref{subsec:SOE}, a direct evaluation of the formula \eqref{eq:DHDC-start} requires $\mathcal{O}(n)$ work per iteration and $\mathcal{O}(n^2)$ total complexity over $n$ iterations. Such growth is prohibitive for large-scale optimization tasks. To reduce this complexity while maintaining the order-agnostic property, we group the gradient history into bins whose effective widths grow geometrically with age.
At time \(t_n=n\Delta t\), define bins indexed by \(b=0,1,\dots,B_n-1\) with nominal width
\[
    \ell_b = 2^b \Delta t, 
\]
and associated look-back windows
\begin{equation}\label{eq:DHDC-window}
    I_b(n)
    =
    \Big[t_n-2^{b+1}\Delta t,\;
         t_n-2^{b}\Delta t\Big].
\end{equation}
Recent gradients are retained at high resolution, whereas older gradients are represented more coarsely, reflecting the slow variation of the Caputo kernel for large time lags.
Figure \ref{fig:DHDC-windows-example} provides a visualization of these dyadic look-back windows for selected iteration instances.

\begin{figure}[htbp]
    \centering
    \includegraphics[width=0.8\textwidth]{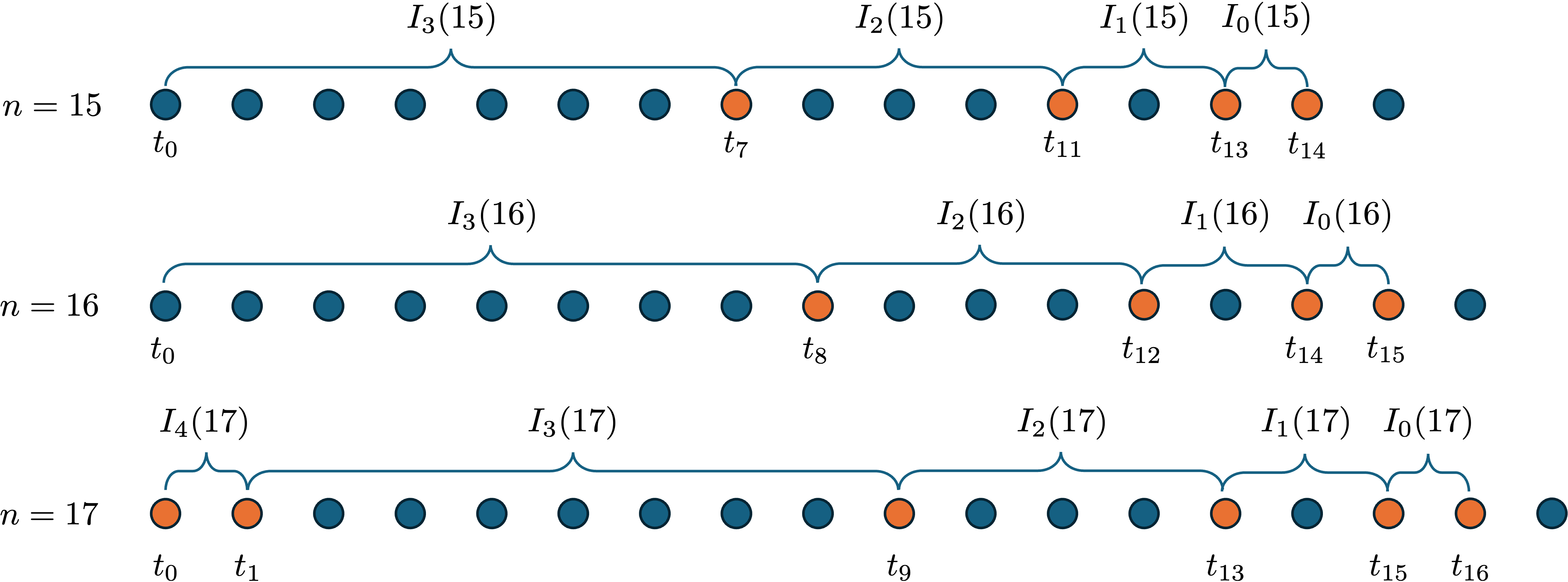}
    \caption{Dyadic look-back window boundaries \(I_b(n)\) for \(n =15,16\), and \(17\) with \(B=5\).}
    \label{fig:DHDC-windows-example}
\end{figure}

Instead of storing all past gradients, DHDC maintains aggregated quantities
\[
    \bm{S}^n_b \in \mathbb{R}^d,
    \qquad
    C^n_b \in \mathbb{N},
    \qquad b=0,1,\dots,B_n-1.
\]
Here, \(\bm{S}^n_b\) represents the aggregated contribution of past gradients stored in bin \(b\), and \(C^n_b\) denotes the corresponding effective count.
Because the carry operation redistributes a fraction of the accumulated contribution between neighboring bins, \(\bm{S}^n_b\) should not be interpreted as the sum over a fixed index set.
Instead, it corresponds to a weighted aggregation of past gradients, consistent with the induced weights \(\omega_{j,b}^{(n)}\), where each gradient may contribute fractionally to multiple bins.
More precisely, the DHDC state induces nonnegative aggregation weights
\[
    \omega^{(n)}_{j,b}\ge 0,
    \qquad j=0,\dots,n-1,\quad b=0,\dots,B_n-1,
\]
such that
\begin{equation}\label{eq:DHDC-weighted-state}
    \bm{S}^n_b
    =
    \sum_{j=0}^{n-1}
    \omega^{(n)}_{j,b}\bm{g}_j,
    \qquad
    C^n_b
    =
    \sum_{j=0}^{n-1}
    \omega^{(n)}_{j,b}.
\end{equation}
The weights satisfy the partition of unity propeprty
\begin{equation}\label{eq:DHDC-mass-conservation}
    \sum_{b=0}^{B_n-1}
    \omega^{(n)}_{j,b}
    =
    1,
    \qquad j=0,\dots,n-1,
\end{equation}
so that every past gradient contributes exactly once in total, although its contribution may be distributed across dyadic bins through the carry operation.

At each iteration, the newest gradient is first inserted into the finest bin:
\begin{equation}\label{eq:DHDC-insertion}
    \bm{S}^n_0 \leftarrow \bm{S}^{n-1}_0 + \bm{g}_n,
    \qquad
    C^n_0 \leftarrow C^{n-1}_0 + 1.
\end{equation}
Equivalently, the newly inserted gradient initially satisfies
\[
    \omega^{(n)}_{n,0}=1,
    \qquad
    \omega^{(n)}_{n,b}=0
    \quad \text{for } b\ge 1.
\]
The bins are then updated by a dyadic carry operation with capacity
\[
    \mathrm{cap}(b)=2^b.
\]
Whenever \(C^n_b>\mathrm{cap}(b)\), a fraction of the aggregated contribution in bin \(b\) is transferred to bin \(b+1\).
We define
\begin{equation}\label{eq:DHDC-half-M}
    h^n_b := \max\big(1,\lfloor C^n_b/2 \rfloor\big),
    \qquad
    \phi^n_b := \frac{h^n_b}{C^n_b},
    \qquad
    \bm{M}^n_b := \phi^n_b \bm{S}^n_b .
\end{equation}
The scalar \(h^n_b\) determines the amount of effective count transferred, while \(\bm{M}^n_b\) represents the corresponding transferred contribution in the aggregated gradient.

The carry update is then
\begin{align}
    \bm{S}^n_b
    &\leftarrow
    \bm{S}^n_b-\bm{M}^n_b,
    &
    C^n_b
    &\leftarrow
    C^n_b-h^n_b,
    \label{eq:minus}
\end{align}
\begin{align}
    \bm{S}^n_{b+1}
    &\leftarrow
    \bm{S}^n_{b+1}+\bm{M}^n_b,
    &
    C^n_{b+1}
    &\leftarrow
    C^n_{b+1}+h^n_b.
    \label{eq:plus}
\end{align}

The same operation admits a precise interpretation in terms of the aggregation weights.
For all gradients represented in bin \(b\), a fraction \(\phi_b^n\) of their contribution is transferred to bin \(b+1\), while the remaining fraction \(1-\phi_b^n\) remains in bin \(b\).
Thus, the carry operation can be written as
\[
    \omega^{(n)}_{j,b}
    \leftarrow
    (1-\phi_b^n)\omega^{(n)}_{j,b},
    \qquad
    \omega^{(n)}_{j,b+1}
    \leftarrow
    \omega^{(n)}_{j,b+1}
    +
    \phi_b^n\omega^{(n)}_{j,b}.
\]
This formulation shows that DHDC performs a redistribution of aggregation weights across dyadic bins, rather than forming a strict partition of the past indices.

\begin{figure}[H]
    \centering
    \includegraphics[width=0.9\linewidth]{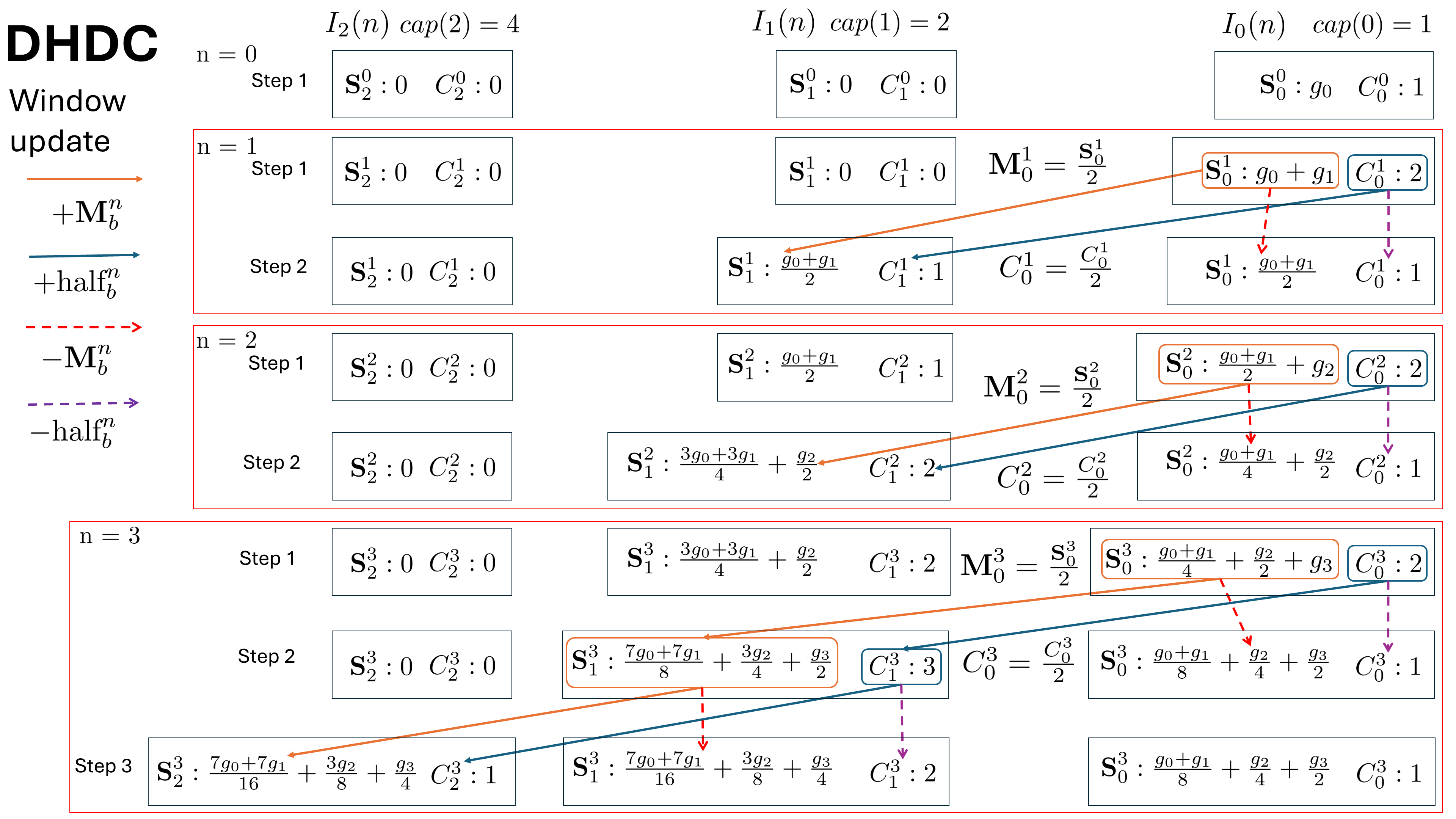}
    \caption{A schematic diagram of the DHDC window update process. The carry operation transfers a fraction of the accumulated gradient mass and count from a finer bin to a coarser bin, yielding a dyadic compressed representation of the gradient history.}
    \label{fig:DHDC_schematic_diagram}
\end{figure}

This operation corresponds to the assumption of uniform redistribution within each bin.
The assumption is used only for compression; the Caputo kernel weighting is applied after the compressed representation has been formed.
The carry rule keeps \(C^n_b\) of order \(2^b\), so that the number of active bins satisfies
\[
    B_n=\mathcal{O}(\log n).
\]
To reconstruct the fractional descent direction, the Caputo kernel is evaluated at the bin level.
For each window \(I_b(n)=[\tau_{b}^{-},\tau_{b}^{+}]\), we define the bin-averaged kernel
\begin{equation}\label{eq:DHDC-kernel}
    \overline{k}_b^{(\alpha)}(n)
    =
    \frac{1}{\Gamma(1-\alpha)}
    \frac{(\tau_{b}^{+})^{1-\alpha}-(\tau_{b}^{-})^{1-\alpha}}
    {(1-\alpha)(\tau_{b}^{+}-\tau_{b}^{-})}\Delta t,
\end{equation}
with endpoint clamping \(t_{\min},t_{\max}\ge \varepsilon_t>0\) to avoid numerical instability.
The compressed convolution is then approximated by
\begin{equation}\label{eq:DHDC-agg}
    \widehat{\bm{g}}^{\mathrm{DHDC}}_n
    =
    \sum_{b=0}^{B_n-1}
    \overline{k}_b^{(\alpha)}(n)\,\bm{S}^n_b.
\end{equation}
Using \eqref{eq:DHDC-weighted-state}, the DHDC approximation can also be written as an effective reweighted convolution:
\begin{equation}\label{eq:DHDC-effective-conv}
    \widehat{\bm{g}}^{\mathrm{DHDC}}_n
    =
    \sum_{j=0}^{n-1}
    \widetilde{w}^{(\alpha)}_{n,j}\bm{g}_j,
    \qquad
    \widetilde{w}^{(\alpha)}_{n,j}
    =
    \sum_{b=0}^{B_n-1}
    \overline{k}_b^{(\alpha)}(n)\,
    \omega^{(n)}_{j,b}.
\end{equation}
Equation \eqref{eq:DHDC-effective-conv} is important because it shows that DHDC remains a discrete convolution-type approximation of the full Caputo memory term, where the effective weights are induced by dyadic redistribution across the compressed history bins.
\begin{remark}
 The architecture of the DHDC differs from the original SBBP formulation in two aspects.
First, DHDC operates directly on the discrete gradient sequence rather than on
subintegrals of a continuous Caputo operator.
Second, it avoids storing higher-order moment information and instead relies on
mass-conserving aggregated statistics, which are sufficient for optimization.
As a result, DHDC preserves the qualitative long-memory structure of the Caputo kernel
while providing a simpler and more efficient implementation. The Algorithm~\ref{alg:DHDC-update} presents the pseudocode of the proposed DHDC.

From an optimization perspective, the per-iteration cost is reduced to
\(\mathcal{O}(\log n)\),
and the stored history is independent of the fractional order.
If \(\alpha\) is adapted during training,
only the scalar weights \(\overline{k}_b(n;\alpha)\) need to be recomputed,
without modifying the compressed history.
\end{remark}

\section{Convergence of Fractional Gradient Descent with Fast-Memory Algorithm}
\label{sec:convergence}

In this section, we establish a convergence guarantee for FGD
under standard smoothness and strong convexity assumptions,
and extend the analysis to practical fast-memory implementations. We begin by analyzing the full discrete Caputo-type convolution \(\widetilde{\bm g}_n\) in \eqref{eq:DHDC-start}.
This full discrete convolution serves as a reference object that captures the intrinsic
long-memory structure of fractional optimization.
Under suitable assumptions, it can be shown that \(\widetilde{\bm g}_n\)
remains positively correlated with the instantaneous gradient,
which ensures descent and linear convergence. Since the fast-memory algorithms (SOE and DHDC) do not compute \(\widetilde{\bm g}_n\) exactly, we construct an approximate direction
\begin{equation}\label{eq:approx_frac_dir}
    \widehat{\bm g}_n = \widetilde{\bm g}_n + \bm e_n,
\end{equation}
where \(\bm e_n\) represents the approximation error induced by
the kernel approximation of the history compression in fast algorithms.

The central question, therefore, is whether such approximations
preserve the convergence properties of \(\widetilde{\bm g}_n\).
To address this, we adopt a perturbation-based perspective:
rather than re-deriving convergence from first principles for each algorithm,
we quantify how deviations from \(\widetilde{\bm g}_n\) affect the descent behavior. Based on this strategy, we show that as long as the approximation error
remains sufficiently small relative to the gradient magnitude,
the fast-memory algorithms retain the descent property and the linear convergence rate
of the full discrete scheme.
This result formally justifies the use of fast-memory algorithms in large-scale learning, 
by demonstrating that computational efficiency can be achieved
without sacrificing the essential convergence guaranties of fractional optimization.

\subsection{Convergence analysis of FGD with full discrete Caputo-tuype convolution}\label{subsec:conv_FGD}

Let $L:\mathbb{R}^d \to \mathbb{R}$ be a differentiable objective function.
Assume that $L$ is $\mu$-strongly convex and has $L$-Lipschitz continuous gradient:
\begin{equation}
    \langle \nabla L(\bm{x}) - \nabla L(\bm{y}), \bm{x} - \bm{y} \rangle
    \ge \mu \|\bm{x} - \bm{y}\|^2,
    \qquad
    \|\nabla L(\bm{x}) - \nabla L(\bm{y})\|
    \le L \|\bm{x} - \bm{y}\|.
\end{equation}

Let $\bm{\theta}^\star = \arg\min L$.
We consider the FGD iteration
\begin{equation}
    \bm{\theta}_{n+1}
    =
    \bm{\theta}_n - \eta\,\widetilde{\bm{g}}_n,
    \qquad
    \widetilde{\bm{g}}_n
    =
    \sum_{j=0}^{n} w_{n-j}\,\bm{g}_j,
    \quad
    \bm{g}_j=\nabla L(\bm{\theta}_j),
\end{equation}
where the weights $\{w_k\}_{k\ge 0}$ are positive and nonincreasing.

\begin{lemma}
\label{lem:correlation}
Assume that there exist constants $C_H>0$ and $\bar{\rho}\ge 1$ such that
\begin{equation}
    \sum_{k=1}^{n} w_k \sum_{i=0}^{k-1} \|\widetilde{\bm{g}}_{n-i-1}\|
    \le
    C_H \max_{0 \le j \le n} \|\bm{g}_j\|,
\end{equation}
and
\begin{equation}
    \max_{0 \le j \le n} \|\bm{g}_j\|
    \le
    \bar{\rho}\,\|\bm{g}_n\|.
\end{equation}
Then, for sufficiently small step size $\eta$, there exists $\gamma\in(0,1)$ such that
\begin{equation}
    \langle \bm{g}_n,\widetilde{\bm{g}}_n\rangle
    \ge
    \gamma w_0 \|\bm{g}_n\|^2.
\end{equation}
\end{lemma}

\begin{proof}
We first expand the inner product between the current gradient $\bm{g}_n$ and the fractional descent direction $\widetilde{\bm{g}}_n$:
\begin{equation}
    \langle \bm{g}_n, \widetilde{\bm{g}}_n \rangle
    = w_0 \|\bm{g}_n\|^2 + \sum_{k=1}^{n} w_k \langle \bm{g}_n, \bm{g}_{n-k} \rangle.
\end{equation}
Using the identity
\[
\langle \bm{g}_n, \bm{g}_{n-k} \rangle
= \|\bm{g}_n\|^2 - \langle \bm{g}_n, \bm{g}_n - \bm{g}_{n-k} \rangle,
\]
we obtain
\begin{equation} \label{eqn:inner_decomposed}
    \langle \bm{g}_n, \widetilde{\bm{g}}_n \rangle
    = \left( \sum_{k=0}^{n} w_k \right) \|\bm{g}_n\|^2 - \sum_{k=1}^{n} w_k \langle \bm{g}_n, \bm{g}_n - \bm{g}_{n-k} \rangle.
\end{equation}
By the $L$-Lipschitz continuity of the gradient,
\begin{equation}
    \|\bm{g}_n - \bm{g}_{n-k}\|
    \le \sum_{i=0}^{k-1}\| \bm{g}_{n-i} - \bm{g}_{n-i-1} \|
    \le L \sum_{i=0}^{k-1} \|\bm{\theta}_{n-i} - \bm{\theta}_{n-i-1}\|
    = L\eta \sum_{i=0}^{k-1} \|\widetilde{\bm{g}}_{n-i-1}\|,
\end{equation}
where we used the update rule $\bm{\theta}_{j} - \bm{\theta}_{j-1} = -\eta \widetilde{\bm{g}}_{j-1}$.
By Cauchy--Schwarz,
\begin{equation}
    \sum_{k=1}^{n} w_k |\langle \bm{g}_n, \bm{g}_n - \bm{g}_{n-k} \rangle|
    \le L\eta \|\bm{g}_n\| \sum_{k=1}^{n} w_k \sum_{i=0}^{k-1} \|\widetilde{\bm{g}}_{n-i-1}\|.
\end{equation}
Using the first assumption in the lemma,
\begin{equation}
    \sum_{k=1}^{n} w_k |\langle \bm{g}_n, \bm{g}_n - \bm{g}_{n-k} \rangle|
    \le L\eta C_H \|\bm{g}_n\| \max_{j\le n} \|\bm{g}_j\|.
\end{equation}
Hence,
\begin{equation}\label{eq:inner_inequality}
    \langle \bm{g}_n, \widetilde{\bm{g}}_n \rangle
    \ge
    \left( \sum_{k=0}^{n} w_k \right)\| \bm{g}_n \|^2 - L\eta C_H\|\bm{g}_n\| \max_{j\le n}\|\bm{g}_j\|.
\end{equation}
Applying the second assumption in the lemma,
\begin{equation}
    \langle \bm{g}_n, \widetilde{\bm{g}}_n \rangle
    \ge
    \left( \sum_{k=0}^{n} w_k - L\eta C_H \bar{\rho} \right) \|\bm{g}_n\|^2
    \ge
    \left( w_0 - L\eta C_H \bar{\rho} \right) \|\bm{g}_n\|^2.
\end{equation}
Therefore, if
\begin{equation}
    \eta \le \frac{(1 - \gamma) w_0}{L C_H \bar{\rho}},
\end{equation}
then
\begin{equation}
    \langle \bm{g}_n, \widetilde{\bm{g}}_n \rangle \geq \gamma w_0 \|\bm{g}_n\|^2.
\end{equation}
This proves the result.
\end{proof}

By using standard smoothness,
\begin{equation}\label{eqn:conv_ineq1}
    L(\bm{\theta}_{n+1})
    \le
    L(\bm{\theta}_n)
    -
    \eta \langle \bm{g}_n,\widetilde{\bm{g}}_n\rangle
    +
    \frac{L\eta^2}{2}\|\widetilde{\bm{g}}_n\|^2.
\end{equation}
By Cauchy--Schwarz,
\begin{equation}\label{eqn:conv_ineq2}
\|\widetilde{\bm{g}}_n\|^2
\le
\Big(\sum_{k=0}^{n} w_k\Big)
\sum_{k=0}^{n} w_k \|\bm{g}_{n-k}\|^2.
\end{equation}
On a finite training horizon \(0\le n\le N\), define
\[
C_\alpha := \max_{0\le m\le N} \left(\sum_{k=0}^{m} w_k\right)^2,
\]
so that
\begin{equation}\label{eqn:conv_ineq3}
    \|\widetilde{\bm{g}}_n\|^2 \le C_\alpha \max_{j\le n}\|\bm{g}_j\|^2.
\end{equation}
By the gradient correlation lemma (Lemma~\ref{lem:correlation}), we have
\[
    \langle \bm{g}_n,\widetilde{\bm{g}}_n\rangle
    \ge
    \gamma w_0 \|\bm{g}_n\|^2.
\]
Substituting \eqref{eqn:conv_ineq3} and the gradient correlation lemma (Lemma~\ref{lem:correlation}) into the descent inequality \eqref{eqn:conv_ineq1} yields
\begin{equation}\label{eqn:conv1}
    L(\bm{\theta}_{n+1})
    \le
    L(\bm{\theta}_n)
    -
    \eta \gamma w_0 \|\bm{g}_n\|^2
    +
    \frac{L\eta^2}{2} C_\alpha \max_{j\le n}\|\bm{g}_j\|^2.
\end{equation}
Recalling the stability ratio \(\bar{\rho}\) assumed in Lemma~\ref{lem:correlation}, we substitute
\(\max_{j \le n} \|\bm{g}_j\|^2 \le \bar{\rho}^2 \|\bm{g}_n\|^2\) into \eqref{eqn:conv1} to obtain
\begin{align}
    L(\bm{\theta}_{n+1})
    &\le L(\bm{\theta}_n) - \eta \gamma w_0 \|\bm{g}_n\|^2 + \frac{L\eta^2 C_\alpha \bar{\rho}^2}{2} \|\bm{g}_n\|^2 \\
    &= L(\bm{\theta}_n) - \eta \left( \gamma w_0 - \frac{L\eta C_\alpha \bar{\rho}^2}{2} \right) \|\bm{g}_n\|^2.
\end{align}
If \(0 < \eta \le \frac{\gamma w_0}{L C_\alpha \bar{\rho}^2}\), then
\begin{equation} \label{eqn:descent_effective}
    L(\bm{\theta}_{n+1}) \le L(\bm{\theta}_n) - \frac{\eta \gamma w_0}{2} \|\bm{g}_n\|^2.
\end{equation}
Subtracting \(L(\bm{\theta}^\star)\) from both sides and applying the Polyak--Łojasiewicz inequality,
which follows from \(\mu\)-strong convexity, gives
\begin{align}
    L(\bm{\theta}_{n+1}) - L(\bm{\theta}^\star)
    &\le L(\bm{\theta}_n) - L(\bm{\theta}^\star) - \frac{\eta \gamma w_0}{2} \left[ 2\mu (L(\bm{\theta}_n) - L(\bm{\theta}^\star)) \right] \\
    &= (1 - \eta \gamma w_0 \mu) (L(\bm{\theta}_n) - L(\bm{\theta}^\star)).
\end{align}
Defining \(E_n := L(\bm{\theta}_n) - L(\bm{\theta}^\star)\), we obtain
\(E_{n+1} \le \kappa E_n\) with contraction factor \(\kappa = 1 - \eta \gamma w_0 \mu\).

To ensure geometric convergence, it remains to verify that \(0 < \kappa < 1\).
Since \(\eta, \gamma, w_0,\) and \(\mu\) are strictly positive, we immediately have \(\kappa < 1\).
Moreover, using the step-size bound \( \eta \le \frac{\gamma w_0}{L C_\alpha \bar{\rho}^2}\), we obtain
\begin{equation}
    \eta \gamma w_0 \mu
    \le
    \frac{(\gamma w_0)^2 \mu}{L C_\alpha \bar{\rho}^2}
    =
    \left( \frac{\mu}{L} \right)
    \gamma^2
    \left( \frac{w_0^2}{C_\alpha} \right)
    \left( \frac{1}{\bar{\rho}^2} \right).
\end{equation}
Since \(\mu \le L\), \(\gamma<1\), \(w_0^2 \le C_\alpha\), and \(\bar{\rho}\ge 1\), the right-hand side is strictly less than \(1\).
Hence \(0 < \kappa < 1\), and therefore
\begin{equation}
    L(\bm{\theta}_n) - L(\bm{\theta}^\star) \le \kappa^n (L(\bm{\theta}_0) - L(\bm{\theta}^\star)).
\end{equation}

Finally, strong convexity implies the quadratic growth bound
\begin{equation}
    L(\bm{\theta}_n) - L(\bm{\theta}^\star) \ge \frac{\mu}{2} \|\bm{\theta}_n - \bm{\theta}^\star\|^2.
\end{equation}

Combining the previous results, we obtain the following inequality:
\begin{equation}
    \|\bm{\theta}_n - \bm{\theta}^\star\|^2
    \le
    \frac{2}{\mu} \left( L(\bm{\theta}_n) - L(\bm{\theta}^\star) \right)
    \le
    \frac{2}{\mu} \kappa^n (L(\bm{\theta}_0) - L(\bm{\theta}^\star)).
\end{equation}
By applying the $L$-smoothness property,
\begin{equation}
    L(\bm{\theta}_0) - L(\bm{\theta}^\star) \le \frac{L}{2}\|\bm{\theta}_0 - \bm{\theta}^\star\|^2,
\end{equation}
therefore, the distance to the optimum at step $n$ is bounded relative to its initial value as follows:
\begin{equation}
    \|\bm{\theta}_n - \bm{\theta}^\star\|^2
    \le
    \frac{L}{\mu}\kappa^n \|\bm{\theta}_0 - \bm{\theta}^\star\|^2.
\end{equation}

This confirms geometric convergence of the parameter sequence to the unique minimizer \(\bm{\theta}^\star\).

\begin{remark}
As $\alpha \to 1$, the weights concentrate near the present step,
and the method approaches classical gradient descent.
For $\alpha < 1$, the history-weighted averaging introduces additional damping,
which enlarges the stability region while preserving linear convergence.
\end{remark}

\subsection{Error estimate of DHDC}

We now quantify the error introduced by the DHDC compression. 
The purpose of this analysis is not to claim that DHDC is an exact discretization of the Caputo convolution, but to show that its approximation error is explicitly controlled by the variation of the Caputo kernel inside each dyadic bin. Let
\[
    k_\alpha(t)=\frac{t^{-\alpha}}{\Gamma(1-\alpha)}\Delta t,
    \qquad 0<\alpha<1,
\]
and consider the full-history discrete Caputo-type convolution
\begin{equation}
    \widetilde{\bm g}_n
    =
    \sum_{j=0}^{n-1}
    k_\alpha((n-j)\Delta t)\bm g_j .
    \label{eq:ideal-dhdc-conv}
\end{equation}
For each iteration $n$, let $\{J_b(n)\}_{b=0}^{B_n-1}$ be a disjoint partition of the history indices $\{0,\ldots,n-1\}$ into DHDC bins. 
For each bin, define the time-lag interval
\[
    I_b(n)=[\tau_b^-,\tau_b^+],
    \qquad
    \tau_b^- \le (n-j)\Delta t \le \tau_b^+
    \quad \text{for all } j\in J_b(n).
\]

The DHDC scheme specialty formulates the descent direction as an effective reweighted convolution, adopting the structural representation established in \cref{eq:DHDC-effective-conv}:
\begin{equation}\label{eq:dhdc-approx-conv}
    \widehat{\bm{g}}^{\mathrm{DHDC}}_n = \sum_{j=0}^{n-1} \widetilde{w}^{(\alpha)}_{n,j}\bm{g}_j, \qquad \widetilde{w}^{(\alpha)}_{n,j} = \sum_{b=0}^{B_n-1} \overline{k}_b^{(\alpha)}(n) \omega^{(n)}_{j,b}.
\end{equation}
The parameter $\overline{k}_b^{(\alpha)}(n)$ denotes the Caputo kernel averaged over the dyadic look-back window $I_b(n) = [\tau_b^-, \tau_b^+]$. Within this framework, the aggregation weights $\omega_{j,b}^{(n)}$ distribute the historical gradient contributions across the dyadic bins while satisfying the partition of unity property:
\begin{equation}\label{eq:partition-of-unity}
    \sum_{b=0}^{B_n-1} \omega_{j,b}^{(n)} = 1,\quad j = 0,1,\ldots,n-1.
\end{equation}

The approximation error between the DHDC direction and the full-history Caputo convolution is evaluated in the following theorem:

\begin{theorem}[\textit{DHDC convolution error bound}] \label{thm:dhdc-error-bound}
    Let $\widetilde{\bm g}_n$ be the full-history convolution in \eqref{eq:ideal-dhdc-conv}, and let $\widehat{\bm g}^{\mathrm{DHDC}}_n$ be the DHDC approximation in \eqref{eq:dhdc-approx-conv}. 
    Assume that the gradient sequence is bounded on each bin in the sense that
    \[
        R_b(n):=\sum_{j\in J_b(n)}\|\bm g_j\| < \infty .
    \]
    Then the DHDC approximation error
    \[
        \bm e_n^{\mathrm{DHDC}}
        :=
        \widehat{\bm g}^{\mathrm{DHDC}}_n-\widetilde{\bm g}_n
    \]
    satisfies
    \begin{equation}
        \|\bm e_n^{\mathrm{DHDC}}\|
        \le
        \sum_{b=0}^{B_n-1}
        \delta_b^{(\alpha)}(n) R_b(n),
        \label{eq:dhdc-error-bound-general}
    \end{equation}
    where
    \begin{equation}
        \delta_b^{(\alpha)}(n)
        :=
        \sup_{t\in I_b(n)}
        \left|
        k_\alpha(t)-\bar k_b^{(\alpha)}(n)
        \right|.
        \label{eq:dhdc-bin-error}
    \end{equation}
    Moreover, since $k_\alpha(t)$ is positive and strictly decreasing,
    \begin{equation}
        \delta_b^{(\alpha)}(n)
        \le
        k_\alpha(\tau_b^-)-k_\alpha(\tau_b^+)
        =
        \frac{(\tau_b^-)^{-\alpha}-(\tau_b^+)^{-\alpha}}
        {\Gamma(1-\alpha)}\Delta t.
        \label{eq:dhdc-explicit-bin-error}
    \end{equation}
    Consequently,
    \begin{equation}
        \|\bm e_n^{\mathrm{DHDC}}\|
        \le
        \frac{\Delta t}{\Gamma(1-\alpha)}
        \sum_{b=0}^{B_n-1}
        \left[
        (\tau_b^-)^{-\alpha}
        -
        (\tau_b^+)^{-\alpha}
        \right]
        R_b(n).
        \label{eq:dhdc-explicit-error}
    \end{equation}
\end{theorem}

\begin{proof}
Using the definitions of $\widetilde{\bm g}_n$ and $\widehat{\bm g}^{\mathrm{DHDC}}_n$, we write
\[
    \bm e_n^{\mathrm{DHDC}}
    =
    \sum_{b=0}^{B_n-1}
    \sum_{j\in J_b(n)}
    \left[
    \bar k_b^{(\alpha)}(n)
    -
    k_\alpha((n-j)\Delta t)
    \right] \omega_{j,b}^{n} \bm g_j .
\]

For every index $j \in J_b(n)$, the corresponding lag $(n-j)\Delta t$ resides within the interval $I_b(n)$. The definition of $\delta_b^{(\alpha)}(n)$ in \cref{eq:dhdc-bin-error} yields the following inequality:
\begin{equation}
    \|\bm e_n^{\mathrm{DHDC}}\|
    \le
    \sum_{b=0}^{B_n-1}
    \sum_{j\in J_b(n)}
    \left|
    \bar k_b^{(\alpha)}(n)
    -
    k_\alpha((n-j)\Delta t)
    \right|
    \|\bm g_j\|.
\end{equation}
For every index $j \in J_b(n)$, the corresponding lag $(n-j)\Delta t$ resides within the interval $I_b(n)$. This membership allows the kernel discrepancy to be controlled by the supremum of the variation:
\begin{equation}
    \left|
    \bar k_b^{(\alpha)}(n)
    -
    k_\alpha((n-j)\Delta t)
    \right|
    \le
    \delta_b^{(\alpha)}(n).
\end{equation}
Substitution of this bound into the norm inequality establishes that
\begin{equation}
    \|\bm e_n^{\mathrm{DHDC}}\|
    \le
    \sum_{b=0}^{B_n-1}
    \delta_b^{(\alpha)}(n)
    \sum_{j\in J_b(n)}\|\bm g_j\|,
\end{equation}
which completes the proof of \cref{eq:dhdc-error-bound-general}.

The Mean Value Theorem for Integrals establishes that the bin average $\bar k_b^{(\alpha)}(n)$ resides within the functional range of $k_\alpha(t)$ over the interval $I_b(n)$. Since $k_\alpha(t)$ is positive and strictly decreasing on $(0, \infty)$, this average value is constrained by the kernel evaluated at the endpoints:
\begin{equation}
    k_\alpha(\tau_b^+) \le \bar k_b^{(\alpha)}(n) \le k_\alpha(\tau_b^-).
\end{equation}
Consequently, for any $t \in I_b(n)$, the distance between the kernel and its average satisfies the following inequality:
\begin{equation}
    \left|k_\alpha(t)-\bar k_b^{(\alpha)}(n)\right| \le k_\alpha(\tau_b^-)-k_\alpha(\tau_b^+).
\end{equation}

Evaluation of the boundary difference using the explicit definition of the kernel $k_\alpha(t) = \frac{t^{-\alpha}}{\Gamma(1-\alpha)}\Delta t$ yields:
\begin{equation}
    k_\alpha(\tau_b^-)-k_\alpha(\tau_b^+) = \frac{(\tau_b^-)^{-\alpha}-(\tau_b^+)^{-\alpha}}{\Gamma(1-\alpha)}\Delta t.
\end{equation}
This identity establishes the validity of \cref{eq:dhdc-explicit-bin-error}, which in turn completes the proof of the total error bound in \cref{eq:dhdc-explicit-error}.

\end{proof}

The previous estimate can be specialized to dyadic DHDC bins. Rather than
seeking an \(N\)-only rate, the dyadic structure is used here to isolate
the dependence on the time step. This shows that, for a fixed training
index and a fixed dyadic depth, the DHDC approximation discrepancy
vanishes as the step size decreases.

\begin{corollary}[\textit{Fixed-index step-size decay of the dyadic-bin DHDC error}]
\label{cor:dhdc-dyadic-step-consistency}
    Assume that the DHDC bins satisfy the dyadic sequence
    \[
        I_b(n)
        =
        [t_n-2^{b+1}\Delta t,\ t_n-2^b\Delta t],
        \qquad
        b=0,\ldots,B_n-1,
    \]
    where \(B_n\in\mathbb N\) is chosen such that \(2^{B_n}<n\). Let
    \[
        \tau_b^- = t_n - 2^{b+1}\Delta t,
        \qquad
        \tau_b^+ = t_n - 2^b\Delta t.
    \]
    Suppose that the gradients satisfy
    \[
        \|\bm g_j\|\le G_n,
        \qquad
        j=0,1,\ldots,n-1,
    \]
    with \(G_n\) independent of \(\Delta t\), and assume that \(n\le N\).
    Then
    \begin{equation}
        \|\bm e_n^{\mathrm{DHDC}}\|
        \le
        \frac{
            \alpha N G_n
        }{\Gamma(1-\alpha)}
        (2^{B_n}-1)
        (n-2^{B_n})^{-\alpha-1}
        \Delta t^{1-\alpha}.
        \label{eq:dhdc-dyadic-step-consistency}
    \end{equation}
    Consequently, for fixed discrete indices \(n\), \(N\), fixed dyadic depth
    \(B_n\), and a \(\Delta t\)-independent gradient bound \(G_n\),
    \[
        \|\bm e_n^{\mathrm{DHDC}}\|
        =
        \mathcal O(\Delta t^{1-\alpha})
        \qquad
        \text{as } \Delta t\to0.
    \]
\end{corollary}

\begin{proof}
    From the bin-wise estimate in \cref{thm:dhdc-error-bound}, we have
    \[
        \|\bm e_n^{\mathrm{DHDC}}\|
        \le
        \sum_{b=0}^{B_n-1}
        \delta_b^{(\alpha)}(n)
        \sum_{j\in J_b(n)}
        \|\bm g_j\|.
    \]
    By the definition of the bin-wise approximation coefficient,
    \[
        \delta_b^{(\alpha)}(n)
        \le
        \frac{
            (\tau_b^-)^{-\alpha}
            -
            (\tau_b^+)^{-\alpha}
        }{\Gamma(1-\alpha)}
        \Delta t.
    \]
    Since
    \[
        \tau_b^-
        =
        t_n-2^{b+1}\Delta t
        =
        (n-2^{b+1})\Delta t,
        \qquad
        \tau_b^+
        =
        t_n-2^b\Delta t
        =
        (n-2^b)\Delta t,
    \]
    we obtain
    \[
    \begin{aligned}
        (\tau_b^-)^{-\alpha}
        -
        (\tau_b^+)^{-\alpha}
        &=
        \Delta t^{-\alpha}
        \left[
            (n-2^{b+1})^{-\alpha}
            -
            (n-2^b)^{-\alpha}
        \right].
    \end{aligned}
    \]
    Therefore,
    \[
    \begin{aligned}
        \|\bm e_n^{\mathrm{DHDC}}\|
        &\le
        \sum_{b=0}^{B_n-1}
        \frac{
            \Delta t^{-\alpha}
            \left[
                (n-2^{b+1})^{-\alpha}
                -
                (n-2^b)^{-\alpha}
            \right]
        }{\Gamma(1-\alpha)}
        \Delta t
        \sum_{j\in J_b(n)}
        \|\bm g_j\| .
    \end{aligned}
    \]
    Using the uniform gradient bound and \(n\le N\), we have
    \[
        \sum_{j\in J_b(n)}
        \|\bm g_j\|
        \le
        nG_n
        \le
        NG_n.
    \]
    Hence,
    \[
    \begin{aligned}
        \|\bm e_n^{\mathrm{DHDC}}\|
        &\le
        \frac{
            \Delta t^{1-\alpha}NG_n
        }{\Gamma(1-\alpha)}
        \sum_{b=0}^{B_n-1}
        \left[
            (n-2^{b+1})^{-\alpha}
            -
            (n-2^b)^{-\alpha}
        \right].
    \end{aligned}
    \]
    The summation telescopes:
    \[
    \begin{aligned}
        \sum_{b=0}^{B_n-1}
        \left[
            (n-2^{b+1})^{-\alpha}
            -
            (n-2^b)^{-\alpha}
        \right]
        &=
        (n-2^{B_n})^{-\alpha}
        -
        (n-1)^{-\alpha}.
    \end{aligned}
    \]
    Hence,
    \[
        \|\bm e_n^{\mathrm{DHDC}}\|
        \le
        \frac{
            \Delta t^{1-\alpha}NG_n
        }{\Gamma(1-\alpha)}
        \left[
            (n-2^{B_n})^{-\alpha}
            -
            (n-1)^{-\alpha}
        \right].
    \]
    Now let
    \[
        f(x)=x^{-\alpha}.
    \]
    Since $f'(x)=-\alpha x^{-\alpha-1}$, the mean value theorem applied on $[n-2^{B_n},\, n-1]$ gives
    \[
    \begin{aligned}
        (n-2^{B_n})^{-\alpha}
        -
        (n-1)^{-\alpha}
        &\le
        \alpha(2^{B_n}-1)
        (n-2^{B_n})^{-\alpha-1}.
    \end{aligned}
    \]
    Therefore,
    \[
        \|\bm e_n^{\mathrm{DHDC}}\|
        \le
        \frac{
            \alpha N G_n
        }{\Gamma(1-\alpha)}
        (2^{B_n}-1)
        (n-2^{B_n})^{-\alpha-1}
        \Delta t^{1-\alpha}.
    \]
    Since \(0<\alpha<1\), we have \(1-\alpha>0\). Thus, for fixed
    \(n\), \(N\), \(B_n\), and \(G_n\),
    \[
        \Delta t^{1-\alpha}\to0
        \qquad
        \text{as } \Delta t\to0.
    \]
    Consequently,
    \[
        \|\bm e_n^{\mathrm{DHDC}}\|
        =
        \mathcal O(\Delta t^{1-\alpha})
        \qquad
        \text{as } \Delta t\to0.
    \]
\end{proof}

\begin{remark}\leavevmode 
    \begin{enumerate}
    \item The estimate in \cref{thm:dhdc-error-bound} shows that the DHDC error is governed by two factors: the variation of the Caputo kernel within each compressed bin and the accumulated gradient magnitude over that bin. Hence, the DHDC scheme should be understood as a controlled compression of the full Caputo convolution rather than as an exact replacement.

    \item In particular, for a fixed dyadic ratio, the relative variation of the kernel inside each bin does not vanish automatically. Therefore, an arbitrarily small approximation error requires additional control, such as adaptive refinement of the bins, exact storage of recent gradients, or smoothness and cancellation in the gradient history.

    \end{enumerate}
\end{remark}

\subsection{Perturbation stability of fractional gradient descent}
\label{subsec:dhdc-convergence-preservation}


In this subsection, we show that the convergence properties of the full discrete convolution FGD are preserved whenever the perturbation \(\bm e_n\) remains below the admissible stability threshold.

\begin{theorem}[\textit{Convergence preservation under fast-memory approximation}]
    \label{thm:approx_fast_memory_convergence}
    Assume that \(L:\mathbb{R}^d\to\mathbb{R}\) is \(\mu\)-strongly convex and has
    \(L\)-Lipschitz continuous gradient.
    Suppose that the ideal fractional direction \(\widetilde{\bm{g}}_n\) satisfies
    \begin{equation}\label{eq:ideal_correlation_assump}
        \langle \bm{g}_n,\widetilde{\bm{g}}_n\rangle
        \ge \gamma w_0 \|\bm{g}_n\|^2
    \end{equation}
    for some \(\gamma\in(0,1)\), as established in Lemma~\ref{lem:correlation}, and that there exists a constant \(C_F>0\) such that
    \begin{equation}\label{eq:ideal_dir_bound}
        \|\widetilde{\bm{g}}_n\|\le C_F \|\bm{g}_n\|.
    \end{equation}
    Assume further that the approximation error satisfies
    \begin{equation}
        \|\bm{e}_n\|\le \rho \|\bm{g}_n\|
        \label{eq:error_relative_bound}
    \end{equation}
    for some \(\rho\ge 0\).
    
    Then the practical update
    \begin{equation}
        \bm{\theta}_{n+1}
        =
        \bm{\theta}_n - \eta \widehat{\bm{g}}_n
        \label{eq:practical_fast_update}
    \end{equation}
    satisfies the descent estimate
    \begin{equation}
        L(\bm{\theta}_{n+1})
        \le
        L(\bm{\theta}_n)
        -
        \eta
        \left(
        \gamma w_0 - \rho - \frac{L\eta}{2}(C_F+\rho)^2
        \right)
        \|\bm{g}_n\|^2.
        \label{eq:approx_descent}
    \end{equation}
    
    In particular, if
    \begin{equation}
        \rho < \gamma w_0
        \label{eq:rho_condition}
    \end{equation}
    and
    \begin{equation}
        0<\eta<
        \frac{2(\gamma w_0-\rho)}{L(C_F+\rho)^2},
        \label{eq:eta_condition_approx}
    \end{equation}
    then \(L(\bm{\theta}_n)\) decreases monotonically.
    Moreover,
    \begin{equation}
        L(\bm{\theta}_{n})-L(\bm{\theta}^\star)
        \le
        \left(
        1-2\mu\eta
        \left(
        \gamma w_0 - \rho - \frac{L\eta}{2}(C_F+\rho)^2
        \right)
        \right)^n
        \big(L(\bm{\theta}_0)-L(\bm{\theta}^\star)\big),
        \label{eq:approx_linear_rate}
    \end{equation}
    so the approximate FGD with a fast-memory algorithm converges linearly.
\end{theorem}

\begin{proof}
Using the smoothness of \(L\),
\begin{equation}
L(\bm{\theta}_{n+1})
\le
L(\bm{\theta}_n)
-
\eta \langle \bm{g}_n,\widehat{\bm{g}}_n\rangle
+
\frac{L\eta^2}{2}\|\widehat{\bm{g}}_n\|^2.
\label{eq:proof_smoothness}
\end{equation}
Substituting \eqref{eq:approx_frac_dir} gives
\[
\langle \bm{g}_n,\widehat{\bm{g}}_n\rangle
=
\langle \bm{g}_n,\widetilde{\bm{g}}_n\rangle
+
\langle \bm{g}_n,\bm{e}_n\rangle.
\]
By \eqref{eq:ideal_correlation_assump} and the Cauchy--Schwarz inequality,
\[
\langle \bm{g}_n,\widehat{\bm{g}}_n\rangle
\ge
\gamma w_0 \|\bm{g}_n\|^2
-
\|\bm{g}_n\|\,\|\bm{e}_n\|.
\]
Using \eqref{eq:error_relative_bound},
\begin{equation}
\langle \bm{g}_n,\widehat{\bm{g}}_n\rangle
\ge
(\gamma w_0-\rho)\|\bm{g}_n\|^2.
\label{eq:proof_inner}
\end{equation}

Similarly,
\[
\|\widehat{\bm{g}}_n\|
\le
\|\widetilde{\bm{g}}_n\|+\|\bm{e}_n\|
\le
(C_F+\rho)\|\bm{g}_n\|,
\]
where we used \eqref{eq:ideal_dir_bound} and \eqref{eq:error_relative_bound}.
Hence
\begin{equation}
\|\widehat{\bm{g}}_n\|^2
\le
(C_F+\rho)^2 \|\bm{g}_n\|^2.
\label{eq:proof_norm}
\end{equation}

Substituting \eqref{eq:proof_inner} and \eqref{eq:proof_norm} into \eqref{eq:proof_smoothness}
yields
\[
L(\bm{\theta}_{n+1})
\le
L(\bm{\theta}_n)
-
\eta
\left(
\gamma w_0-\rho-\frac{L\eta}{2}(C_F+\rho)^2
\right)
\|\bm{g}_n\|^2,
\]
which proves \eqref{eq:approx_descent}.

If \eqref{eq:rho_condition} and \eqref{eq:eta_condition_approx} hold, the coefficient of
\(\|\bm{g}_n\|^2\) is strictly positive, and the loss decreases monotonically.

Finally, strong convexity implies the Polyak--Łojasiewicz inequality
\[
\|\bm{g}_n\|^2
\ge
2\mu \big(L(\bm{\theta}_n)-L(\bm{\theta}^\star)\big),
\]
which, combined with \eqref{eq:approx_descent}, gives \eqref{eq:approx_linear_rate}.
\end{proof}

\begin{corollary}[\textit{Implication for SOE and DHDC}]
\label{cor:SOE_DHDC_convergence}

Suppose that the SOE or DHDC approximation satisfies
\[
\|\bm{e}_n\|\le \rho \|\bm{g}_n\|, \qquad \rho<\gamma w_0.
\]
Then the corresponding fast-memory FGD method inherits the descent property
and linear convergence of the ideal fractional method up to the perturbation level \(\rho\).

In particular, smaller approximation error enlarges the admissible step-size range
and yields a convergence rate closer to that of the ideal Caputo-consistent direction.

\end{corollary}

\section{Numerical Experiments}\label{sec:numerical_experiments}

We evaluate the SOE and the newly proposed DHDC schemes, comparing them against the standard full-history DC method. The SOE implementation requires $\mathcal{O}(M)$ memory and $\mathcal{O}(M)$ operations per iteration, where $M$ denotes the number of exponential modes. The decay parameters $\{\xi_m\}$ are selected on a logarithmic grid over the interval $[1/t_{\max},1/t_{\min}]$, and the weights $\{\omega_m\}$ are obtained by least-squares fitting of the Caputo kernel over $t\in[t_{\min},t_{\max}]$. This strategy preserves the long-memory structure while enabling recursive updates.

The DHDC implementation achieves $\mathcal{O}(\log n)$ memory and computational cost at iteration $n$ by hierarchically aggregating past gradients. The bin-averaged kernel admits an analytic expression, ensuring that positivity and monotonic decay are retained.
When $\alpha \to 1$, the logarithmic limit form of the kernel average is used to maintain numerical stability. To ensure stable descent, the learning rate $\eta$ is chosen according to the convergence condition derived in Section~\ref{sec:convergence}.
In practice, a safe heuristic over a finite training horizon is
\[
    \eta \approx c\,\frac{w_0/L}{\sum_{k=0}^{N} w_k}, \qquad c\in(0,1],
\]
which guarantees monotone decrease under smooth strong convexity. To assess the performance of the proposed SOE and DHDC, we consider one optimization benchmark problem and three representative engineering and learning tasks: battery state-of-health estimation, stochastic signal regression, and CIFAR-10 image classification.

\begin{table}[H]
    \centering
    \caption{Training configurations for each benchmark problem.}
    \begin{tabular}{lccccccc}
        \toprule
        task    & optimizer & loss          & adaptive LR & architecture & batch      & epochs$\times$iterations \\
        \midrule
        Rastrigin & None & Rastrigin function & Yes & MLP & full-batch & 10000$\times$1 \\
        battery & RMSprop   & MSE-based     & Yes         & MLP          & full-batch & 1000$\times$1   \\
        signal  & Adam      & MSE-based     & Yes         & TCN          & 128        & 1$\times$150000 \\
        image   & Adam      & Cross-entropy & Yes         & ResNet       & 256        & 25$\times$195   \\
        \bottomrule
    \end{tabular}
\end{table}

For the stochastic signal regression task, each iteration is treated as an individual epoch to maintain consistency in our performance tracking.

\subsection{Rastrigin Benchmark Problem}\label{sec:rastri}


We consider the Rastrigin benchmark optimization problem to investigate the convergence behavior and optimization dynamics of the proposed FGD methods. The Rastrigin function is widely used as a challenging nonconvex benchmark due to its highly oscillatory landscape containing numerous local minima surrounding a unique global minimum.

The baseline optimizers considered in this experiment are standard gradient descent (GD), momentum-based gradient descent, and Adam. These methods are compared against three distinct variants of FGD, namely the DC, SOE, and DHDC schemes. All optimizers share identical hyperparameter settings whenever applicable. The momentum coefficient is fixed at $0.9$, and the fractional order is set to $\alpha = 0.5$ unless stated otherwise.

The Rastrigin function itself is used directly as the loss function. Because of its large number of local minima, standard gradient-based optimizers typically encounter two competing difficulties: large learning rates often lead to instability or divergence, whereas small learning rates frequently cause the optimization trajectory to become trapped in suboptimal local minima. To address this issue, an adaptive learning-rate strategy is employed throughout the experiment. The learning rate is gradually reduced once the optimization trajectory approaches the vicinity of the global minimizer.



The Rastrigin function is defined as follows
\begin{equation}\label{eq:rastri}
    f(\mathbf{x}) = An + \sum_{i=1}^{n} \left[ x_i^2 - A \cos(2\pi x_i) \right], \quad A = 10, \ n = 10.
\end{equation}
\begin{equation}
    \mathbf{x} = [x_1,...,x_n]
\end{equation}
The initial point is chosen at a suitable location that does not coincide with any local minima, as $\textbf{x}^0=[2.22,...,2.22]$
The optimization point is defined as the location at which the loss function~\eqref{eq:rastri} attains zero, as $\textbf{x}^\ast=[0.0,...,0.0], \ f(\textbf{x}^\ast)=0$.

\begin{equation}
\mathcal{L}_{\mathrm{Rastrigin}}
=
\left|
f(\mathbf{x})
\right|
.
\label{eq:loss_rastrigin}
\end{equation}

\subsection{Battery State-of-Health Estimation}\label{sec:battery}

To illustrate the capability of the proposed FGD framework, we consider battery state-of-health and capacity estimation, a central problem in lithium-ion energy storage systems.
Battery degradation is governed by diffusion and charge-transfer processes, which exhibit long-memory behavior and are well captured by fractional-order equivalent circuit models (FO-ECMs).

\begin{table}[h]
    \centering
    \caption{variable table in Battery State-of-Health Estimation}
    \label{tab:variables}
    \begin{tabular}{cl} 
        \toprule
        \textbf{sign} & \textbf{means}  \\
        \midrule
        $V(t)$ & Voltage at the time \textit{t}  \\
        $I(t)$ & Current at the time \textit{t}  \\
        $z(t)$ & State-of-charge(SoC) at the time \textit{t}  \\
        $C_1$ & Fractional capacitance \\
        $R_1$ & Polarization resistance \\
        $R_0$ & Ohmic resistance \\
        $v_1$ & Fractional side-brach voltage \\
        $Q$ & Nominal Capacity \\
        $\alpha$ & fractional differentiation order \\
        \bottomrule
    \end{tabular}
\end{table}


In general, an FO-ECM replaces the classical capacitor relation with a fractional constitutive law of the form
\begin{equation}\label{eq:fo-ecms}
I(t)=C_1\,{}^{C}D_t^\alpha V(t)+\frac{1}{R_1}V(t),
\quad 0<\alpha<1,
\end{equation}
thereby introducing power-law memory into the voltage–current dynamics. In the FO-ECM \eqref{eq:fo-ecms}.


In the adopted FO-ECM, the fractional side-branch voltage $v_1(t)$
satisfies the Caputo fractional differential equation
\begin{equation}
{}^{C}D_t^\alpha v_1(t)
=
-\frac{1}{R_1 C_1} v_1(t)
+
\frac{1}{C_1} I(t),
\qquad 0<\alpha<1,
\label{eq:ocvv1}
\end{equation}
The terminal voltage is modeled as
\begin{equation}
V(t)
=
\mathrm{OCV}(z(t))
-
R_0 I(t)
-
v_1(t),
\end{equation}
where $R_0$ represents the ohmic resistance,
$\mathrm{OCV}(z)$ denotes the open-circuit voltage as a function of the state-of-charge (SoC) $z(t)$.
The SoC evolves according to Coulomb counting:
\begin{equation}
z(t+\Delta t)
=
z(t)
-
\frac{I(t)\Delta t}{3600\,Q},
\end{equation}

Figure \ref{fig:loss_decomp} illustrates the complete training pipeline, showing how the input measurements, trainable FO-ECM parameters, and the neural network jointly produce the voltage prediction and fractional residual, which are subsequently combined with physics-based regularization to form the unified battery loss.

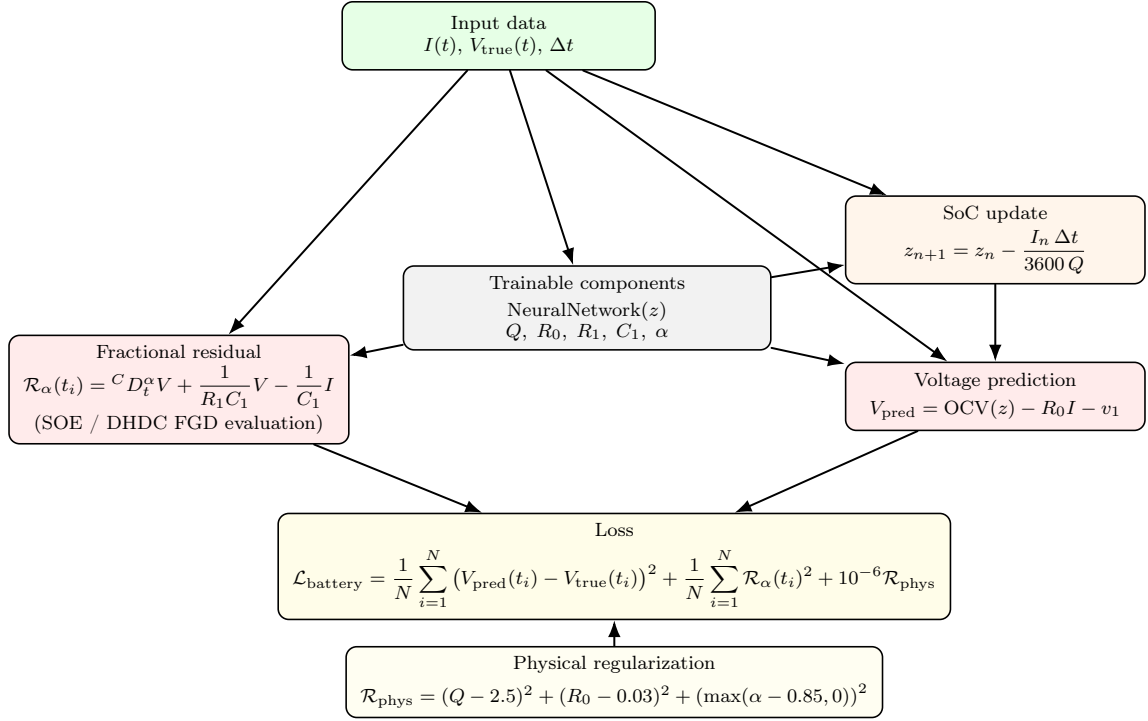
\begin{figure}[H]
\centering
\begin{tikzpicture}[
  scale=0.9,
  transform shape,
  font=\footnotesize,
  >=Latex,
  line cap=round,
  line join=round,
  arrow/.style={->, line width=0.75pt},
  box/.style={
    draw, rounded corners=4pt,
    minimum height=10mm,
    align=center,
    line width=0.6pt,
    inner xsep=6pt, inner ysep=4pt
  }
]

\node[box, fill=green!10, minimum width=46mm] (data) at (5.7,5)
{Input data \\ $I(t),\,V_{\mathrm{true}}(t),\,\Delta t$};

\node[box, fill=gray!10, minimum width=54mm] (params) at (7,1)
{Trainable components \\[2pt]
$\mathrm{Neural Network}(z)$ \\[-1pt]
$Q,\;R_0,\;R_1,\;C_1,\;\alpha$};

\node[box, fill=orange!8, minimum width=44mm] (soc) at (13,2)
{SoC update \\[2pt]
$z_{n+1}=z_n-\dfrac{I_n\,\Delta t}{3600\,Q}$};

\node[box, fill=red!8, minimum width=44mm] (residual) at (1,-0.2)
{Fractional residual \\[2pt]
$\mathcal{R}_\alpha(t_i)
=
{}^{C}D_t^\alpha V
+\dfrac{1}{R_1C_1}V
-\dfrac{1}{C_1}I$ \\[2pt]
(SOE / DHDC FGD evaluation)};

\node[box, fill=red!8, minimum width=44mm] (vpred) at (13,-0.3)
{Voltage prediction \\[2pt]
$V_{\mathrm{pred}}
=
\mathrm{OCV}(z)-R_0 I - v_1$};

\node[box, fill=yellow!10, minimum width=60mm] (loss) at (7.4,-2.8)
{Loss \\[4pt]
$\displaystyle
\mathcal{L}_{\mathrm{battery}}
=
\frac{1}{N}\sum_{i=1}^{N}
\left(V_{\mathrm{pred}}(t_i)-V_{\mathrm{true}}(t_i)\right)^2
+
\frac{1}{N}\sum_{i=1}^{N}
\mathcal{R}_\alpha(t_i)^2
+
10^{-6}\mathcal{R}_{\mathrm{phys}}
$};

\node[box, fill=yellow!6, minimum width=60mm] (reg) at (7.4,-4.5)
{Physical regularization \\[4pt]
$\mathcal{R}_{\mathrm{phys}}
=
(Q-2.5)^2
+
(R_0-0.03)^2
+
\left(\max(\alpha-0.85,0)\right)^2$};

\draw[arrow] (data) -- (params);
\draw[arrow] (params) -- (soc);
\draw[arrow] (data) -- (soc);

\draw[arrow] (params) -- (residual);
\draw[arrow] (data) -- (residual);

\draw[arrow] (soc) -- (vpred);
\draw[arrow] (params) -- (vpred);
\draw[arrow] (data) -- (vpred);

\draw[arrow] (vpred) -- (loss);
\draw[arrow] (residual) -- (loss);
\draw[arrow] (reg) -- (loss);

\end{tikzpicture}
\caption{Battery State Estimation Flowchart Combining Neural Network and FO-ECM}
\label{fig:loss_decomp}
\end{figure}

This network consists of three hidden layers of width $32$ with $\tanh$ activation,
providing a flexible representation of nonlinear voltage–SoC relationships.
The FO-ECM \eqref{eq:ocvv1} parameters
$(Q, R_0, R_1, C_1, \alpha)$
are treated as trainable variables.
To preserve physical admissibility,
positivity and bound constraints are enforced through suitable reparameterization.

Let $\{t_i\}_{i=1}^N$ denote the discrete sampling times.
The battery loss is defined as
\begin{equation}
\mathcal{L}_{\mathrm{battery}}
=
\frac{1}{N}\sum_{i=1}^{N}
\left(
V_{\mathrm{pred}}(t_i)-V_{\mathrm{true}}(t_i)
\right)^2
+
\frac{1}{N}\sum_{i=1}^{N}
\left(
\mathcal{R}_\alpha(t_i)
\right)^2
+
10^{-6}\,\mathcal{R}_{\mathrm{phys}}.
\label{eq:loss_battery}
\end{equation}
The first term measures the discrepancy between predicted and measured terminal voltage.
The second term penalizes deviations from the fractional dynamical constraint
\begin{equation}
\mathcal{R}_\alpha(t_i)
=
{}^{C}D_t^\alpha v_1(t_i)
+
\frac{1}{R_1C_1}v_1(t_i)
-
\frac{1}{C_1}I(t_i),
\label{eq:fractional_residual}
\end{equation}
thereby enforcing consistency with the underlying FO-ECM \eqref{eq:ocvv1}.
Here, ${}^{C}D_t^\alpha v_1(t_i)$ is discretized by the $L1$ scheme; see Appendix~\ref{Appendix:A}.
The regularization term
\begin{equation}
\mathcal{R}_{\mathrm{phys}}
=
(Q-2.5)^2
+
(R_0-0.03)^2
+
\left(\max(\alpha-0.85,0)\right)^2,
\label{eq:reg_battery}
\end{equation}
discourages physically unrealistic parameter drift.
The small weighting factor $10^{-6}$ of \eqref{eq:loss_battery} ensures that regularization \eqref{eq:reg_battery} stabilizes
the optimization without overwhelming the data-fitting objective.




Battery training pipeline under the proposed FGD framework. Input measurements are mapped to trainable FO-ECM parameters and the neural network. 

\subsection{Stochastic Signal Estimation}\label{sec:signal}
To evaluate the effectiveness of fractional memory in a controlled setting,
we consider the estimation of a stochastic signal with long-range temporal dependence.
Many real-world time series exhibit correlations that decay slowly over time,
a phenomenon that cannot be adequately captured by memoryless optimization dynamics.
This setting provides a natural testbed for examining whether the intrinsic
power-law memory of FGD
offers advantages over classical gradient descent.


The synthetic signal is generated using a fractional linear filter
\[
    c_k = c_{k-1}\frac{(k-1)+d}{k},
    \qquad c_0=1,
    \qquad d=0.1,
\]
where $\varepsilon_t \sim \mathcal{N}(0,\sigma^2)$ with $\sigma=10$.
The coefficients $\{c_k\}$ decay algebraically, so that $c_k \sim k^{d-1}$ as $k\to\infty$. Signals value at time t is determined at \eqref{eq:sig}. 
\begin{equation}\label{eq:sig}
    y_{true}(t) = \sum_{k=0}^{\infty} c_k\,\varepsilon_{t-k},
\end{equation}
This power-law decay induces persistent correlations, and the autocovariance of $y_{true}(t)$ decreases slowly rather than exponentially.
Consequently, the present state of the signal depends on information from the distant past.



Figure \ref{fig:tcn} illustrates the architecture of the TCN used in the signal estimation task. The task consists of stacked causal convolutional blocks with exponentially increasing dilation factors $d = 1, 2, 4, 8, 16$. Each block applies two causal convolution layers with kernel size $k=3$, followed by $\tanh$ activation and dropout, and maintains 64 feature channels throughout the network\cite{TCNECG}. The left-sided padding ensures that each output at time $t$ depends only on past inputs, thereby preventing future leakage. The use of dilation enlarges the receptive field without increasing network depth, enabling the task to capture long-range temporal dependencies efficiently. The final layer aggregates the hierarchical temporal features to produce the output sequence. This architecture allows for stable and parallelizable training while preserving strict temporal causality.


\begin{figure}[H]
    \centering
    \includegraphics[width=0.5\linewidth]{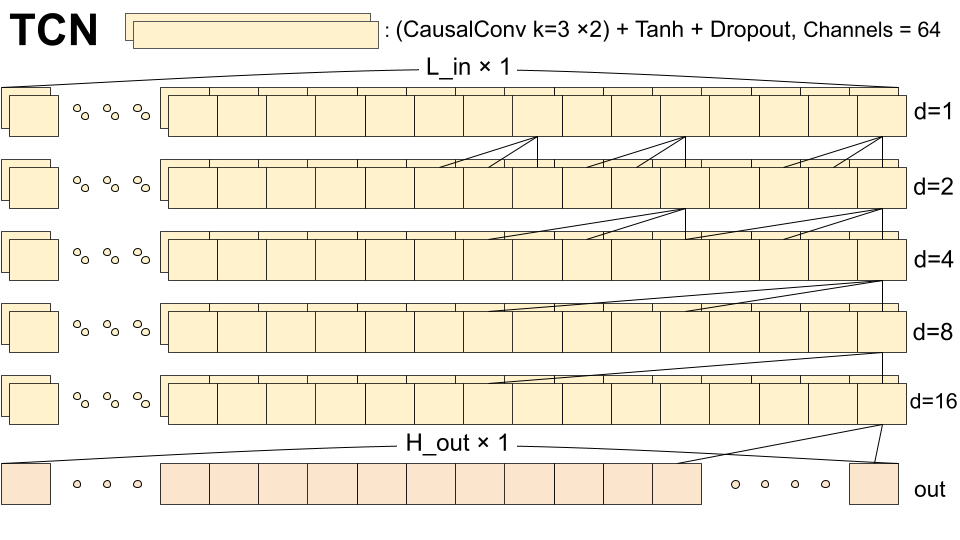}
    \caption{Schematic diagram of TCN}
    \label{fig:tcn}
\end{figure}

Let $\{t_i\}_{i=1}^{N}$ denote the discrete time indices.
The first term in \eqref{eq:loss_signal} represents the empirical reconstruction error, measuring the discrepancy between the predicted and observed signals. Let the input tensor be denoted as $\mathbf{X} \in \mathbb{R}^{B \times L \times D}$, where $B$ is the batch size, $L = 512$ is the look-back window length, and $D = 1$ is the number of input features. For each sample $i$ within a batch, the input sequence $\mathbf{x}^{(i)}$ is defined as:$$y_t = y_{true}(t)\qquad\mathbf{x}^{(i)} = [y_{t-L}, y_{t-L+1}, \dots, y_{t-1}]^\top.$$The objective of the task is to learn a mapping function $f: \mathbb{R}^{L \times D} \to \mathbb{R}$ that minimizes the discrepancy between the predicted value $y_{pred}$ and the ground truth $y_{true}$:$$y_{pred}(t_i) = f(\mathbf{x}^{(i)}; \theta),$$where $\theta$ represents the trainable parameters of the architecture.
The unified loss for stochastic signal estimation is defined as
\begin{equation}\label{eq:loss_signal}
    \mathcal{L}_{\mathrm{signal}}
    =
    \frac{1}{N}\sum_{i=1}^{N}
    \left(
    y_{\mathrm{pred}}(t_i)
    -
    y_{\mathrm{true}}(t_i)
    \right)^2
    +
    \lambda
    \frac{1}{N}\sum_{i=1}^{N}
    \left(
    {}^{C}D_t^\alpha y_{\mathrm{pred}}(t_i)
    \right)^2,
\end{equation}
where $\lambda = 10^{-3}$.


The second term penalizes the squared Caputo fractional derivative of the predicted signal,
thereby enforcing fractional dynamical consistency. In locally stationary regimes, the fractional derivative should remain small.
Minimizing this term promotes smooth temporal behavior
while respecting the intrinsic long-memory structure of the signal.
The weighting parameter $\lambda = 10^{-3}$
controls the trade-off between predictive fidelity and fractional smoothness.

Through the formulation \eqref{eq:loss_signal},
the loss function integrates data fitting and fractional regularization
within a unified variational framework.
Because both the data-generating process and the regularization term
exhibit power-law memory,
the training objective aligns naturally with the optimization dynamics
introduced by FGD.



\subsection{CIFAR-10 Image Classification}\label{sec:image}

To evaluate the proposed FGD framework in a high-dimensional nonconvex setting, we consider image classification on the CIFAR-10 dataset. This experiment allows us to investigate whether fractional optimization provides benefits beyond regression-type tasks and extends to deep convolutional architectures.

Unlike the stochastic signal experiment, where long-range dependence is embedded in the data-generating process, the CIFAR-10 task provides a different perspective: here, memory is not inherent in the input data but arises in the optimization landscape itself. Deep convolutional networks induce highly nonconvex loss surfaces characterized by sharp minima, saddle regions, and stochastic gradient noise introduced by mini-batch training.

We adopt a ResNet-18 architecture as the backbone task.
ResNet is introduced to address the degradation problem observed in very deep convolutional networks, where increasing depth leads to training instability and performance saturation. Instead of directly learning a transformation $H(x)$, each residual block learns a residual mapping
\begin{equation}
    F(x) = H(x) - x,
\end{equation}
so that the block output becomes
\begin{equation}
    y = x + F(x).
\end{equation}

This identity-based skip connection improves gradient propagation and mitigates vanishing or exploding gradient issues. 
As a result, deep networks maintain stable optimization behavior, making ResNet an appropriate architecture for isolating the effects of the optimization scheme.

Figure \ref{fig:resnet} illustrates the ResNet-18 architecture used in our experiments.

\begin{figure}[H]
    \centering
    \includegraphics[width=0.65\linewidth]{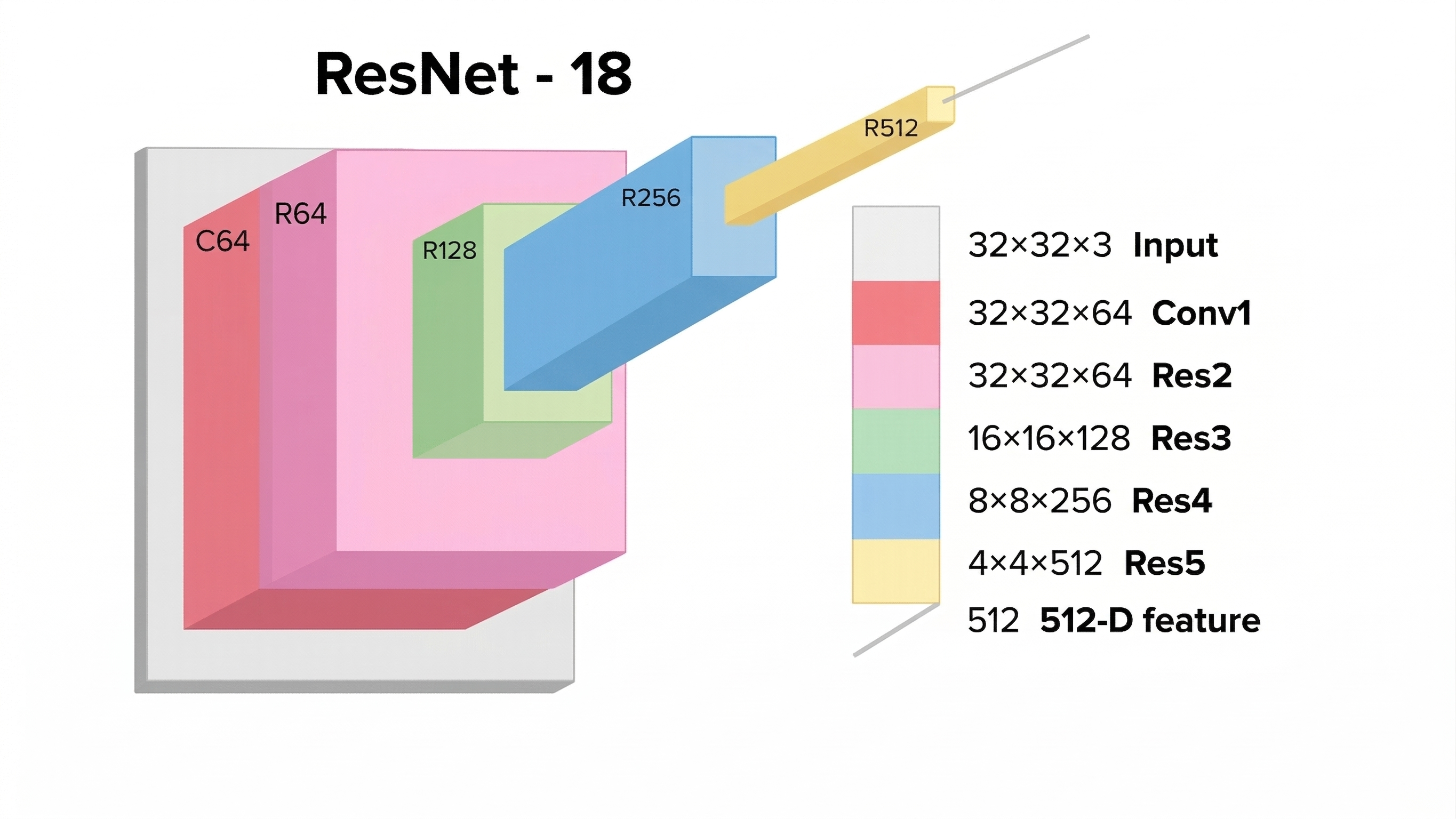}
    \caption{Schematic diagram of ResNet-18}
    \label{fig:resnet}
\end{figure}

Let $\mathbf{z}_i$ denote the logits for sample $i$, 
and let $\mathbf{p}_i = \mathrm{softmax}(\mathbf{z}_i)$, and $\mathbf{y}_i$ be the one-hot encoded ground-truth label, $\mathbf{y}_i^{\mathrm{rand}}$ a randomly permuted label vector
They are available each $C=10$ classes. Therefore, they are defined by the class and sample indices, which are 
denoted respectively as subscripts. 
Then, The mixed cross-entropy loss is defined as
\begin{equation}
\mathcal{L}_{\mathrm{ce}}
=
\frac{1}{N}
\sum_{i=1}^{N}
\left[
\beta
\left(
- \sum_{c=1}^{C}
y_{i,c}\log p_{i,c}
\right)
+
(1-\beta)
\left(
- \sum_{c=1}^{C}
y^{\mathrm{rand}}_{i,c}\log p_{i,c}
\right)
\right].
\label{eq:loss_image_ce}
\end{equation}

This formulation \eqref{eq:loss_image_ce} combines supervision from the true labels with a randomly permuted label distribution. 
The mixing coefficient $\beta \in [0,1]$ controls the contribution of each term, acting as a regularizer to prevent overconfident predictions.

To further improve robustness, we introduce a symmetric Kullback–Leibler (KL) regularization term. Let $\mathbf{q}_i$ denote a noise-perturbed prediction obtained from a stochastic forward pass. 

$\mathbf{q}_i$ is specifically defined as:

\begin{enumerate}
    \item Geometric transformations : Random cropping with reflective padding and horizontal flipping.
    \item AutoAugment : The CIFAR-10 policy is applied for automated strategy selection.
    \item Noise \& Occlusion : We further inject Gaussian noise ($\sigma=0.10$) and apply Random Erasing ($p=0.5$) to encourage the model to learn invariant features even under severe information loss.
\end{enumerate}



The symmetric KL divergence is defined as

\begin{equation}\label{eq:loss_image_kl}
    \mathcal{L}_{\mathrm{KL}}
    =
    \frac{1}{N}
    \sum_{i=1}^{N}
    \left[
    \sum_{c=1}^{C}
    p_{i,c}\log\frac{p_{i,c}}{q_{i,c}}
    +
    \sum_{c=1}^{C}
    q_{i,c}\log\frac{q_{i,c}}{p_{i,c}}
    \right].
\end{equation}

The final classification objective is therefore

\begin{equation}\label{eq:loss_image}
    \mathcal{L}_{\mathrm{image}}
    =
    \mathcal{L}_{\mathrm{ce}}
    +
    0.05\,\mathcal{L}_{\mathrm{KL}}.
\end{equation}

The symmetric KL term discourages excessive divergence between stochastic forward passes, promoting stable probability distributions during training.


Training is performed using mini-batches to control memory usage, making the experiments feasible on standard personal computing hardware. 
Although mini-batch training reduces per-iteration memory cost, it introduces gradient noise, thereby providing a realistic nonconvex stochastic optimization environment.


\section{Numerical Results}\label{sec:numerical_result}





This section presents the empirical results for the various benchmark tasks introduced in Section \ref{sec:numerical_experiments}. We evaluate the performance of our proposed methods across these tasks using specific neural network architectures and loss functions.

First, the Rastrigin Benchmark Problem \ref{sec:rastri} is utilized as a fundamental case to analyze the behavioral differences between FGD and standard GD within complex local minima environments. Subsequently, the battery \ref{sec:battery}, signal \ref{sec:signal}, and image \ref{sec:image} tasks are presented as representative real-world applications. Through the numerical results of these real-world tasks—namely battery, signal, and image processing—we evaluate and compare the performance of our proposed fast algorithms, DHDC and SOE, against standard GD and the first-order Taylor series approximation. Furthermore, details regarding the efficiency of these fast algorithms, specifically their theoretical time complexity and actual execution times observed in real-world problems, are provided in \ref{Appendix:timecom}.

\subsection{Optimization Dynamics and Empirical Computational Complexity on the Rastrigin Function}\label{subsec:rastrigin}

In this subsection, we present the numerical results for the Rastrigin benchmark function previously introduced in \cref{sec:rastri}. Given that the Rastrigin landscape is characterized by a high density of local minima, it serves as an ideal environment to examine the behavior of each optimizer within a highly nonconvex setting.

\begin{figure}[H]
    \centering
    \includegraphics[width=0.5\linewidth]{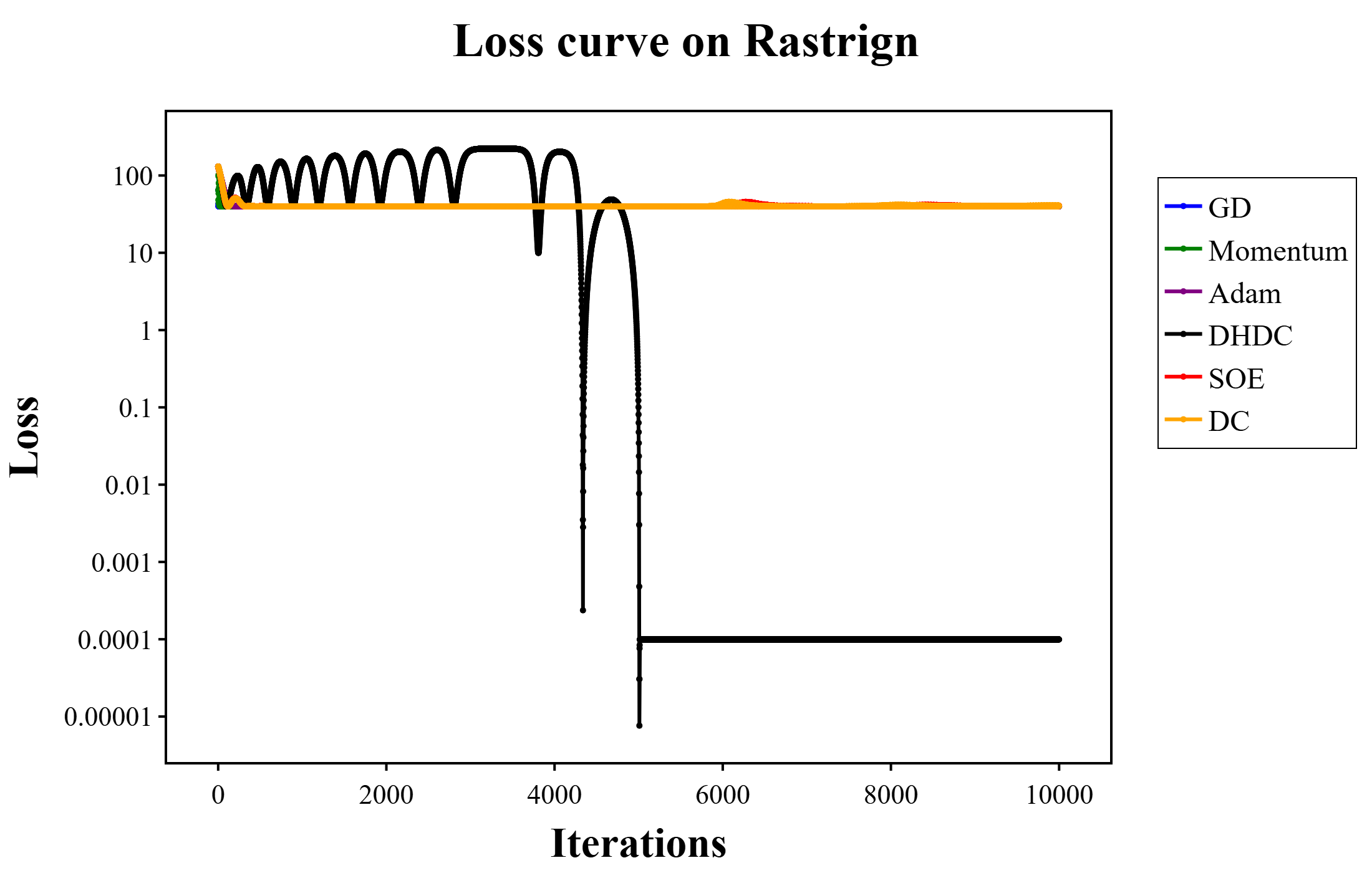}
    \caption{Per-iteration loss on the Rastrigin function.}
    \label{fig:rastrigin}
\end{figure}

\begin{figure}[H]
    \centering
    \includegraphics[width=0.47\linewidth]{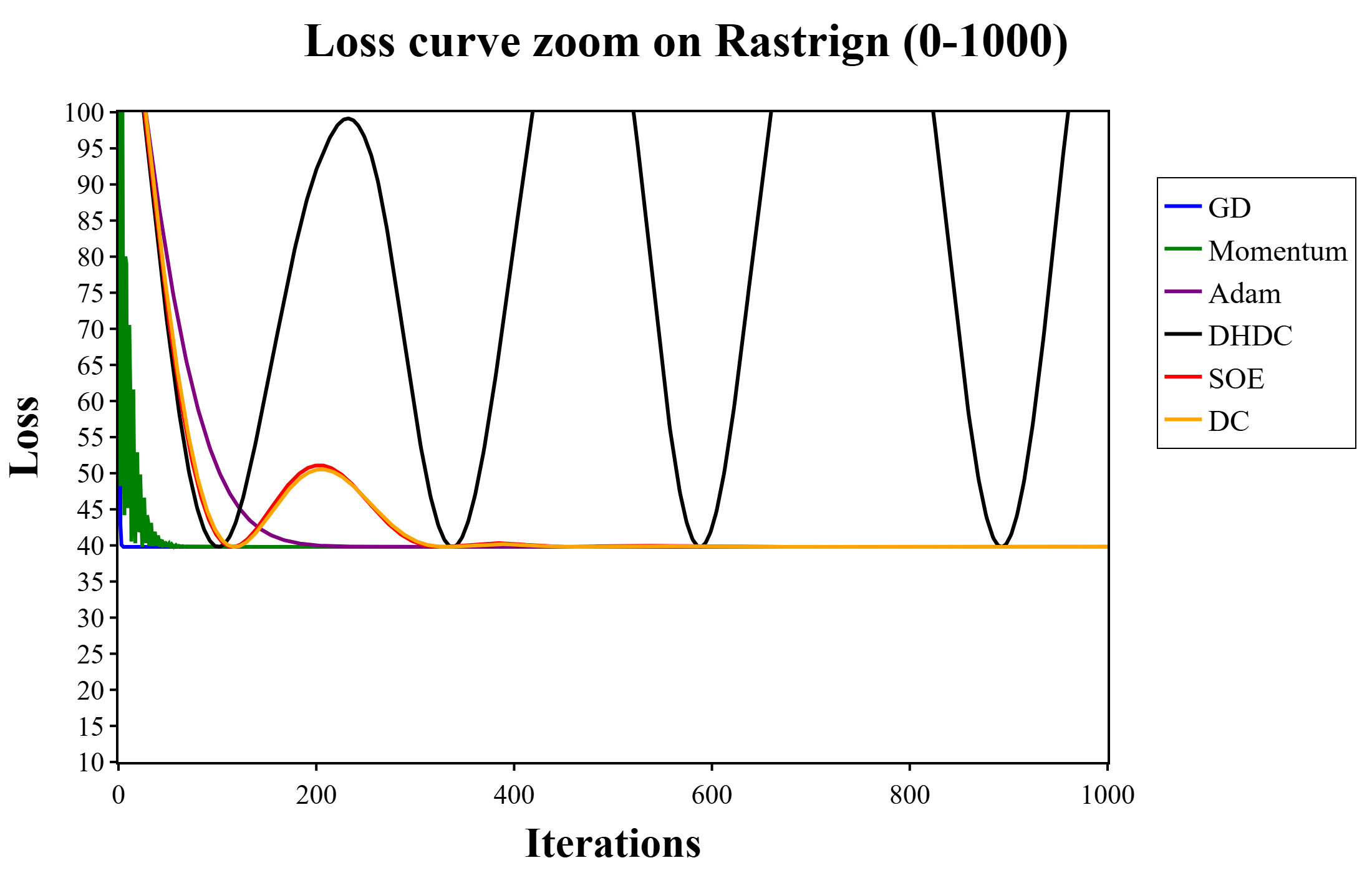}
    \includegraphics[width=0.47\linewidth]{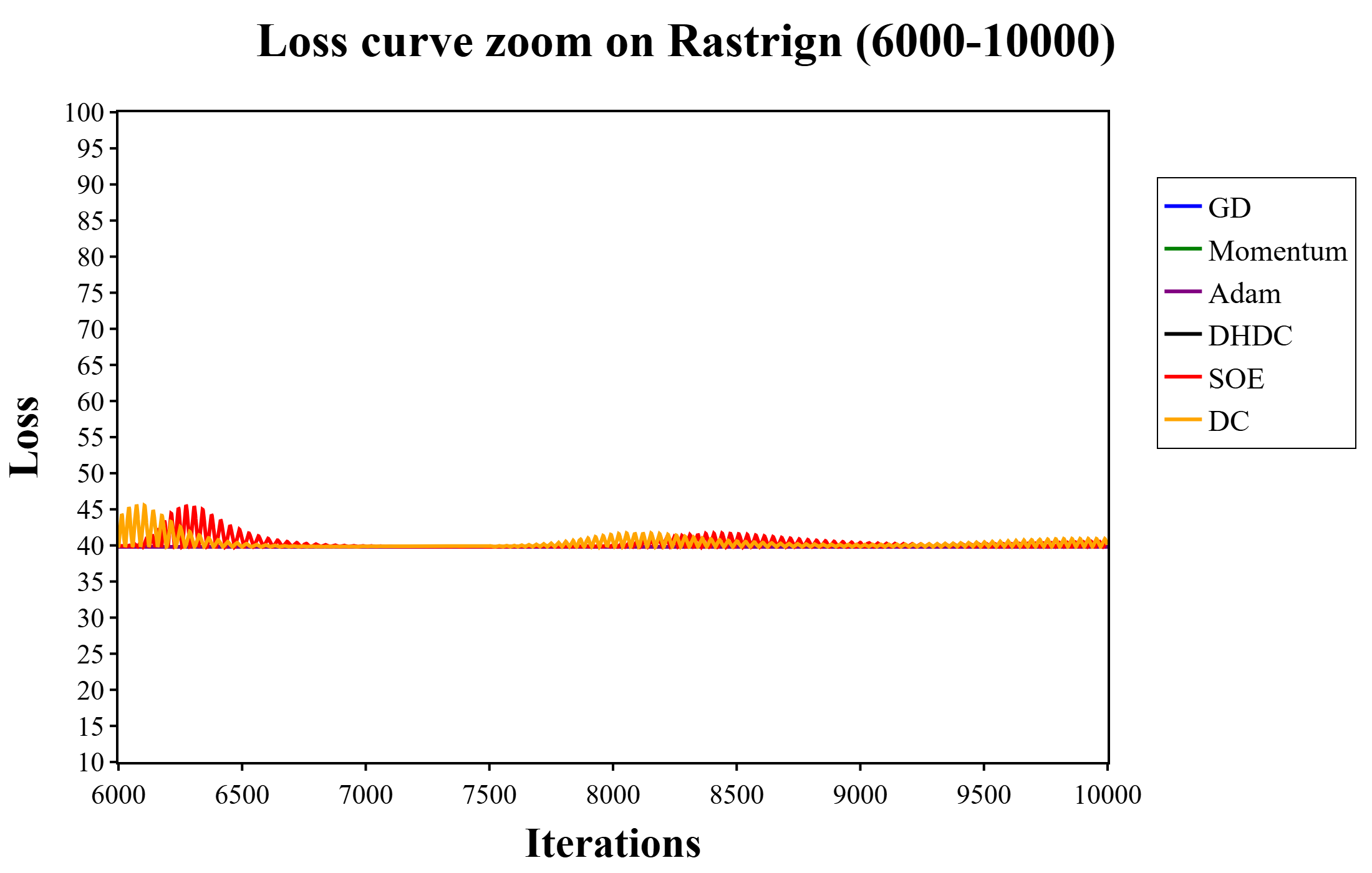}
    \caption{Per-iteration loss on the zoomed Rastrigin function.}
    \label{fig:rastrigin_zoom}
\end{figure}

Figure~\ref{fig:rastrigin} illustrates the loss trajectories obtained for this task. While all other compared methods rapidly reduce the loss in the early stages only to eventually settle into local minima, the proposed DHDC continues to move across multiple local minima, ultimately approaching the global minimum. 

A more detailed analysis of these dynamics is provided in Figure~\ref{fig:rastrigin_zoom}. The left panel offers a magnified view of the early-stage dynamics (within the ranges $10 \le L \le 100$ and $0 \le n \le 1000$), while the right panel highlights the later-stage behavior ($10 \le L \le 100$ and $6\times 10^3 \le n \le 10^4$). Although the specific trajectories vary, a clear trend emerges: while most methods become trapped relatively early, DHDC continues to explore the landscape for a significantly longer duration, avoiding premature convergence.



In Figure \ref{fig:gradrastri}, Its behavior appears to be closely related to the persistent oscillatory dynamics induced by the compressed long-memory structure
In particular, we see that the FGD methods maintain larger gradient norms than standard GD throughout most of the optimization process. Unlike conventional GD methods, where the gradient norm rapidly decays near stationary points, DHDC continues to exhibit a moderate but non-negligible level of gradient activity even after reaching low-loss regions. As a result, the optimization trajectory does not become prematurely stationary and can repeatedly escape shallow local minima.
This persistent but bounded oscillatory behavior allows DHDC to continue exploring the optimization landscape while still stabilizing near the global minimizer. These results suggest that the compressed memory dynamics introduced by DHDC can be especially effective in highly multimodal optimization problems, where conventional methods often suffer from premature convergence.

\begin{figure}[H]
    \centering
    \includegraphics[width=0.5\linewidth]{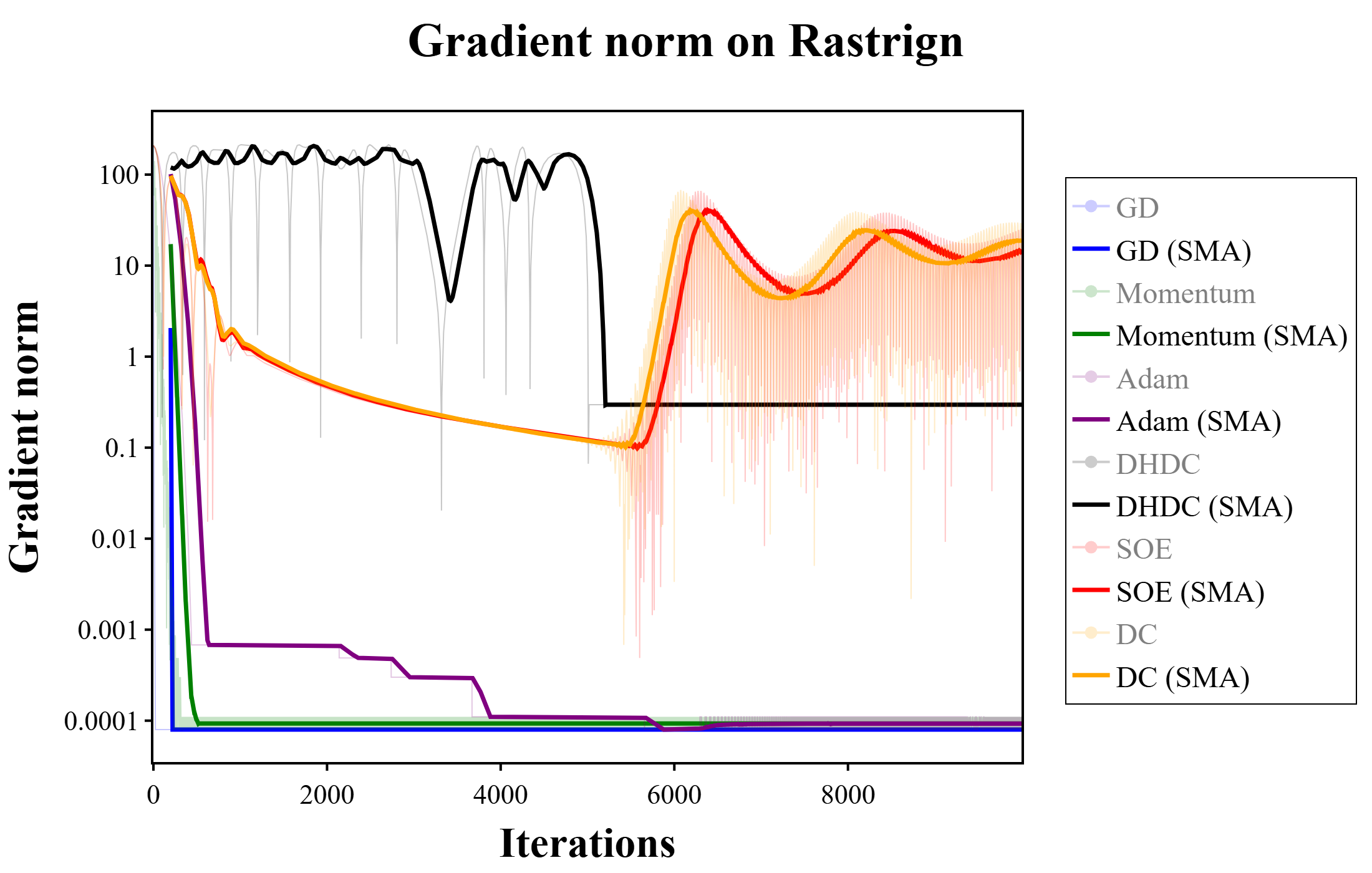}
    \caption{Per-iteration gradient norm. SMA means simple moving average}
    \label{fig:gradrastri}
\end{figure}

The oscillatory behavior induced by fractional memory can improve exploration when sufficiently controlled, but excessive oscillation may also destabilize the optimization trajectory. To investigate this effect, we measure the magnitude of oscillation while varying the fractional order \(\alpha\) from \(0.1\) to \(0.9\) in increments of \(0.1\).
The oscillation indicator is defined by

\begin{equation}\label{eq:oscill}
    \mathcal{O} = \frac{1}{N-1} \sum_{n=1}^{N-1}\left| \mathcal{L}(\bm{\theta}_{n+1})-\mathcal{L}(\bm{\theta}_n)\right|,
\end{equation}
where $\mathcal{L}(\bm{\theta}_i)$ denotes the evaluation of the $f(\mathbf{x})$ at the $i$-th iteration in \eqref{eq:rastri}.
\begin{figure}[H]
    \centering
    \includegraphics[width=0.5\linewidth]{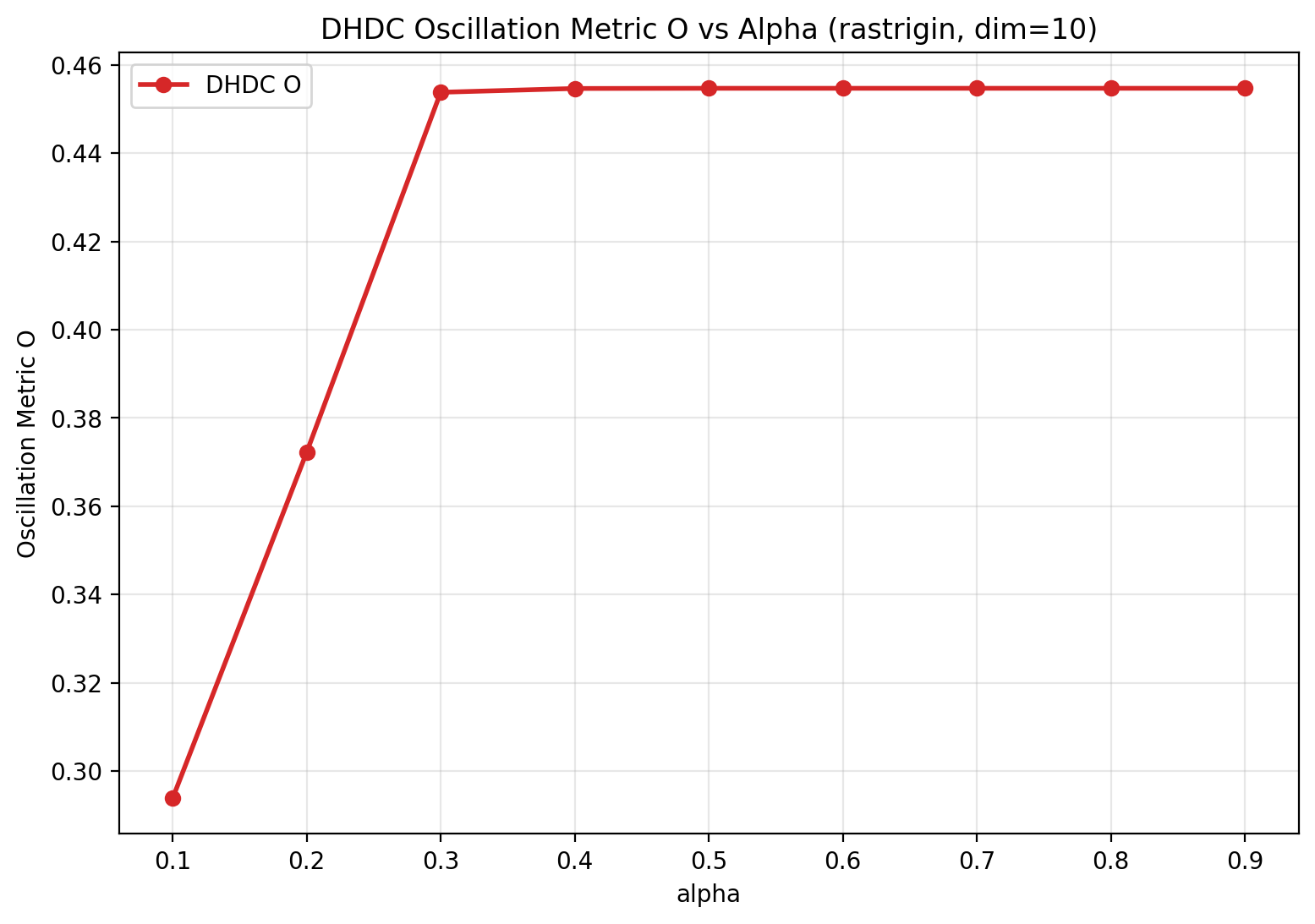}
    \caption{Values of $\mathcal{O}$ as a function of the fractional order. As $\alpha$ varies, $\mathcal{O}$ does not diverge, and it exhibits the convergence trend as $\alpha$ increases.}
    \label{fig:ocill}
\end{figure}





The Fig.~\ref{fig:ocill} shows that the oscillation magnitude remains bounded once the fractional order becomes moderately large. This reflects the transition toward more localized gradient behavior similar to classical gradient descent.
On the other hand, when \(\alpha\) becomes smaller than approximately \(0.3\), the strong history-averaging effect substantially attenuates the effective gradient magnitude. As a result, the oscillatory component of the optimization dynamics is also weakened, leading to smaller values of \(\mathcal{O}\).

\subsubsection{Experimental difference between SOE and DC}\label{par:soedc}

Since SOE is a fast algorithm proposed based on the DC framework, the numerical results presented here demonstrate that there is a negligible difference between SOE and DC in terms of accuracy. In this regard, from the perspective of loss and accuracy, SOE and DC can be considered equivalent methods, as the discrepancy between them falls within mathematically permissible error margins. Consequently, in Section \ref{par:soedc}, we demonstrate the methodological equivalence of SOE and DC by utilizing the results obtained from the Rosenbrock benchmark problem. To evaluate the numerical fidelity of the SOE approximation, we compare its performance with the standard DC method using the Rosenbrock function, which is a standard benchmark characterized by its curved and ill-conditioned valley. It is defined as follows:

\begin{equation}\label{eq:rosen}
    f(\mathbf{x}) = \sum_{i=1}^{n-1} \left[ 100(x_{i+1} - x_i^2)^2 + (1 - x_i)^2 \right] \quad n = 10
\end{equation}




The global minimizer $\mathbf{x}^\star$ and the initial point $\mathbf{x}^{(0)} = [x_1^{(0)},x_2^{(0)},\ldots,x_n^{(0)}]$ for the numerical test are defined as:
\begin{align}
    \mathbf{x}^\star &= [1.0, \dots, 1.0]^\top, \quad f(\mathbf{x}^\star) = 0, \label{mini} \\
    x_i^{(0)} &= \begin{cases} 
    -1.2 & \text{if } i \text{ is odd} \\
    1.0 & \text{if } i \text{ is even}
    \end{cases} \quad (i = 1, 2, \dots, n). \label{ini}
\end{align}




\begin{table}[htbp]
    \centering
    \begin{tabular}{cc}
        \toprule
        Epoch & Loss Gap \\
        \midrule
        100 & 0.02462 \\
        200 & 0.02196 \\
        300 & 0.00992 \\
        400 & 0.00264 \\
        500 & 0.00020 \\
        600 & 0.00059 \\
        700 & 0.00003 \\
        800 & 0.00003 \\
        900 & 0.00003 \\
        \bottomrule
    \end{tabular}
    \caption{Loss gap between SOE and DC across epochs.}
    \label{tab:soel1}
\end{table}

Table~\ref{tab:soel1} presents the results regarding the difference between the SOE and DC methods. The results reveal that the absolute difference in loss between the two schemes remains within a strictly admissible numerical bound from the early stages of training. Notably, as the optimization proceeds, this gap diminishes and eventually converges toward zero, reaching a negligible magnitude of $3.0 \times 10^{-5}$ by epoch 700. These findings provide a numerical justification for replacing the computationally intensive DC method with SOE in real-world applications, as SOE achieves nearly identical accuracy with significantly higher efficiency.


\subsection{Optimization of Battery State-of-Health Estimation}\label{subsec:battery}

This subsection presents the numerical results for the battery State-of-Health estimation task introduced in \cref{sec:battery}. As observed in the results, the battery estimation task exhibits oscillations during the early stages of training, which gradually diminish as the optimization proceeds. Such oscillatory behavior is often observed in systems with long-memory effects. In this experimental setting, the FGD-based methods demonstrate improved performance compared to standard GD.





\begin{figure}[H]
    \centering
    \includegraphics[width=0.6\linewidth]{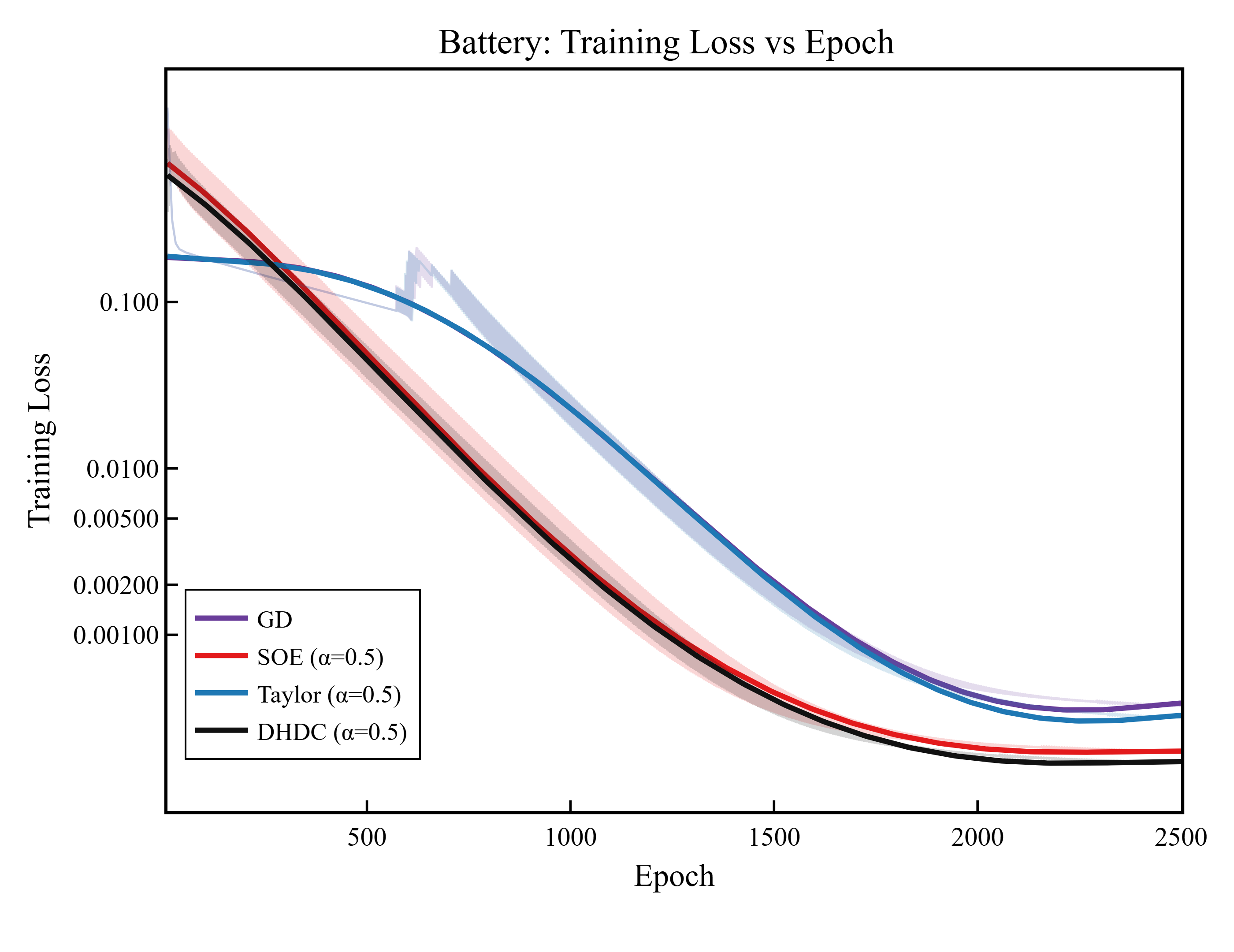}
    \caption{Training loss curves over the full training horizon (epochs 10--2500).}
    \label{fig:battery_full}
\end{figure}

Looking at Figure~\ref{fig:battery_full}, which shows the training loss over epochs, the bold lines represent the Step-wise Moving Average(SMA), allowing for a clearer observation of the overall convergence trends. This figure reveals that DHDC achieves a lower training loss compared to the competing methodologies. While fluctuations persist due to the nonconvexity of the problem, the overall trajectory demonstrates that both the stability and final performance of DHDC have been improved. Notably, from the initial transient phase of the first 1,000 epochs to the later stages of training where the optimization reaches a stable regime, DHDC consistently maintains lower loss values and demonstrates superior performance relative to the other algorithms.

\subsection{Optimization of the Stochastic Signal Estimation}\label{subsec:signal}

In this subsection, we present the numerical results for the stochastic signal estimation task previously introduced in \cref{sec:signal}. Since the signal task exhibits sustained oscillations throughout the training phase, we employ the SMA to capture the overall convergence trends more clearly. In this experimental setting, the FGD methods demonstrate superior performance compared to standard GD.

Figure \ref{fig:sigloss} illustrates that while the training loss exhibits significant oscillations during the late optimization phase, it follows a discernible trend. To enhance visual clarity, raw loss values are plotted with reduced opacity, and bold lines represent the SMA. These SMA trends clearly demonstrate that both DHDC and SOE provide superior convergence properties over conventional GD. Notably, DHDC demonstrates the most effective loss minimization, maintaining a visible gap below both SOE and GD.

\begin{figure}[H]
    \centering
    \includegraphics[width=0.67\linewidth]{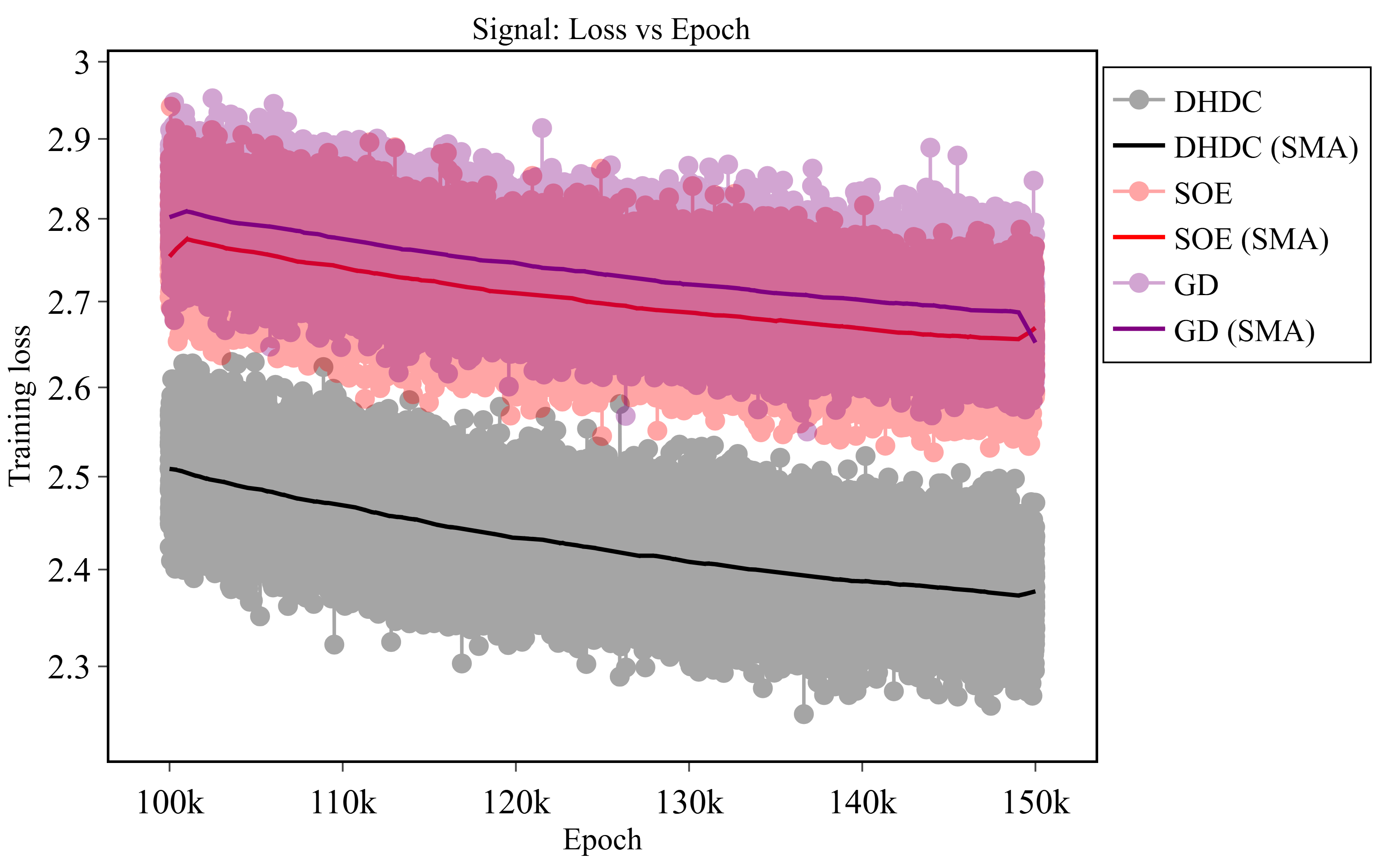}
    \caption{Evolution of training loss during the late optimization phase ($100,000$--$150,000$ epochs).}
    \label{fig:sigloss}
\end{figure}

To further assess the degree of convergence, we present the evolution of gradient norms in Figure \ref{fig:signorm}. These norms reflect the rate of change in the loss at each epoch, where a diminishing norm signifies that the optimization has approached a stable regime with minimal variation. Following the same visualization convention as the loss plots, raw data points are displayed with reduced opacity, and bold lines represent the SMA to highlight the underlying trends. The results demonstrate that both DHDC and SOE converge to significantly deeper local minima than conventional GD, as evidenced by their drastically lower gradient norm levels.

\begin{figure}[H]
    \centering
    \includegraphics[width=0.67\linewidth]{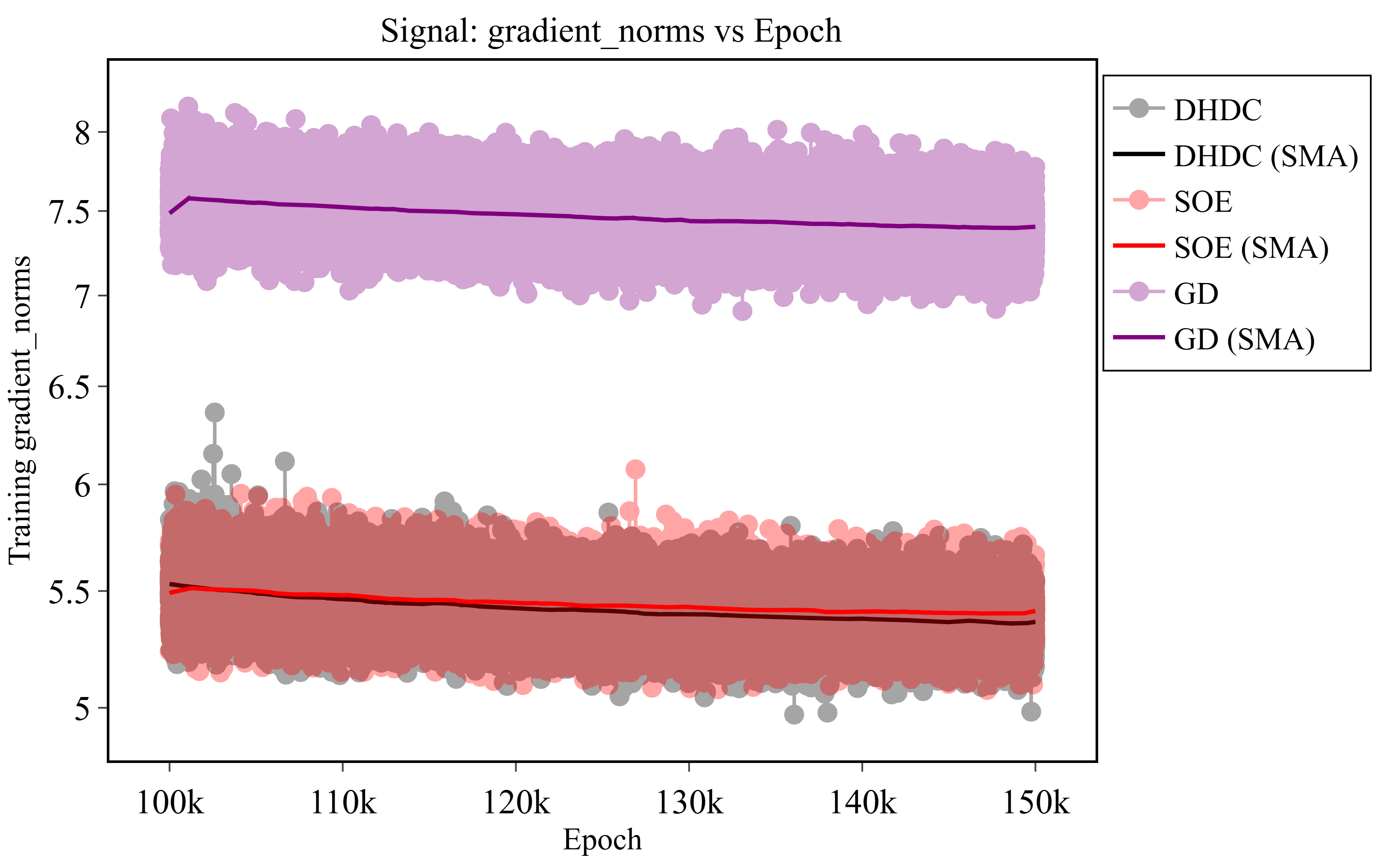}
    \caption{Evolution of training gradient norms during the late optimization phase ($100,000$--$150,000$ epochs).}
    \label{fig:signorm}
\end{figure}

Upon calculating the SMA for each task, DHDC exhibits the most robust performance. Furthermore, the gradient norm analysis reveals that both DHDC and SOE have reached a near-converged state, significantly outpacing GD. 
Conventional GD updates parameters using only the instantaneous gradient.
While this approach is sufficient for short-memory or rapidly mixing signals, it does not explicitly account for long-range correlations present in the data.
On the other hand, because FGD employs a similar power-law memory kernel in its update rule, this signal task provides a principled benchmark for evaluating whether optimizer memory aligned with data memory improves learning efficiency.


\subsection{Optimization of CIFAR-10 Image Classification}\label{subsec:image}

In this subsection, we present the numerical results for the CIFAR-10 image classification task introduced in \cref{sec:image}. Since this task utilizes the widely adopted ResNet-18 architecture, the training process remains relatively stable and does not exhibit the significant oscillations observed in previous examples. However, as ResNet-18 possesses a much more complex structure compared to a standard MLP, it is necessary to analyze the potential for vanishing gradients within the DHDC framework. To this end, we examine the weight trajectories using PCA, as illustrated in Figure~\ref{fig:weightmove}.

\begin{figure}[H]
        \centering
        \includegraphics[width=0.5\linewidth]{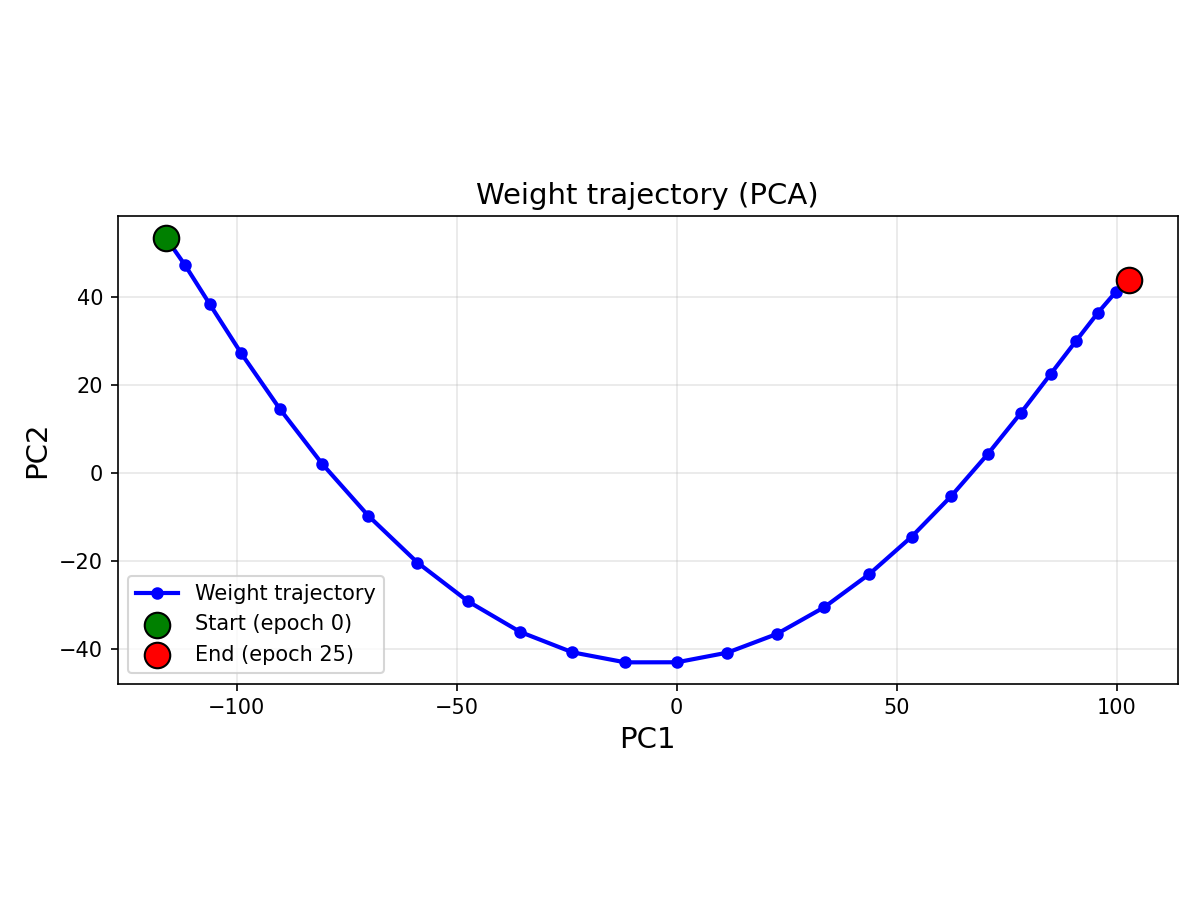}
        \caption{This is a graph plotting the position on a 2-dimensional plane per epoch, where the two dimensions with the maximum variance were selected from the weights of a single image task result.}
        \label{fig:weightmove}
    \end{figure}

Check the whole parameter with the plot is hard in complex system. So Fig \ref{fig:weightmove} compares the geometry of optimization trajectories using only the two principal weight space directions with the largest variance, denoted by $u_1$ and $u_2$. There inner product is zero because $u_2$ is vertical direction of $u_1$

Selecting the two weight moving condition called $PC1$ and $PC2$ is follow these several processes. 
At the end of each epoch $e$, save the full parameter vector:
\begin{equation}
    w_e \;:=\; \mathrm{vec}(\theta_e) \in \mathbb{R}^{D},
    \label{eq:image_vec_param}
\end{equation}
where $\mathrm{vec}(\cdot)$ concatenates all trainable tensors at epoch $e$ into a single vector.

Construct the data matrix which has whole epochs weight vectors:
\begin{equation}
    W \;:=\; [w_0, w_1, \dots, w_E]^\top \in \mathbb{R}^{(E+1)\times D}.
    \label{eq:image_pca_matrix}
\end{equation}

Compute PCA on $W$ and denote the first two principal components by $u_1,u_2 \in \mathbb{R}^{D}$.
Project each epoch parameter:
\begin{equation}
    z_e \;:=\; 
    \begin{bmatrix}
    u_1^\top (w_e-\overline{w})\\
    u_2^\top (w_e-\overline{w})
    \end{bmatrix}
    \in \mathbb{R}^2,
    \qquad
    \overline{w} := \frac{1}{E+1}\sum_{e=0}^{E} w_e.
    \label{eq:image_pca_projection}
\end{equation}

In \eqref{eq:image_pca_projection}, $(w_e-\overline{w})$ is used in magnitude of direction $u_1^\top$ and $u_1^\top$. $z_e$ is function to make blue points of each epoch at Fig \ref{fig:weightmove}.

Figure \ref{fig:weightmove} suggest the $PC1$ and $PC2$ is moving in each epoch. This mean $(w_e-\overline{w})$ has the different size in each epoch e. 
This shows the weight changes when solving the image task using the DHDC method. This confirms that even in large-scale models such as ResNet, training proceeds smoothly without gradient vanishing.

In Table \ref{tab:classification_accur}, while the accuracy increases consistently. In the early stage of training, the methods achieve comparable performance; however, as the number of epochs increases, DHDC and SOE exhibit faster convergence and achieve better performance than GD.

\begin{table}[H]
    \centering
    \resizebox{0.7\linewidth}{!}{
    \begin{tabular}{lccccccc}
        \toprule
        epochs && 1 & 5 & 10 & 15 & 20 & 25  \\
        \midrule
        &GD     & 0.3317 & 0.5926 & 0.6937 & 0.7449 & 0.7673 & \textbf{0.7951} \\
        FGD&SOE & 0.3000 & 0.5808 & 0.6950 & 0.7069 & 0.7091 & \textbf{0.8090}  \\
        &DHDC   & 0.3093 & 0.5922 & 0.6752 & 0.7090 & 0.7253 & \textbf{0.8083}  \\
        \bottomrule
    \end{tabular}%
    }
    \caption{Classification accuracy over training epochs for the CIFAR-10 task.}
    \label{tab:classification_accur}
\end{table}

Classification accuracy is reported at selected training epochs to illustrate the evolution of performance over time. As shown in Table~\ref{tab:classification_accur}, GD exhibits slightly faster initial accuracy gains, while SOE and DHDC achieve higher final accuracy.

\begin{figure}[H]
    \centering
    \includegraphics[width=0.65\linewidth]{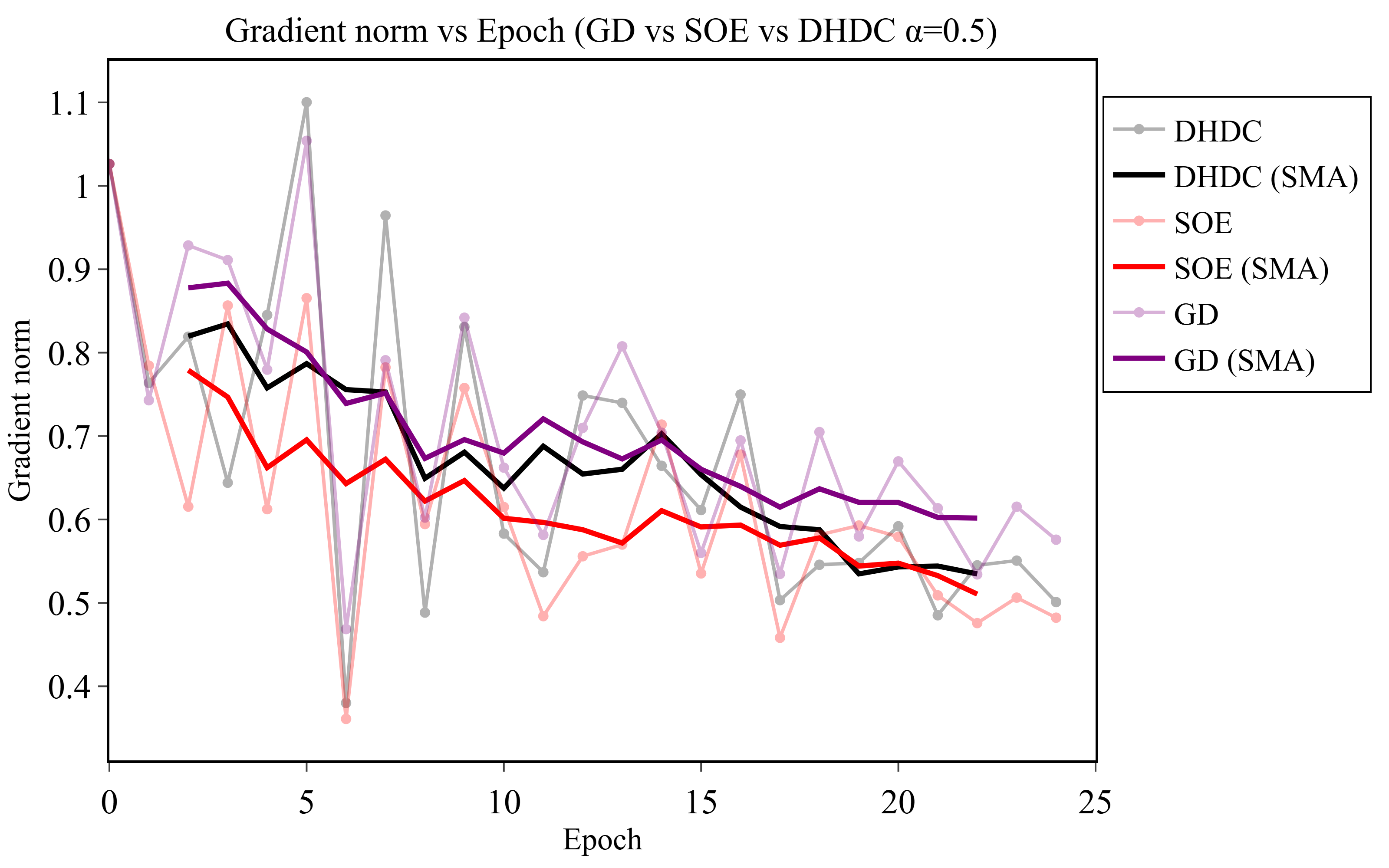}
    \caption{Comparison of gradient norms per epoch for GD, DHDC ($\alpha$=0.5), and SOE. To mitigate the visual impact of training oscillations (faded lines), a 4-epoch moving average (bold lines) is applied. The smoothed trajectories show demonstrate that both proposed fractional-order methods (DHDC and SOE) achieve a faster and more stable reduction in gradient magnitude compared to the standard GD baseline.}
    \label{fig:imagenorm}
\end{figure}

The accuracy trends further elucidate the impact of fractional memory on optimization dynamics. 
As illustrated in Figure \ref{fig:imagenorm}, while standard GD exhibits a significant slowdown during the final 15 epochs, both SOE and DHDC based FGD maintain consistent improvement. 
This behavior is corroborated by the evolution of the gradient norms; specifically, the simple moving averages of the gradient norms demonstrate continued decay for SOE and DHDC, whereas GD suffers from early stagnation. 
These observations suggest that the power-law memory inherent in FGD supports more stable local refinement during the later stages of training. 

While classification accuracy increases with the number of epochs for both classical GD and FGD, we restrict training to 25 epochs in order to emphasize differences in convergence dynamics rather than final asymptotic performance. 
In deep networks with highly nonlinear loss landscapes, lower fractional orders in FGD introduce power-law memory into the update rule, which may lead to smoother and more delicate parameter trajectories compared with conventional integer-order optimization.

\subsection{Summarized results in battery, signal, image tasks}\label{subsec:sum}

Table~\ref{tab:loss} summarizes the performance across various tasks and fractional orders. 
To ensure statistical reliability, each experiment is repeated multiple times: 10 trials for the battery task and 5 trials each for the signal and image tasks. 
For the Taylor-based method in the image task, additional runs (15 trials) were conducted for $\alpha \le 0.4$ to account for instances of training instability. 
Performance is reported as accuracy for the image classification task, while scaled loss values, averaged over multiple trials, are provided for the battery and signal tasks to facilitate consistent comparison. 
Entries marked with ``X'' denote configurations that are either inapplicable, such as GD which does not utilize a fractional order, or those that failed to yield stable training results.

\begin{table}[H]
    \centering
    \resizebox{\linewidth}{!}{
    \begin{tabular}{lccccccccccccc}
        \toprule
        \multirow{2}{*}{FGD} &
        \multicolumn{3}{c}{GD} &
        \multicolumn{3}{c}{taylor} &
        \multicolumn{3}{c}{DHDC} &
        \multicolumn{3}{c}{SOE} \\
        \cmidrule(lr){3-5} \cmidrule(lr){6-8} \cmidrule(lr){9-11} \cmidrule(lr){12-14}
        \multirow{2}{*}{Task} & & battery & signal & image & battery & signal & image & battery & signal & image & battery & signal & image \\
        & & loss($\times10^{-4}$) & loss & accuracy & loss($\times10^{-4}$) & loss & accuracy & loss($\times10^{-4}$) & loss & accuracy & loss($\times10^{-4}$) & loss & accuracy \\
        \midrule
        &0.1   & X & X & X & 5.8789 & 6.4042 & 0.1000 & 1.7409 & 2.4781 & 0.8036 & \textbf{2.1797}& 2.5013 & 0.8030 \\
        &0.2   & X & X & X & 17.8058 & 1.1617 & 0.1000 & 1.7114 & 2.4802 & 0.8051 & 2.2552& \textbf{2.3736} & \textbf{0.8075} \\
        &0.3   & X & X & X & 17.3667 & 0.3424 & 0.1000 & 0.7410 & 2.5727 & 0.8053 & 2.3092& 2.5596 & 0.8037 \\
        &0.4   & X & X & X & 5.6016 & 0.3876 & 0.1000 & 0.7556 & 2.4656 & 0.8043 & 2.3495& 2.6963 & 0.8060 \\
        fractional order&0.5   & X & X & X & \textbf{5.4612} & 7.3059 & 0.1618 & 0.6811 & \textbf{2.3815} & \textbf{0.8083} & 2.3926& 2.6813 & 0.7979 \\
        &0.6   & X & X & X & 5.4960 & 3.8233 & 0.2829 & 0.5412 & 2.3989 & 0.8028 & 2.4200& 2.7162 & 0.7992 \\
        &0.7   & X & X & X & 17.7946 & 0.3011 & 0.3307 & \textbf{0.4623} & 2.5183 & 0.7998 & 2.4642& 2.6156 & 0.8054 \\
        &0.8   & X & X & X & 5.7407 & \textbf{0.1749} & 0.4442 & 0.7882 & 2.3989 & 0.8033 & 2.5383& 2.4096 & 0.8009 \\
        &0.9   & X & X & X & 17.2877 & 0.1783 & \textbf{0.5692} & 0.5087 & 2.6108 & 0.8060 & 2.5888& 2.8555 & 0.8045 \\
        &1.0   & 17.7862 & 2.7102 & 0.7951 &X & X & X & X & X & X & X& X & X \\
        \bottomrule
    \end{tabular}%
    }
    \caption{This is a table showing the final train loss for battery and signal, and the final train accuracy for image, obtained by applying GD, Taylor method, and the fast methods of DC-DHDC and SOE-to the battery, signal, and image models, respectively. The most ideal values are indicated in bold.}
    \label{tab:loss}
\end{table}
 
For each task in Table~\ref{tab:loss}, the fractional order achieving the lowest loss or highest accuracy is highlighted in bold.
Overall, no monotonic relationship is observed between the fractional order $\alpha$ and the performance for DHDC and SOE.
This suggests that the optimal memory depth depends on both the task characteristics and the interaction between the optimization dynamics and the data structure.
In contrast, the Taylor-based method exhibits a more pronounced trend in the image classification task than other case in Table \ref{tab:loss}, where performance generally improves as $\alpha$ approaches $1$.
This behavior is consistent with the fact that the Taylor approximation becomes closer to the classical gradient as $\alpha \to 1$.
Across tasks, DHDC shows strong performance in the battery model, consistently achieving low loss values.
For the image task, DHDC and SOE yield comparable accuracy levels.
In the signal task, the Taylor-based method achieves competitive results at certain fractional orders; however, its performance varies considerably across $\alpha$, including cases with poor convergence.
These observations indicate that while the Taylor-based approach can perform well under specific configurations, its performance is sensitive to the choice of fractional order.
In contrast, DHDC and SOE provide more consistent behavior across a wider range of $\alpha$, suggesting improved robustness with respect to memory configuration.
Overall, the results indicate that fast-memory implementations based on DHDC and SOE offer a more stable trade-off between performance and robustness across different tasks and fractional orders.

Under normal circumstances, the Taylor method is limited to remembering parameters only from the immediately preceding epoch, whereas DHDC and SOE grow as the number of iterations increases. Therefore, among the FGD methods, Taylor should have the fastest speed. 
However, since DHDC and SOE methods' time complexity is increasing when they pass iterations, they exhibit faster speeds in the image task, which train with 4875 iterations. In contrast, while the image task has a smaller number of iterations, the signal task requires 150000 iterations. Because of this high number of iterations, the overhead of parameter accumulation in DHDC and SOE becomes more significant in the signal task. For this reason, Taylor remains the fastest among the FGD methods in the signal task
For the battery task, due to the structure of the loss function, there exists a term that approximates the Caputo fractional derivative using the L1 method and compares it with the result of the equation in which the actual fractional derivative is used. The L1 approximation process consumes resources independently on FGD methods, resulting in similar outcomes across all FGD methods.

In addition, FGD methods in Table \ref{tab:loss} are visualized in the following graphs at Figures \ref{fig:fractional_order} and \ref{fig:fractional_order_ac}.

\begin{figure}[H]
    \centering
    \includegraphics[width=0.45\textwidth]{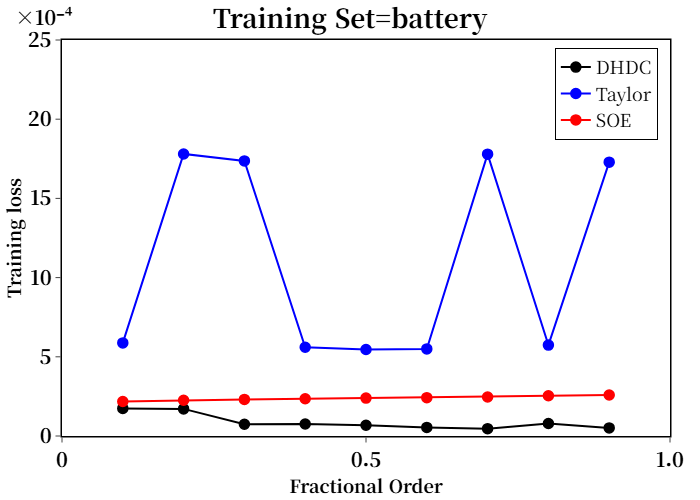}
    \includegraphics[width=0.45\textwidth]{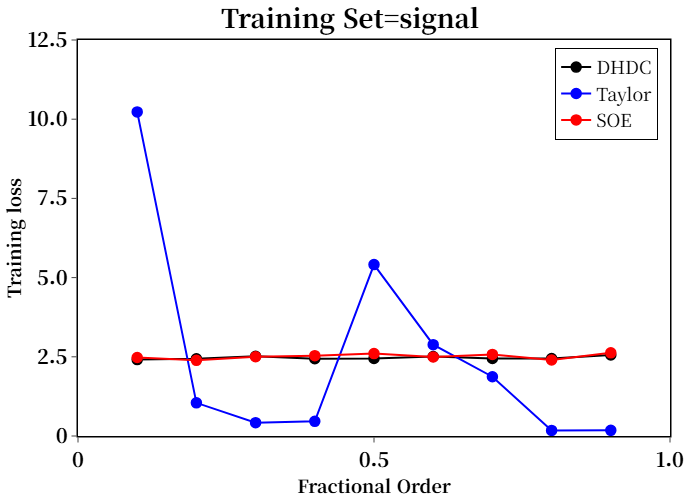}
    \caption{Training loss as functions of the fractional order parameter for each individual training dataset.} 
    \label{fig:fractional_order}
\end{figure}

\begin{figure}[H]
    \centering
    \includegraphics[width=0.45\textwidth]{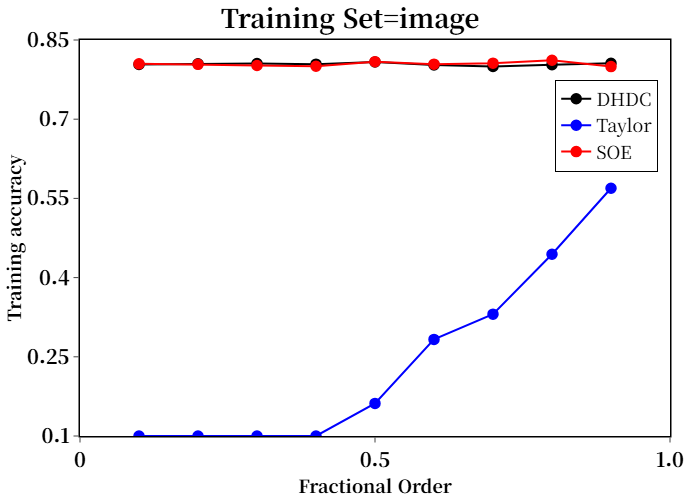}
    \caption{Classification accuracy as functions of the fractional order parameter for each individual training dataset.} 
    \label{fig:fractional_order_ac}
\end{figure}

In \ref{Appendix:E}, box plots are created using the results from all trials for each fractional order.

\section{Conclusion}

This paper focuses on one of the main practical difficulties of Caputo-based fractional gradient descent: the high computational cost caused by the accumulation of long-memory history terms. To address this issue, we reformulated the fractional descent direction as a discrete convolution over past gradients and developed two efficient fast-memory implementations. The first uses a sum-of-exponentials (SOE) approximation of the Caputo kernel, while the second introduces a new dyadic hierarchical discrete convolution (DHDC) framework that compresses the gradient history through multiscale aggregation. Both approaches preserve the characteristic power-law memory structure of fractional optimization while significantly reducing computational and storage costs.

From the theoretical side, the proposed fast-memory schemes were interpreted as controlled perturbations of the ideal Caputo fractional descent direction. Under standard smooth strong convexity assumptions, we showed that the essential convergence properties of the original fractional method are preserved when the approximation error remains sufficiently controlled. This viewpoint helps explain why compressed fractional memory can still produce stable optimization dynamics in practice.

The numerical experiments support these theoretical observations across problems with very different characteristics. On the highly multimodal Rastrigin benchmark problem, the proposed DHDC method exhibited persistent but bounded oscillatory behavior that allowed the optimizer to repeatedly escape shallow local minima while still converging stably near the global minimizer. Additional experiments on battery state estimation, stochastic signal reconstruction, and CIFAR-10 image classification further demonstrated that fractional memory can improve optimization stability and robustness in problems involving long-range dependence, ill-conditioning, and nonconvex optimization landscapes.

Overall, the results suggest that fractional-order optimization can be made computationally practical without losing the essential advantages of long-memory dynamics. More broadly, this work provides a connection between fractional numerical analysis and scalable optimization algorithms, showing that the nonlocal memory effects of Caputo operators can be incorporated into modern learning problems in an efficient and stable manner.

\bibliographystyle{ieeetr}
\bibliography{ref}

\appendix

\section{Numerical Schemes for Caputo Derivative: Left sum and L1 Approximation}
\label{Appendix:A}

\subsection{Constant Approximation scheme}\label{subsec:constant_approximation_scheme}

The simplest approach to numerically approximating the Caputo fractional derivative operator is to approximate all functions within the integrand on each sub-interval $[t_j, t_{j+1}]$ by a constant value evaluated at the left endpoint $t_j$. In this scheme, not only the derivative term $u'(\tau)$ but also the kernel function $(t_n - \tau)^{-\alpha}$, which exhibits a weak singularity, are treated as representative constants for the interval and are extracted from the integral together. As a result, this constant approximation transforms the complex process of kernel integration into a straightforward algebraic weighted sum, offering both intuitive numerical implementation and computational efficiency.

Let the interval \([t_0,t_n]\) be divided into \(n\) uniform sub-intervals
\[
    [t_j,t_{j+1}],
    \qquad
    j=0,1,\ldots,n-1,
\]
with step size \(\Delta t\). Starting from the Caputo fractional derivative,
we have
\begin{equation}
\label{eq:constant_approximation_start}
\begin{aligned}
    {}_0^C D_t^\alpha u(t_n)
    &=
    \frac{1}{\Gamma(1-\alpha)}
    \int_{t_0}^{t_n}
    (t_n-\tau)^{-\alpha}u'(\tau)\,d\tau  \\
    &=
    \frac{1}{\Gamma(1-\alpha)}
    \sum_{j=0}^{n-1}
    \int_{t_j}^{t_{j+1}}
    (t_n-\tau)^{-\alpha}u'(\tau)\,d\tau .
\end{aligned}
\end{equation}

Applying the constant approximation at the left endpoint $t_j$ as described above, the expression is derived as follows:
\begin{equation}
    \label{eq:constant_approximation_derivation}
    \begin{aligned}
        {}_0^C D_t^\alpha u(t_n)
        &\approx
        \frac{1}{\Gamma(1-\alpha)}
        \sum_{j=0}^{n-1}
        (t_n-t_j)^{-\alpha}u'(t_j)
        \int_{t_j}^{t_{j+1}} 1\,d\tau  \\
        &=
        \frac{1}{\Gamma(1-\alpha)}
        \sum_{j=0}^{n-1}
        \bigl((n-j)\Delta t\bigr)^{-\alpha}
        \Delta t\,u'(t_j)  \\
        &=
        \frac{\Delta t^{1-\alpha}}{\Gamma(1-\alpha)}
        \sum_{j=0}^{n-1}
        (n-j)^{-\alpha}u'(t_j).
    \end{aligned}
\end{equation}

Hence, the constant approximation can be written as the weighted sum
\begin{equation}\label{eq:constant_approximation_weighted_sum}
    {}_0^C D_t^\alpha u(t_n)
    \approx
    \sum_{j=0}^{n-1}
    \widehat{w}_{n-j}^{(\alpha)}u'(t_j),
\end{equation}
where the constant-approximation weights are defined by
\begin{equation}\label{eq:constant_approximation_weights}
    \widehat{w}_{k}^{(\alpha)}
    =
    \frac{\Delta t^{1-\alpha}}{\Gamma(1-\alpha)}
    k^{-\alpha},
    \qquad
    k=1,2,\ldots,n.
\end{equation}

\subsection{Left sum Approximation scheme}\label{app:LSA}

The Left sum approximation follows a similar partitioning of the domain $[t_0, t_n]$ into $n$ uniform sub-intervals $[t_j, t_{j+1}]$ with a step size $\Delta t$. In this scheme, as with the previous method, the derivative term $u'(\tau)$ is approximated by its value at the left endpoint of each sub-interval, $u'(t_j)$. However, a crucial distinction lies in the treatment of the weakly singular kernel $(t_n - \tau)^{-\alpha}$; while the derivative is factored out as a constant, the kernel itself remains within the integral and is integrated analytically over each interval. This approach results in the following discrete formulation of the Caputo fractional derivative:

\begin{equation}\label{eq:appena1}
    \begin{aligned}
        {}_0^C D_t^\alpha u(t_n) &= \frac{1}{\Gamma(1-\alpha)} \int_{t_0}^{t_n} (t_n - \tau)^{-\alpha} u'(\tau) d\tau \\
        &= \frac{1}{\Gamma(1-\alpha)} \sum_{j=0}^{n-1} \int_{t_j}^{t_{j+1}} (t_n - \tau)^{-\alpha} u'(\tau) d\tau \\
        &\approx \frac{1}{\Gamma(1-\alpha)} \sum_{j=0}^{n-1} u'(t_j) \int_{t_j}^{t_{j+1}} (t_n - \tau)^{-\alpha} d\tau \\
        &= \frac{1}{\Gamma(1-\alpha)} \sum_{j=0}^{n-1} u'(t_j) \left[ -\frac{1}{1-\alpha} (t_n - \tau)^{1-\alpha} \right]_{t_j}^{t_{j+1}} \\
        &= \frac{\Delta t^{1-\alpha}}{\Gamma(2-\alpha)} \sum_{j=0}^{n-1} u'(t_j) \left\{ (n-j)^{1-\alpha} - (n-j-1)^{1-\alpha} \right\} \\
        &= \sum_{j=0}^{n-1} w_{n-j}^{(\alpha)} u'(t_j).
    \end{aligned}
\end{equation}
In this formulation, the discrete weight coefficients $w_{k}^\alpha$ are defined as:
\begin{equation}
    w_{k}^{(\alpha)} = \frac{\Delta t^{1-\alpha}}{\Gamma(2-\alpha)} \left[ k^{1-\alpha} - (k-1)^{1-\alpha} \right], \quad k=1, 2, \dots, n.
\end{equation}

\subsection{L1 scheme}

The L1 scheme \cite{Oldham1974, langlands2005accuracy} is a robust numerical approximation for evaluating the Caputo fractional derivative, particularly valued for its temporal convergence rate of $\mathcal{O}(\Delta t^{2-\alpha})$. While the previous schemes relied on a constant approximation of the derivative at a single point, the core principle of the L1 scheme involves approximating the function $u(\tau)$ itself using linear Lagrange interpolation over each sub-interval $[t_j, t_{j+1}]$. By partitioning the temporal domain $[0, t_n]$ into $n$ uniform sub-intervals of length $\Delta t$, the derivative $u'(\tau)$ is replaced by the constant slope of the linear interpolant within each interval. This refined treatment of the derivative term, combined with an analytical integration of the power-law kernel, leads to a significant improvement in numerical accuracy. Starting from the definition of the Caputo fractional derivative for $0 < \alpha < 1$:
\begin{equation}
    \prescript{C}{0}{D}_{t}^{\alpha} u(t_n) = \frac{1}{\Gamma(1-\alpha)} \int_{t_0}^{t_n} (t_n-\tau)^{-\alpha}u'(\tau) d\tau.
\end{equation}

By employing a piecewise linear interpolation for $u(\tau)$ over each sub-interval $[t_j, t_{j+1}]$, the derivative $u'(\tau)$ is approximated by the linear interpolation $\frac{u(t_{j+1}) - u(t_j)}{\Delta t}$. Consequently, the integral is discretized as follows:
\begin{equation}
    \prescript{C}{0}{D}_{t}^{\alpha} u(t_n) \approx \frac{1}{\Gamma(1-\alpha)} \sum_{j=0}^{n-1} \frac{u(t_{j+1}) - u(t_j)}{\Delta t} \int_{t_j}^{t_{j+1}} (t_n - \tau)^{-\alpha} d\tau.
\end{equation}

The analytical evaluation of the kernel integral over each sub-interval yields:
\begin{equation}
    \begin{aligned}
        \int_{t_j}^{t_{j+1}} (t_n - \tau)^{-\alpha} d\tau &= \left[ -\frac{(t_n - \tau)^{1-\alpha}}{1-\alpha} \right]_{t_j}^{t_{j+1}} \\
        &= \frac{\Delta t^{1-\alpha}}{1-\alpha} \left( (n-j)^{1-\alpha} - (n-j-1)^{1-\alpha} \right).
    \end{aligned}
\end{equation}

Defining the weight coefficients as $a_{k}^{(\alpha)} = (k+1)^{1-\alpha} - k^{1-\alpha}$, the discrete L1 operator is expressed as:
\begin{equation}
    \mathcal{D}_{L1}^{\alpha} u(t_n) = \frac{\Delta t^{-\alpha}}{\Gamma(2-\alpha)} \sum_{j=0}^{n-1} a_{n-j-1}^{(\alpha)} (u(t_{j+1}) - u(t_j)).
\end{equation}

For practical computational implementation, this summation can be rearranged to group terms by the function values at each time step $u(t_j)$:
\begin{equation}
    \mathcal{D}_{L1}^{\alpha} u(t_n) = \frac{\Delta t^{-\alpha}}{\Gamma(2-\alpha)} \left[ a_0^{(\alpha)} u(t_n) - \sum_{j=1}^{n-1} (a_{n-j-1}^{(\alpha)} - a_{n-j}^{(\alpha)}) u(t_j) - a_{n-1}^{(\alpha)} u(t_0) \right].
\end{equation}

This final form provides an efficient history-dependent update rule, where the influence of past states is governed by the difference between consecutive weight coefficients, reflecting the decaying memory property of the Caputo kernel.

\section{Sum-of-Exponential method \cite{jiang2017fast}}\label{Appendix:B}

The weights for the sum-of-exponentials approximation are determined as follows. First, we choose $M$ nodes $\{\xi_m\}_{m=1}^M$ on a logarithmic scale according to
\[
    \xi_m=\exp\left(\log\left(\frac{1}{t_{\min}}\right)+\frac{m-1}{M-1}\left(\log\left(\frac{1}{t_{\max}}\right)-\log\left(\frac{1}{t_{\min}}\right)\right)\right).
\]
Then, we use $\xi_m$ to make exponential sum to approximate kernel following the next steps.

\begin{enumerate}
    \item \textbf{Fitting weights.} 
    We sample the kernel at logarithmic points $\xi_m, m = 1,\ldots,M$ 
    and solve the least squares problem, 
    the kernel is defined with t to the power of negative alpha
    \[
        \min_{\omega_m\ge 0}\;\sum_\ell 
        \left|\,k(t_\ell)-\sum_{m=1}^M \omega_m e^{-\xi_m t_\ell}\right|^2.
    \]
    where $k(t)=\frac{1}{\Gamma(1-\alpha)}\,t^{-\alpha}, 0 < \alpha < 1$.
    
    \item \textbf{Recursive states.} 
    
    \label{sec:soe_derivation}
    
    To rigorously justify the explicit summation form and the Markovian update rule of the auxiliary state $R_m^{(n)}$, we trace its derivation back to the continuous-time definition of the fractional gradient. The continuous fractional gradient direction $\hat{u}(t)$ is defined by the convolution of the Caputo kernel $k(t)$ and $u(t)$:
    \[
        \hat{u}(t) = \int_0^t k(t-\tau) u(\tau) d\tau.
    \]
    Substituting the SOE approximation $k(t) \approx \sum_{m=1}^M \omega_m e^{-\xi_m t}$ into the convolution integral yields:
    \[
        \hat{u}(t) \approx \int_0^t \left( \sum_{m=1}^M \omega_m e^{-\xi_m (t-\tau)} \right) u(\tau) d\tau.
    \]
    
    Since the integral is a linear operator, we can interchange the summation and integration. This procedure naturally separates the non-local history integral into $M$ independent auxiliary states. Specifically, the approximation of $\hat{u}(t)$ is expressed as:
    \begin{equation}\label{eq:u_hat_approx}
        \hat{u}(t) \approx \sum_{m=1}^M \omega_m \hat{R}_m(t),
    \end{equation}
    where each auxiliary variable $\hat{R}_m(t)$ captures the historical contribution corresponding to a specific exponential mode:
    \begin{equation}\label{eq:Rm_definition}
        \hat{R}_m(t) = \int_0^t e^{-\xi_m (t-\tau)} u(\tau) d\tau.
    \end{equation}
    To enable efficient numerical implementation, we consider the discrete-time evaluation of \eqref{eq:Rm_definition} at $t_{n+1}$. By splitting the integration interval, we obtain the following recurrence relation:
    \begin{align}
        \hat{R}_m^{(n+1)} &= \int_{t_0}^{t_{n+1}} e^{-\xi_m(t_{n+1} - \tau)} u(\tau) d\tau \nonumber \\
        &= \int_{t_n}^{t_{n+1}} e^{-\xi_m(t_{n+1} - \tau)} u(\tau) d\tau + e^{-\xi_m \Delta t} \int_{t_0}^{t_n} e^{-\xi_m(t_n - \tau)} u(\tau) d\tau \nonumber \\
        &= \int_{t_n}^{t_{n+1}} e^{-\xi_m(t_{n+1} - \tau)} u(\tau) d\tau + e^{-\xi_m \Delta t} \hat{R}_m^{(n)}, \label{eq:soe_recursive_step}
    \end{align}
    where $\hat{R}_m^{(0)} = 0$. By assuming that $u(\tau) \approx u(t_n)$ remains constant over the small interval $[t_n, t_{n+1}]$, the non-local convolution is transformed into a local inductive structure. Under this approximation, the first integral in \eqref{eq:soe_recursive_step} evaluates to:
    \begin{equation}
        \hat{R}_m^{(n+1)} \approx \frac{1-e^{-\xi_m \Delta t}}{\xi_m} u(t_n) + e^{-\xi_m \Delta t} \hat{R}_m^{(n)}.
    \end{equation}
    Let $C_m = \frac{1-e^{-\xi_m \Delta t}}{\xi_m}$. To achieve a simplified recursive form, we define the computational auxiliary state $R_m^{(n+1)}$ such that $\hat{R}_m^{(n+1)} = C_m R_m^{(n+1)}$. Substituting this into the above approximation yields:
    \begin{equation}
        R_m^{(n+1)} = e^{-\xi_m \Delta t} R_m^{(n)} + u(t_n), \quad R_m^{(0)} = 0.
    \end{equation}
    Finally, the update direction $\hat{u}(t_{n+1})$ is computed by consolidating the scaling factors into the modified weights $\hat{\omega}_m = \omega_m C_m$:
    \begin{equation}
        \hat{u}(t_{n+1}) \approx \sum_{m=1}^M \hat{\omega}_m R_m^{(n+1)}.
    \end{equation}
    
\end{enumerate}

Unlike conventional numerical schemes that rely on piecewise approximations over partitioned intervals, the SOE method approximates the singular kernel using a weighted sum of exponential functions. In this framework, each $R_m^{(n)}$ serves as an auxiliary state that captures the history of $u'(\tau)$ through an exponential filter with a decay rate $\xi_m$. By leveraging the recursive structure inherent in these exponential terms, the SOE mechanism transforms the entire convolution history into a finite set of $M$ evolution states. Consequently, this approach maintains numerical accuracy comparable to standard methods while effectively reducing the per-iteration complexity from $\mathcal{O}(n)$ to $\mathcal{O}(M)$.

\section{Boxplot Analysis of Performance Distributions \ref{subsec:sum}}\label{Appendix:E}

For the battery, signal, and image tasks, final results are obtained using each of the FGD methods, namely the Taylor-based approach, DHDC, and SOE. The results from multiple independent runs were compiled and visualized via boxplots to represent the distributional characteristics, including the median, interquartile range, and outliers across the various configurations.

\begin{figure}[H]
    \centering
    \subfloat[DHDC]{%
    \includegraphics[width=0.45\textwidth]{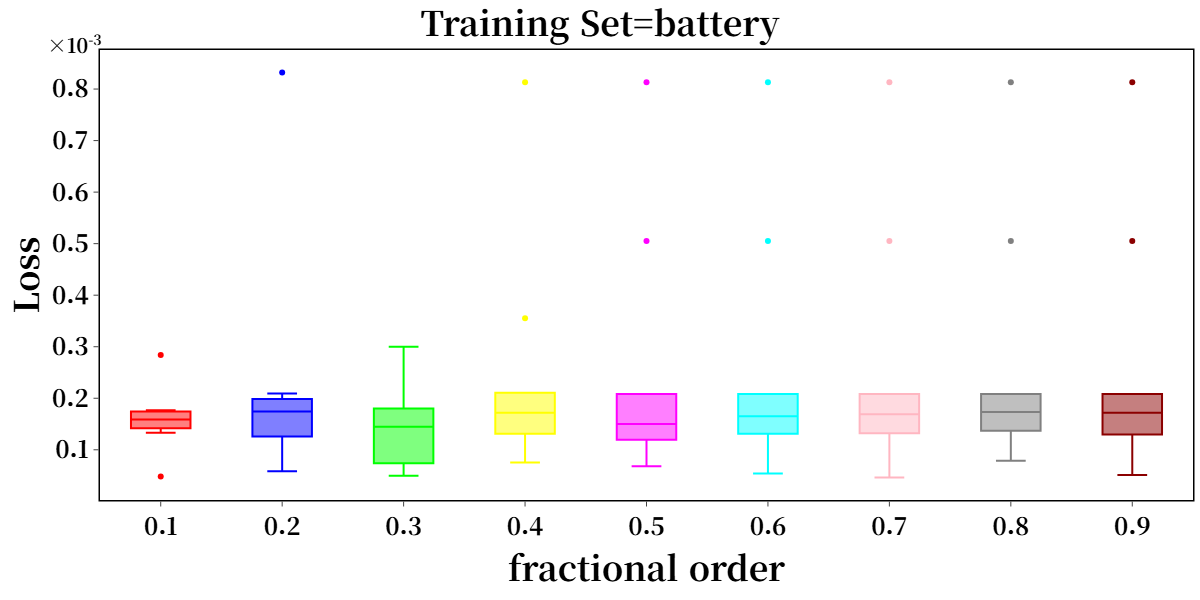}}
    \hspace{1cm}
    \subfloat[taylor]{%
    \includegraphics[width=0.45\textwidth]{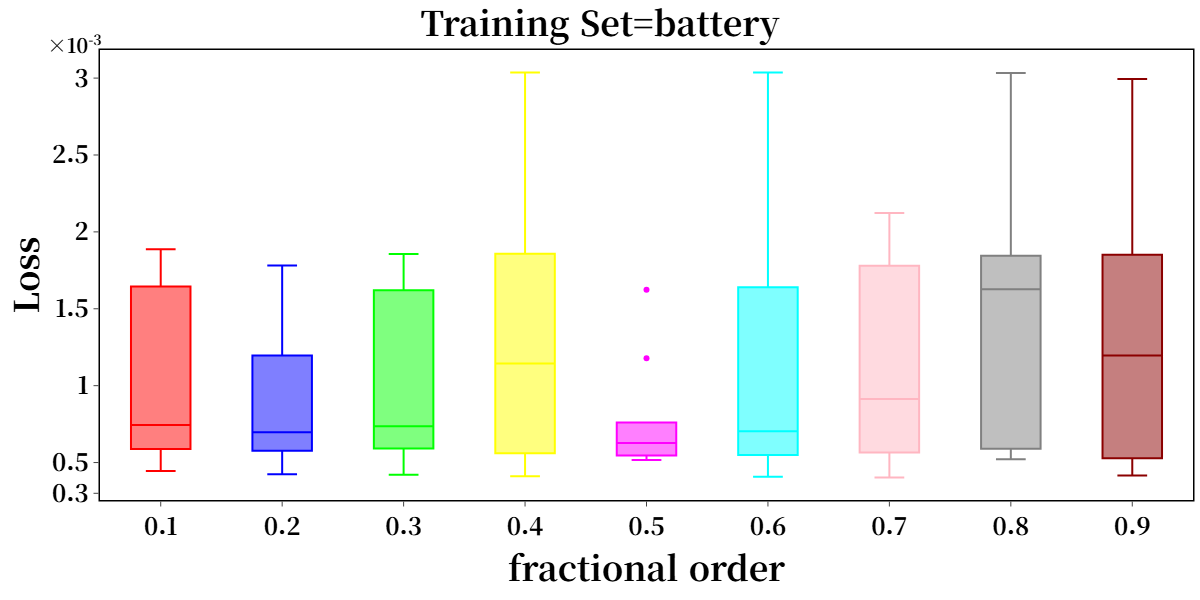}}
    
    \subfloat[SOE]{%
    \includegraphics[width=0.45\textwidth]{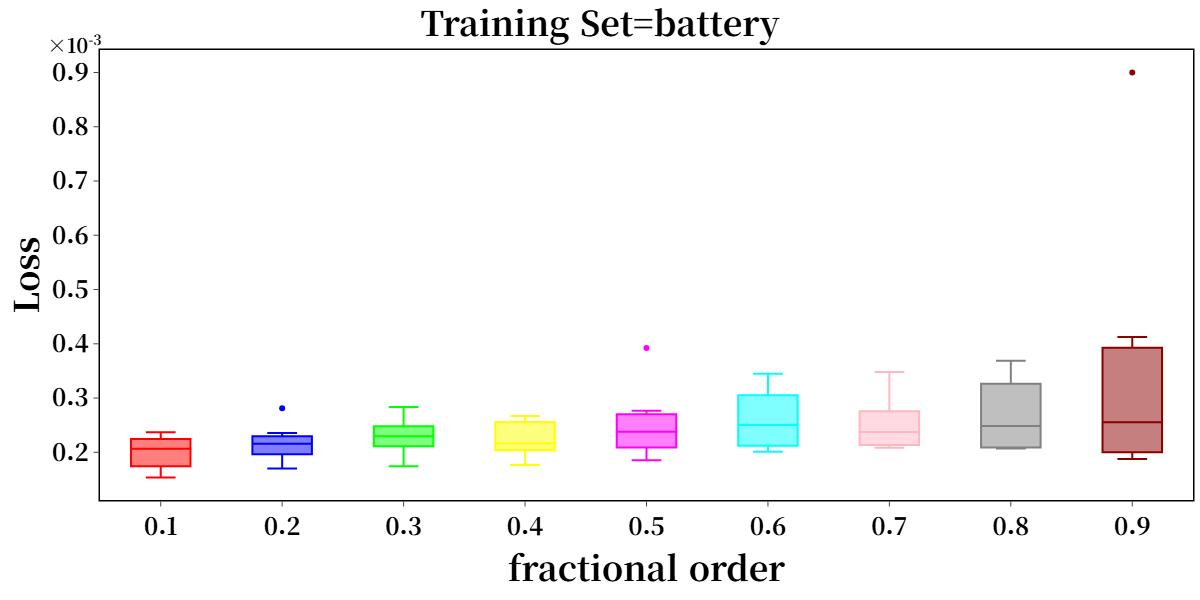}}
\end{figure}

\begin{figure}[H]
    \centering
    \subfloat[DHDC]{%
    \includegraphics[width=0.45\textwidth]{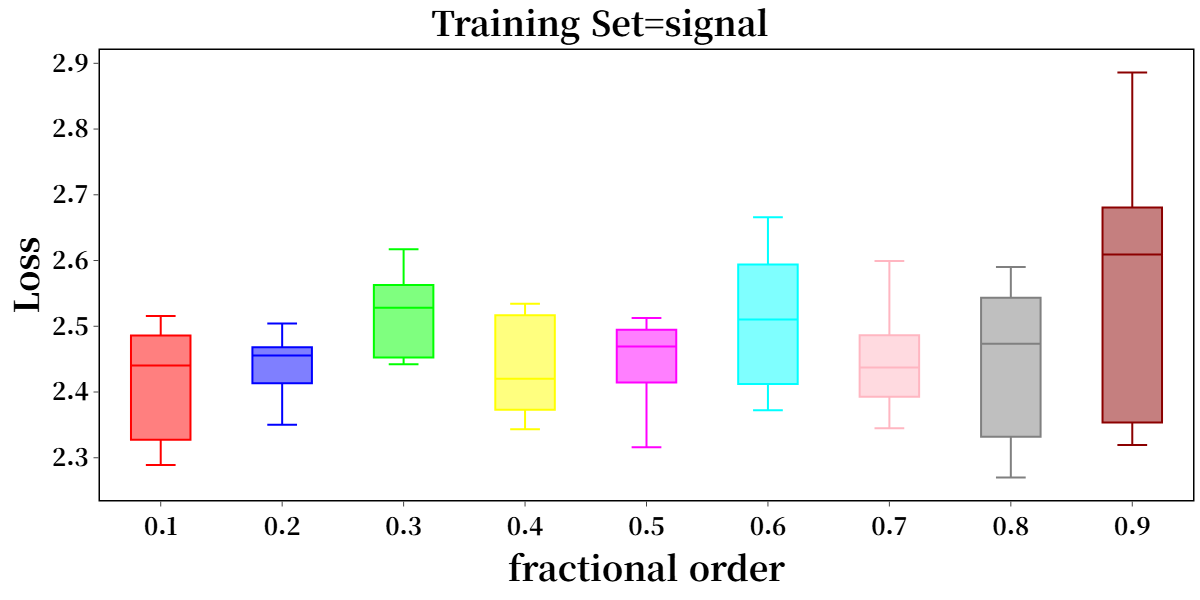}}
    \hspace{1cm}
    \subfloat[taylor]{%
    \includegraphics[width=0.45\textwidth]{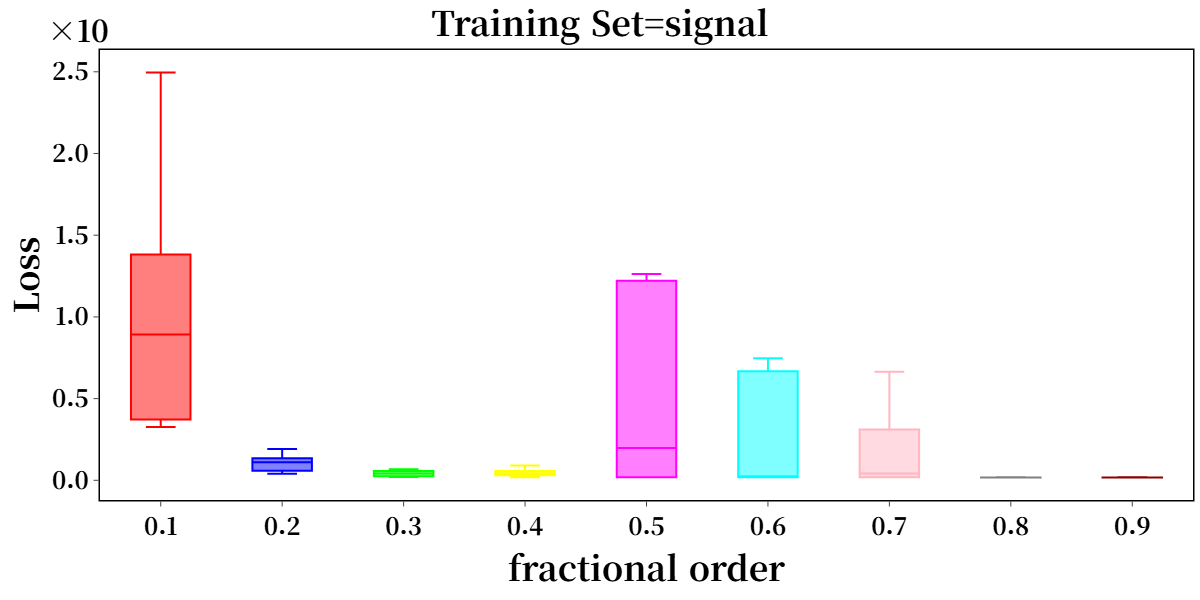}}
    
    \subfloat[SOE]{%
    \includegraphics[width=0.45\textwidth]{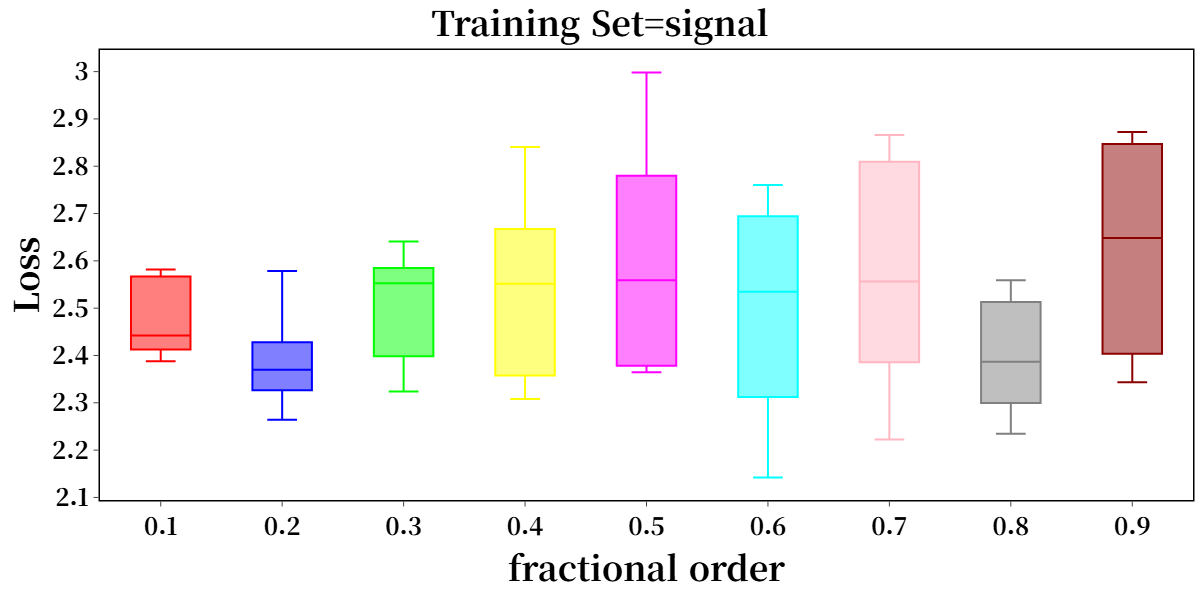}}
\end{figure}

\begin{figure}[H]
    \centering
    \subfloat[DHDC]{%
    \includegraphics[width=0.45\textwidth]{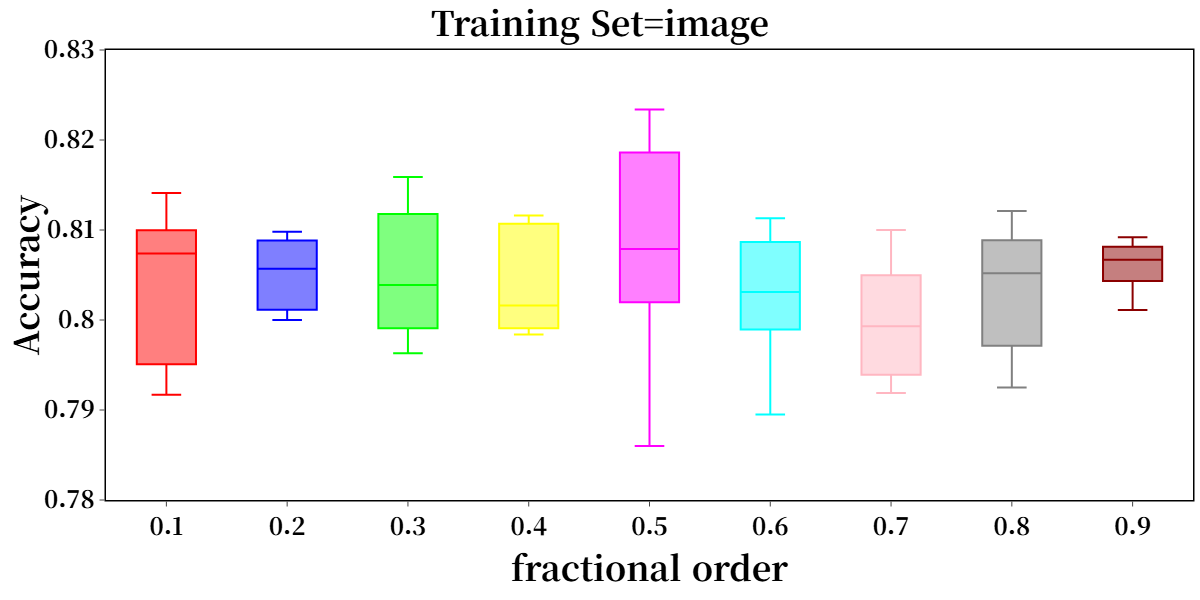}}
    \hspace{1cm}
    \subfloat[taylor]{%
    \includegraphics[width=0.45\textwidth]{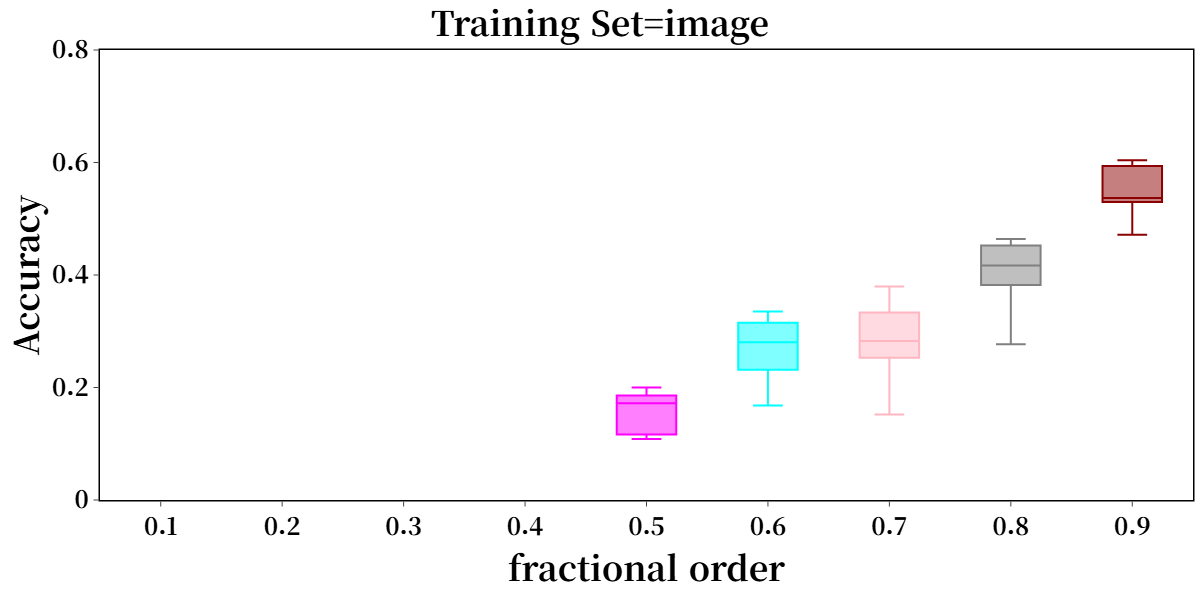}}
    
    \subfloat[SOE]{%
    \includegraphics[width=0.45\textwidth]{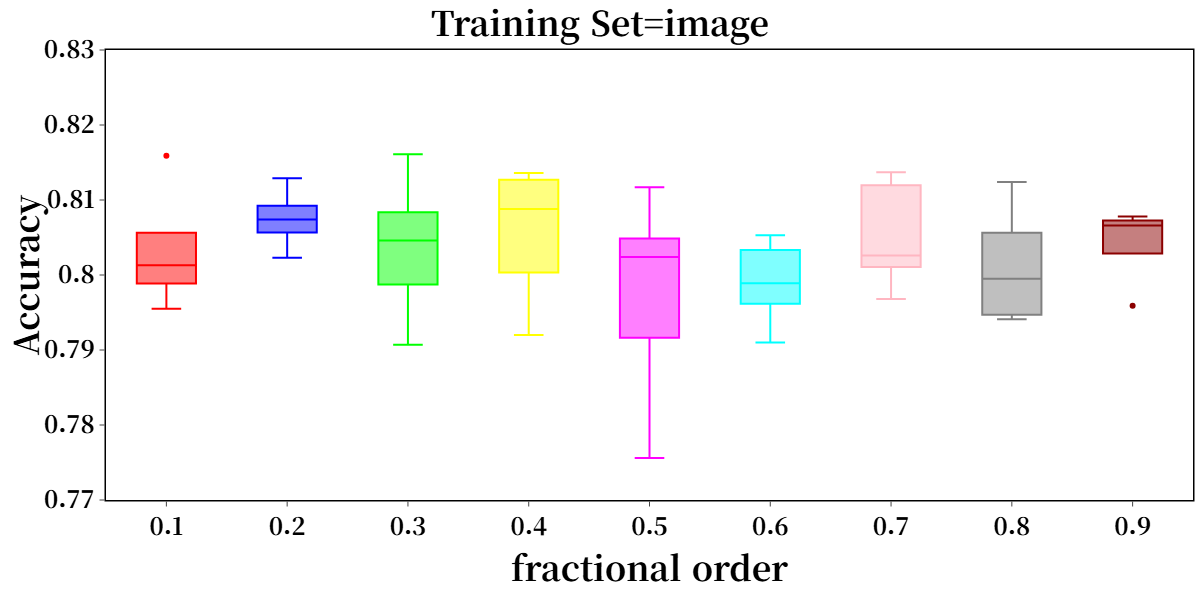}}
    \caption{Boxplot distributions of final training accuracy and training loss across fractional orders ranging from 0.1 to 0.9, evaluated over multiple independent runs, illustrating the sensitivity of convergence behavior to the choice of fractional order and the variability across repeated experiments.}
    \label{fig:fractional_order_box}
\end{figure}

Figure \ref{fig:fractional_order_box} presents box plots generated by repeating each FGD method on their respective tasks at least five times with different random seeds. The first row shows the boxplots of the final loss values for the battery task. The second row corresponds to the signal task, and the third row presents the results for the image task.

In the image task, the Taylor method yields no valid results for fractional orders ranging from 0.1 to 0.4, as the loss diverges during training, resulting in an accuracy of 0.1, which corresponds to the theoretical lower bound for the ten-class CIFAR-10 dataset.

\section{Computational Complexity and Execution Times for Real-World Problems} \label{Appendix:timecom}

In this appendix, we empirically validate the computational complexity of the discussed FGD methods. Let $N$ denote the total number of epochs and $M$ represent the number of SOE terms. As detailed in Section~\ref{sec:SOE_and_DHDC}, the theoretical complexities for DC, SOE, and DHDC are $O(N^2)$, $O(MN)$, and $O(N\log N)$, respectively. To verify these theoretical properties, we evaluate the alignment between experimental measurements and their corresponding theoretical predictions with MSEloss. To calculate the MSEloss, We can obtain a mathematical basis of time complexity.

Define the experimental measurement at i epoch to $y_{ex}^i$ and define the theoretical prediction at i epoch to $y_{pred}^i$ 
\begin{equation}
    MSEloss = \frac{1}{N-min}\sum^N_{i=min}(y_{ex}^i - y_{pred}^i)^2
\end{equation}
initial point of epoch is defined with $min$.

Graph legends $f(x) = y$ means $O(f(x))$ when $y_{ex}^i$ and $y_{pred}^i$ has acceptable MSEloss. For $y_{ex}^i$ is satisfied this condition. $y_{ex}^i$ is normalized with values of $y_{pred}^i$. 

$x=y$ normalized reference in fig \ref{fig:OSOE1}, \ref{fig:OSOE2} is defined at \eqref{eq:n}.

$x^2 = y$ normalized reference in fig \ref{fig:ODC} and $xlog(x) = y$ normalized reference in fig \ref{fig:ODHDC} are defined at \eqref{eq:n2} and \eqref{eq:nlog} each. max(y) and min(y) means the maximum and minimum value for whole epoch. $i$ means epoch used at graphs horizontal axis.

\begin{equation}
\label{eq:n}
    y_{ex}^i = min(y_{pred}) + x_{norm}^i \times(max(y_{pred}) - min(y_{pred})) \quad x_{norm}^i = \frac{(i - min)}{N-min}
\end{equation}

\begin{equation}
\label{eq:n2}
    y_{ex}^i = ratio \times i ^2 \quad ratio = \frac{min(y_{pred})}{min^2}
\end{equation}

\begin{equation}
\label{eq:nlog}
    y_{ex}^i = ratio \times ilog_2(i) \quad ratio = \frac{min(y_{pred})}{min \times log_2{(min)}}
\end{equation}

Whole time complexity graphs are using same task. Used task is Rosenbrock function which defined at \eqref{eq:rosen}. Use the same initial condition at \eqref{ini}.

\begin{figure}[H]
    \centering
    \includegraphics[width=0.45\linewidth]{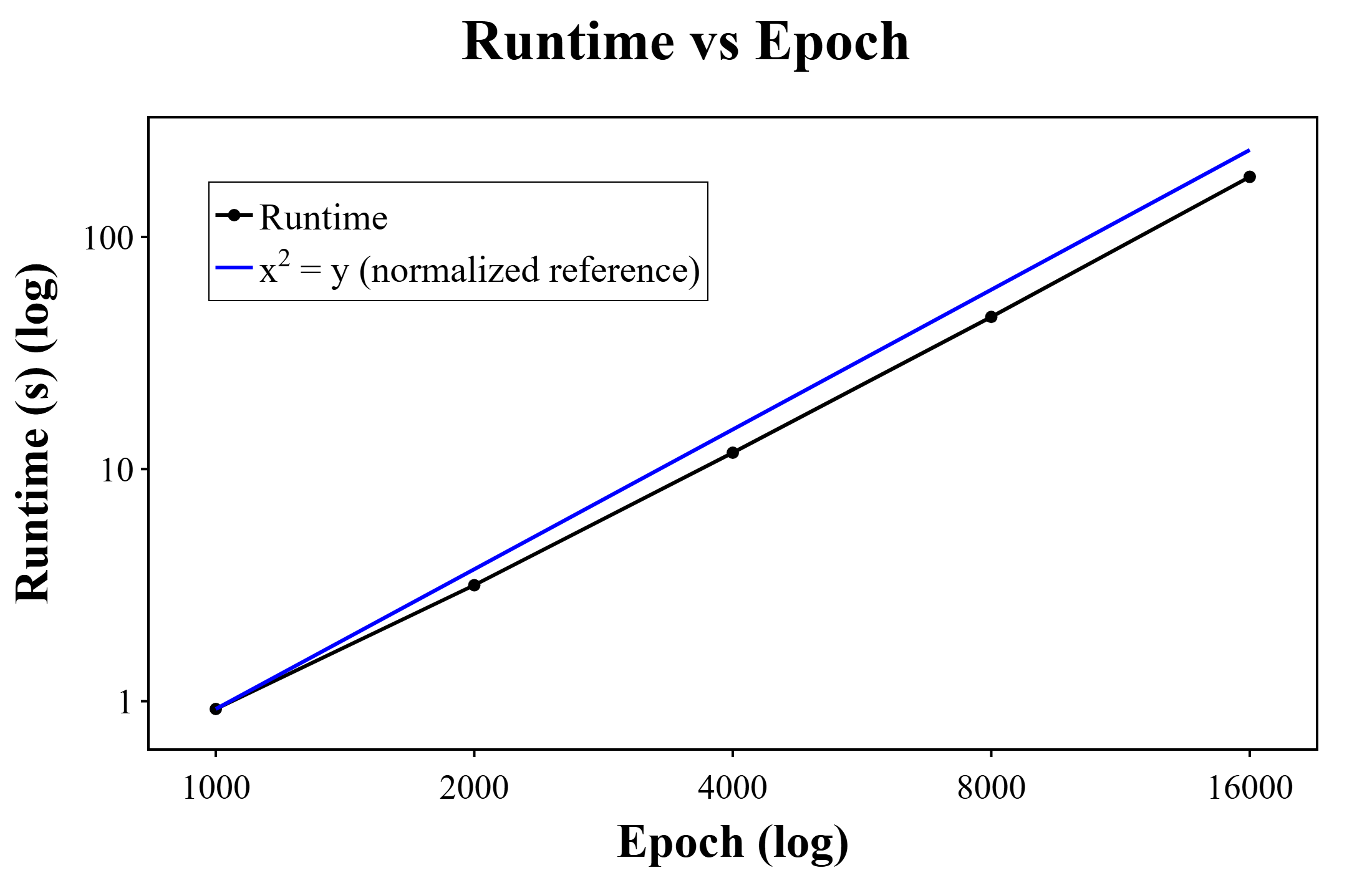}
    \caption{Total elapsed time per epoch for the DC method with log-log plot. The blue curve represents the trajectory that would result if the method followed the indicated complexity.}
    \label{fig:ODC}
\end{figure}



Figure~\ref{fig:ODC} presents the elapsed time of the DC method as a function of $N$, compared against a theoretical quadratic curve. The discrepancy between the empirical measurements and the $O(N^2)$ model, expressed as an MSE loss, is $133$. While this value may seem high, it is primarily due to the large scale of the $y$-axis. Furthermore, a log--log regression analysis yields a scaling of $O(N^{1.941})$, confirming that the empirical time complexity of DC follows $O(N^2)$.

\begin{figure}[H]
    \centering
    \includegraphics[width=0.45\linewidth]{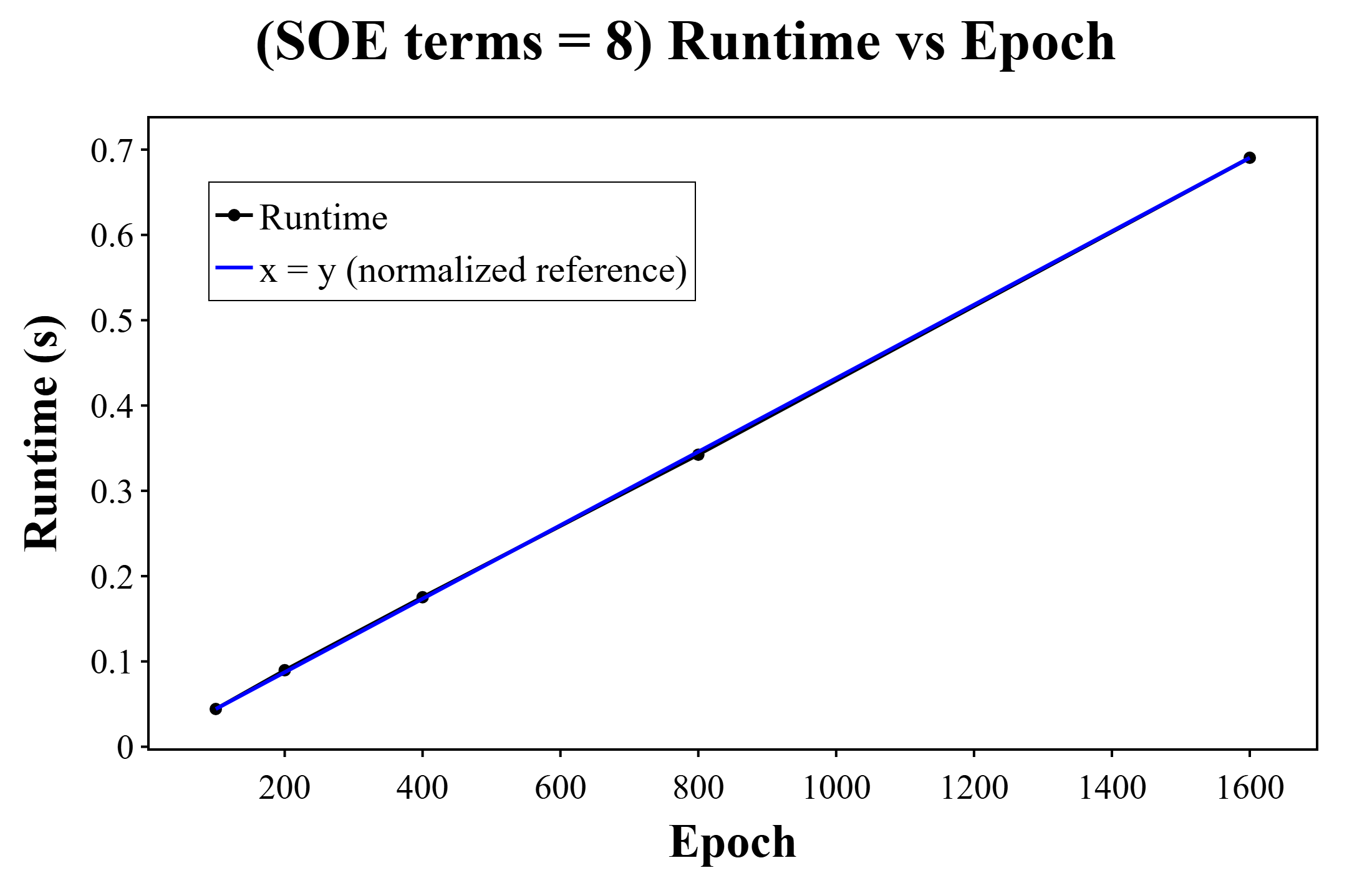}
    \caption{Total elapsed time per epoch for the SOE versus . The blue curve represents the trajectory that would result if the method followed the indicated complexity. }
    \label{fig:OSOE1}
\end{figure}
Figure~\ref{fig:OSOE1} shows the elapsed time versus $N$ when the number of SOE terms is fixed 8. The discrepancy between the theoretical and measured values, expressed as MSE loss, is $1.42\times10^{-7}$. This indicates runtime is in direct proportion to epochs named N.

\begin{figure}[H]
    \centering
    \includegraphics[width=0.45\linewidth]{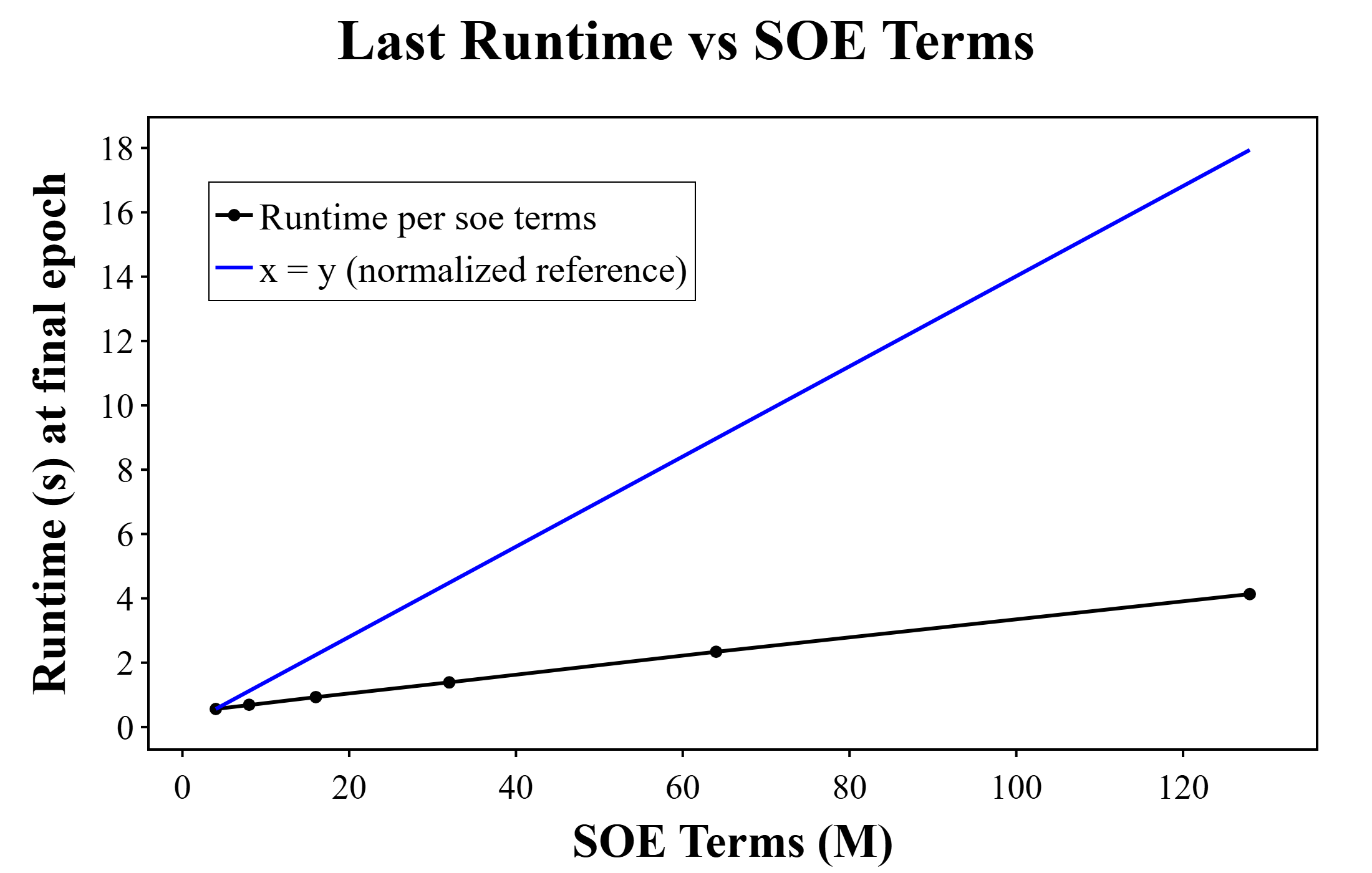}
    \caption{Total elapsed time for the SOE versus SOE terms. The blue curve represents the trajectory that would result if the method followed the indicated complexity.}
    \label{fig:OSOE2}
\end{figure}
Figure~\ref{fig:OSOE2} shows the elapsed time versus $M$. The discrepancy between the theoretical and measured values, expressed as MSE loss, is $9.62$. This indicates runtime is in direct proportion to SOE Terms.

From \ref{fig:OSOE1}, \ref{fig:OSOE2}, SOE exhibits a linear increase in cost with respect to both $M$ and $N$, which indicates that the time complexity of SOE is $O(MN)$. 

\begin{figure}[H]
    \centering
    \includegraphics[width=0.45\linewidth]{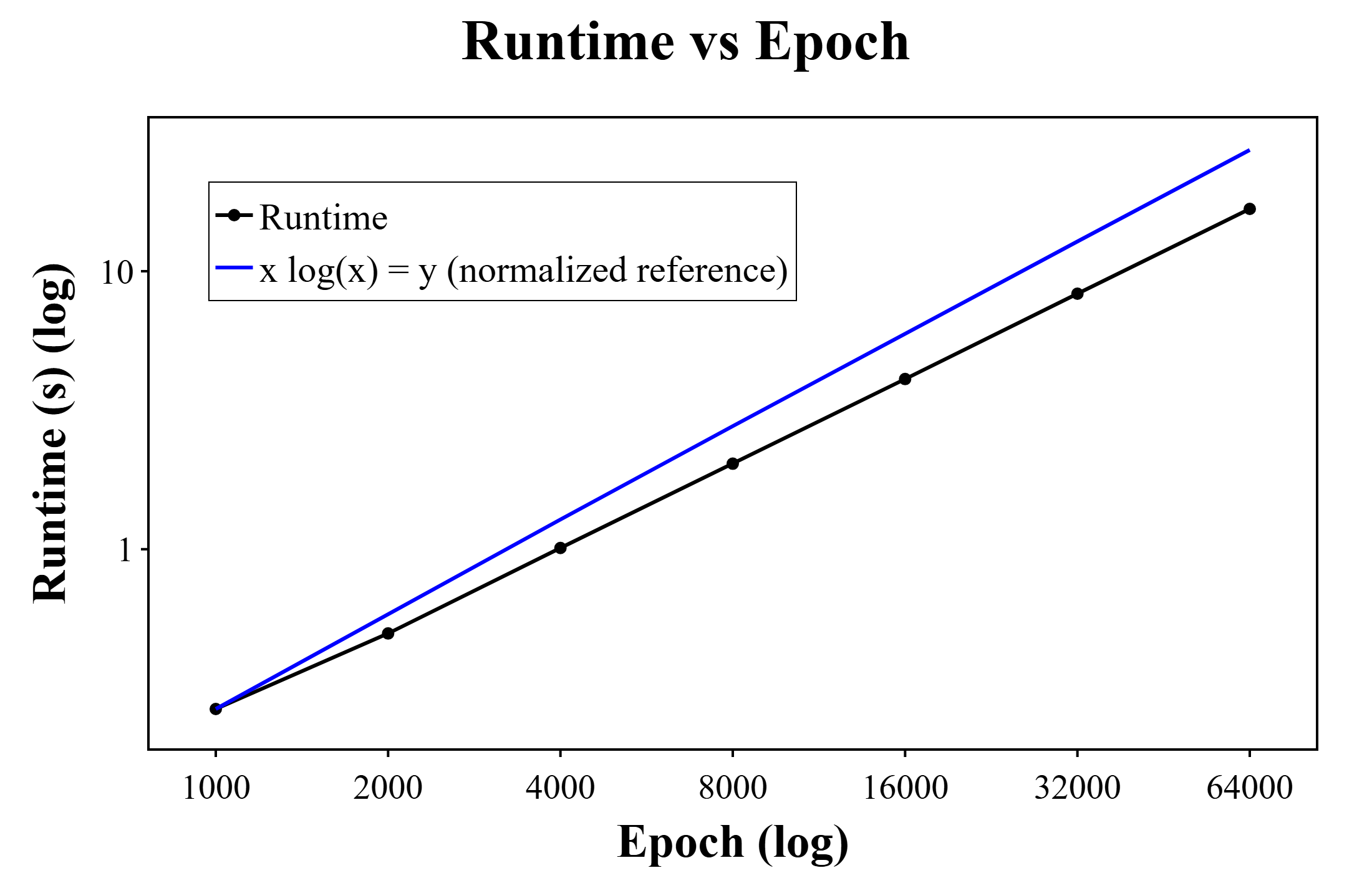}
    \caption{Total elapsed time per epoch for the DHDC method with log-log plot. The blue curve represents the trajectory that would result if the method followed the indicated complexity.}
    \label{fig:ODHDC}
\end{figure}

Figure~\ref{fig:ODHDC} is the log-log plot of cpu time at each epochs. The discrepancy between the theoretical and measured values, expressed as MSEloss, is $3.57$. MSEloss is acceptably low and Runtime graph is under the normalized reference. Big-O time complexity is meaning the maximum value when normalized. This indicates that the computational complexity of DHDC is $O(N\log N)$. 

\begin{table}[H]
    \centering
    \resizebox{\linewidth}{!}{
    \begin{tabular}{lcccccccccc}
    \toprule
    \multirow{2}{*}{FGD} & \multicolumn{3}{c}{DC} & \multicolumn{3}{c}{DHDC} & \multicolumn{3}{c}{SOE} \\
    \cmidrule(lr){3-5} \cmidrule(lr){6-8} \cmidrule(lr){9-11} 
    \multirow{2}{*}{Task} & & battery  & signal   & image    & battery  & signal   & image    & battery  & signal   & image    \\
                          & & times(s) & times(s) & times(s) & times(s) & times(s) & times(s) & times(s) & times(s) & times(s) \\
    \midrule
    CPU times&   & 116 & 26109 & - &  88 & 5349 & 1253 & 89& 5164 & 1314 \\
    \bottomrule
    \end{tabular}%
    }
    \caption{Computational Efficiency Comparison of DC and fast method of DC-DHDC and SOE-}
    \label{tab:CPU_time}
\end{table}

Table~\ref{tab:CPU_time} compares the computational efficiency of DC, DHDC, and SOE across the three real-world tasks. Here, CPU time is defined as the total duration required to complete the full training cycles (epochs and iterations), excluding the initial code setup.

The results clearly demonstrate the superior efficiency of the proposed fast algorithms; both DHDC and SOE provide significant speed improvements over the standard DC method. Notably, the signal task exhibits the longest execution times due to its large number of epochs, even exceeding the CPU time of the image task which employs the more architecturally complex ResNet-18.

The entry for DC in the image task is left blank (indicated by a dash) because its execution is deemed computationally intractable for practical evaluation. As the resource requirements for DC increase substantially with the number of iterations, its expected computation time is instead estimated via regression analysis, with further details provided in Appendix~\ref{Appendix: C}.

\begin{table}[H]
    \centering
    \begin{tabular}{lcccc}
        \toprule
        \cmidrule(lr){2-5}
         & FGD method & DC & DHDC & SOE  \\ 
        \midrule
         & Vram(mb) & >32000  & 6962 & 5529 \\ 
        \bottomrule
    \end{tabular}
    \caption{Memory Efficiency and Practical Feasibility}
    \label{tab:optimizers}
\end{table}

Table~\ref{tab:optimizers} summarizes the peak GPU memory (VRAM) consumption for the CIFAR-10 image task. Since this task exhibits the most substantial increase in resource usage among the evaluated scenarios, it was selected as the benchmark for measuring maximum VRAM usage with the fractional order set to 0.5. As shown in the table, the baseline DC method exceeds 32,000 MB, surpassing typical hardware capacities for this task. In contrast, the proposed SOE and DHDC approximations significantly reduce the memory footprint to 5,529 MB and 6,962 MB, respectively, demonstrating their superior scalability for large-scale optimization.

\subsection{Computational Cost Estimation for Discrete Convolution in Image Tasks \ref{Appendix:timecom}}\label{Appendix: C}

The standard FGD using the Caputo fractional derivative exhibits an $O(n^2)$ computational complexity. As the number of iterations $n$ increases, the non-local dependency of the fractional operator leads to a significant escalation in the computational load per step.

To quantify this overhead, a linear regression analysis was conducted based on the initial training phase. Observations indicated that the execution time per iteration increases linearly with the iteration index $i$. For a total of $4,875$ iterations, the cumulative computation time $S$ is modeled as follows:
\begin{equation}
    S = \sum_{i=0}^{4874} (1.08415 i + 16.9310).
\end{equation}
The regression analysis yielded a slope of $1.08415$ and an intercept of $16.9310$, with a standard error of $0.0274$, confirming the reliability of the estimation.

In a configuration with a batch size of 256, these 4,875 iterations would require an estimated 12,962,647 seconds. However, due to GPU VRAM capacity constraints, the actual batch size for the standard FGD must be set to 1 to accommodate the memory requirements. Under this constraint, the total number of iterations per epoch increases to 49,920. Applying the regression model to this scenario results in a projected computation time exceeding 26,772 years. This result demonstrates that standard FGD is computationally infeasible for this optimization task, necessitating the implementation of the SOE and DHDC schemes proposed in this work.

\section{Algorithm}\label{Appendix:Algorithm}
This section presents the pseudo-codes for the methodologies introduced in the preceding sections to provide a clear description of their algorithmic implementation.

\begin{algorithm}[H]
    \caption{SOE memory update at step $n$ (time $t_n=n\Delta t$)}
    \label{alg:soe-update}
    \begin{algorithmic}[1]
        \Require gradients $\{\bm{g}_j\}_{j\ge 0}$ streamed online; SOE parameters $\{(\omega_m,\xi_m)\}_{m=1}^M$ with $\omega_m>0$, $\xi_m>0$; step size $\Delta t>0$
        \Statex \textbf{State:} auxiliary modes $\{\bm{R}_m\}_{m=1}^M$ (each in $\mathbb{R}^d$), initialized by $\bm{R}_m \gets \bm{0}$
        \Statex \textbf{Output:} fractional (history-aware) gradient $\widehat{\bm{g}}_n$
        
        \Statex \textbf{Precompute once:} $\rho_m = \exp(-\xi_m \Delta t)$ for $m=1,\dots,M$
        \State \textbf{Input at step $n$:} current gradient $\bm{g}_n=\nabla L(\bm{\theta}_n)$
        
        \For{$m=1$ \textbf{to} $M$}
            \State $\bm{R}_m \gets \rho_m\,\bm{R}_m + \bm{g}_n$
        \EndFor
        
        \State $\widehat{\bm{g}}_n \gets \bm{0}$
        \For{$m=1$ \textbf{to} $M$}
            \State $\widehat{\bm{g}}_n \gets \widehat{\bm{g}}_n + \omega_m\,\bm{R}_m$
        \EndFor
        
        \Statex \textbf{Note:} The update includes the current gradient $\bm{g}_n$, 
        so that $\widehat{\bm{g}}_n$ corresponds to a convolution of the form 
        $\sum_{j=0}^{n} w_{n-j}\bm{g}_j$.
        
        \State \Return $\widehat{\bm{g}}_n$
    \end{algorithmic}
\end{algorithm}

Algorithm \ref{alg:soe-update} is the pseudocode of the SOE introduced in section \ref{subsec:SOE}

\begin{algorithm}[H]
\caption{DHDC update at step $n$ (time $t_n=n\Delta t$)}
\label{alg:DHDC-update}
\begin{algorithmic}[1]
\State \textbf{Input:} gradient $\bm{g}_n$, step size $\Delta t$, order $\alpha$, tolerance $\varepsilon_t$
\State \textbf{State:} bins $\{(\bm{S}_b^{(n)}, C_b^{(n)})\}_{b=0}^{B-1}$
\State \textbf{Insert (youngest bin):} $\bm{S}_0 \gets \bm{S}_0 + \bm{g}_n$, \quad $C_0 \gets C_0 + 1$

\For{$b=0$ \textbf{to} $B-2$}
  \State $\mathrm{cap} \gets 2^b$
  \If{$C_b > \mathrm{cap}$}
     \State $\mathrm{half} \gets \max(1,\lfloor C_b/2 \rfloor)$,\quad $\varphi \gets \mathrm{half}/C_b$
     \State $\bm{M} \gets \varphi\,\bm{S}_b$
     \State $\bm{S}_b \gets \bm{S}_b - \bm{M}$,\quad $C_b \gets C_b - \mathrm{half}$
     \State $\bm{S}_{b+1} \gets \bm{S}_{b+1} + \bm{M}$,\quad $C_{b+1} \gets C_{b+1} + \mathrm{half}$
  \EndIf
\EndFor

\State $\widehat{\bm{g}}_n \gets \bm{0}$
\For{$b=0$ \textbf{to} $B-1$}
  \If{$C_b = 0$}
     \State \textbf{continue}
  \Else
    \State $t_{\min} \gets \max\!\big(\varepsilon_t,\; t_n - 2^{b+1}\Delta t\big)$
    \State $t_{\max} \gets \max\!\big(\varepsilon_t,\; t_n - 2^{b}\Delta t\big)$
    \State $\delta \gets t_{\max}-t_{\min}$

    \If{$\delta \le \varepsilon_t$}
       \State $t_{\mathrm{mid}} \gets \max\!\big(\varepsilon_t,\tfrac12(t_{\min}+t_{\max})\big)$
       \State $\overline{k}_b \gets t_{\mathrm{mid}}^{-\alpha}/\Gamma(1-\alpha)$
    \ElsIf{$|1-\alpha| > 10^{-8}$}
       \State $\overline{k}_b \gets \dfrac{t_{\max}^{1-\alpha}-t_{\min}^{1-\alpha}}{(1-\alpha)\,\delta\,\Gamma(1-\alpha)}$
    \Else
       \State $\overline{k}_b \gets \dfrac{\log\!\big(t_{\max}/t_{\min}\big)}{\delta\,\Gamma(1-\alpha)}$
    \EndIf

    \State $\widehat{\bm{g}}_n \gets \widehat{\bm{g}}_n + \overline{k}_b\,\bm{S}_b$
  \EndIf
\EndFor
\State \textbf{Return} $\widehat{\bm{g}}_n$
\end{algorithmic}
\end{algorithm}

Algorithm \ref{alg:DHDC-update} is the pseudocode of the DHDC introduced in section \ref{subsec:DHDC}

\end{document}